\documentclass[a4paper, english]{amsart}
\usepackage[utf8]{inputenc}

\usepackage{amsthm}
\usepackage{amsfonts}
\usepackage{mathtools}
\usepackage{amssymb}
\usepackage{paralist}
\usepackage[all,cmtip]{xy}
\usepackage{enumitem}
\usepackage{kantlipsum}

\usepackage[colorlinks=true, linkcolor=blue, citecolor=blue, linktocpage=true]{hyperref}
\usepackage[capitalise, nameinlink]{cleveref}
\crefname{subsection}{Subsection}{Subsections}
\crefname{equation}{}{}

\renewcommand\le\leqslant
\renewcommand\ge\geqslant

\calclayout

\usepackage[backend=bibtex, style=alphabetic]{biblatex} 
\AtBeginBibliography{\footnotesize}

\setlist[itemize]{align=parleft,left=0pt..12pt,topsep=0pt,label={\(-\)}}
\setlist[enumerate]{align=parleft,left=0pt..18pt,topsep=0pt}

\newcommand{\fg}[0]{\mathfrak{g}}

\newcommand{\bC}[0]{\mathbb{C}}
\newcommand{\bN}[0]{\mathbb{N}}
\newcommand{\bZ}[0]{\mathbb{Z}}
\newcommand{\bR}[0]{\mathbb{R}}
\newcommand{\cA}[0]{\mathcal{A}}
\newcommand{\cB}[0]{\mathcal{B}}
\newcommand{\cF}[0]{\mathcal{F}}
\newcommand{\cG}[0]{\mathcal{G}}
\newcommand{\cO}[0]{\mathcal{O}}

\newcommand{\tH}[0]{\textnormal{H}}
\newcommand{\compO}[0]{\widehat{\mathcal{O}}}
\newcommand{\compA}[0]{\widehat{\mathcal{A}}}
\DeclareMathOperator{\sheafHom}{\mathcal{H}\hspace{-1pt}\textit{om}}
\DeclareMathOperator{\sheafEnd}{\mathcal{E}\hspace{-1pt}\textit{nd}}

\DeclareMathOperator{\End}{End}
\DeclareMathOperator{\Hom}{Hom}
\DeclareMathOperator{\Aut}{Aut}

\usepackage{amsthm}

\numberwithin{equation}{section}
\renewcommand*{\theequation}{%
  \ifnum\value{section}>0 %
    \thesection.%
  \fi
    \arabic{equation}%
}

\makeatletter
\def\th@plain{%
  \thm@notefont{}
  \itshape 
}
\def\th@definition{%
  \thm@notefont{}
  \normalfont 
}
\makeatother

\newtheorem{theoremcounter}{theoremcounter}[section]
\newtheorem{maintheoremcounter}{maintheoremcounter}

\theoremstyle{plain}

\newtheorem{lemma}[theoremcounter]{Lemma}
\newtheorem{proposition}[theoremcounter]{Proposition}
\newtheorem{theorem}[theoremcounter]{Theorem}

\newtheorem{maintheorem}[maintheoremcounter]{Theorem}

\newtheoremstyle{example_style}
  {5pt} 
  {5pt} 
  {} 
  {} 
  {\itshape} 
  {.} 
  {10pt} 
  {} 

\theoremstyle{example_style}
\newtheorem{examplex}[theoremcounter]{Example}
\newenvironment{example}
  {\pushQED{\qed}\examplex}
  {\popQED\endexamplex}

\theoremstyle{definition}

\newtheorem{definition}[theoremcounter]{Definition}
\newtheorem{notation}[theoremcounter]{Notation}

\theoremstyle{example_style}
\newtheorem{remarkx}[theoremcounter]{Remark}
\newenvironment{remark}
  {\pushQED{\qed}\remarkx}
 {\popQED\endremarkx}

\title[Geometrization of generalized \(r\)-matrices]{Geometrization of solutions of the generalized classical Yang-Baxter equation and a new proof of the Belavin-Drinfeld trichotomy}
\author{Raschid Abedin}\address{
Universit\"at Paderborn\\
Institut f\"ur Mathematik \\
Warburger Stra\ss{}e 100 \\
33098 Paderborn \\
Germany
}
\email{rabedin@math.uni-paderborn.de}

\begin{document}

\maketitle

\begin{abstract}
    We study solutions to a generalized version of the classical Yang-Baxter equation (CYBE) with values in a central simple Lie algebra over a field of characteristic 0 from an algebro-geometric perspective.
    In particular, we describe such solutions using cohomology free sheaves of Lie algebras on projective curves. 
    This framework leads to a new proof of the Belavin-Drinfeld trichotomy, which asserts that over the ground field \(\bC\) any non-degenerate solution of the CYBE is either elliptic, trigonometric or rational. Furthermore, we obtain new structural results for solutions of the generalized CYBE over arbitrary fields of characteristic 0. For example, we prove that such solutions can be extended in an appropriate way to rational maps on the product of two curves.\\
\end{abstract}

\section*{Introduction}
Let \(\fg\) be a finite-dimensional simple complex Lie algebra. The classical Yang-Baxter equation (CYBE) with one spectral parameter is the functional equation
\begin{equation}\label{eq:intro_one_parameter_CYBE}
    [r^{12}(x),r^{13}(x+y)]+[r^{12}(x),r^{23}(y)] + [r^{13}(x+y),r^{23}(y)] = 0,
\end{equation}
where \(r\colon U \to \fg \otimes \fg\) is a meromorphic map for some connected open subset \(U\) of \(\bC\). The commutators on the left-hand side of \cref{eq:intro_one_parameter_CYBE} are thereby understood in the triple tensor product of the universal enveloping algebra \(\textnormal{U}(\fg)\) of \(\fg\) and e.g.\ \(r^{13}(z) = (r(z))^{13}\) where \((a\otimes b)^{13} = a \otimes 1\otimes b \in \textnormal{U}(\fg)^{\otimes 3}\) for all \(a,b\in \fg\). Solutions of \cref{eq:intro_one_parameter_CYBE} are called \(r\)-matrices and originate in the study of the quantum inverse scattering method as the quasi-classical limit of solutions of the (quantum) Yang-Baxter equation. They have found diverse  applications in mathematical physics, most importantly in the theory of classical integrable systems (see  e.g.\ \cite{faddeev_takhtajan,babelon_bernard_talon,adler_moerbeke_vanhaecke}) and quantum groups (see e.g.\ \cite{chari_pressley,etingof_schiffmann}). The simplest solution of \cref{eq:intro_one_parameter_CYBE} is Yang's \(r\)-matrix \(r_{\textnormal{Yang}}(x,y) = \frac{\gamma}{x-y}\), where \(\gamma\in \fg \otimes \fg\) is the Casimir element of \(\fg\).

Fundamental results in the structure theory of \(r\)-matrices were achieved by Belavin and Drinfeld in \cite{belavin_drinfeld_solutions_of_CYBE_paper} where they examined solutions \(r\) of \cref{eq:intro_one_parameter_CYBE} which are non-degenerate, i.e.\ \(r(z_0)\in\fg\otimes\fg\) induces an isomorphism \(\fg^*\to \fg\) for some \(z_0\) in the domain of definition of \(r\). They prove that these are automatically skew-symmetric, i.e.\ \(r(z) = -\tau(r(-z))\) for \(\tau(a\otimes b) = b\otimes a\), and extend uniquely to meromorphic functions \(\bC \to \fg\otimes \fg\) which are, up to a certain equivalence, either elliptic, trigonometric (i.e.\ rational functions of exponentials) or rational. A comprehensive coordinate-free rework of the proof of this trichotomy can be found in \cite{kierans_kreussler}. In \cite{belavin_drinfeld_solutions_of_CYBE_paper} the authors also classify all elliptic and trigonometric \(r\)-matrices and note that solving the classification problem of rational \(r\)-matrices is hopeless since it contains the problem of classifying commuting square matrices. In \cite{stolin_sln,stolin_maximal_orders} Stolin observed that rational solutions of \cref{eq:intro_one_parameter_CYBE} are determined by certain Lagrangian Lie algebras and developed the structure theory of these \(r\)-matrices based upon this observation.

In general, it is more natural to consider the CYBE in two parameters 
\begin{equation}\label{eq:intro_two_parameter_CYBE}
    [r^{12}(x_1,x_2),r^{13}(x_1,x_3)]+[r^{12}(x_1,x_2),r^{23}(x_2,x_3)] + [r^{13}(x_1,x_3),r^{23}(x_2,x_3)] = 0,
\end{equation}
which has meromorphic maps \(r\colon U \times U \to \fg \otimes \fg\) as variables. For \(r(x,y) = s(x-y)\), it is easy to see that this equation reduces to  \cref{eq:intro_one_parameter_CYBE}, so it is unambiguous to call solutions of \cref{eq:intro_two_parameter_CYBE} \(r\)-matrices as well. In \cite{belavin_drinfeld_diffrence_depending} it is shown that every non-degenerate solution of \cref{eq:intro_two_parameter_CYBE} comes from one of \cref{eq:intro_one_parameter_CYBE} in this way up to certain equivalence relations, but it is in general unclear why these relations preserve the type (i.e.\ ellitpic, trigonometric, rational) of this solution making the trichotomy unadjusted to this context at first glance. However, as a consequence every non-degenerate \(r\)-matrix \(r\) is seen to be skew-symmetric in the sense that \(r(x,y) =-\tau(r(y,x))\).  

From the point of view of many applications the skew-symmetry of non-degenerate \(r\)-matrices is unnecessary (see e.g.\ the works of Skrypnyk  \cite{skrypnyk_spin_chains,skrypnyk_infinite_dimensional_Lie_algebras} and references therein), namely one can consider solutions of the generalized CYBE (GCYBE) 
\begin{equation}\label{eq:intro_two_parameter_GYBE}
    [r^{12}(x_1,x_2),r^{13}(x_1,x_3)]+[r^{12}(x_1,x_2),r^{23}(x_2,x_3)] + [r^{32}(x_3,x_2),r^{13}(x_1,x_3)] = 0,
\end{equation}
which are called generalized \(r\)-matrices. The name stems from the fact that the non-degenerate \(r\)-matrices are precisely the skew-symmetric generalized ones, while e.g.\ \(yr_{\textnormal{Yang}
}(x,y)\) is a non-degenerate non-skew-symmetric generalized \(r\)-matrix. 

It can be shown that any non-degenerate generalized \(r\)-matrix \(r\colon U \times U \to \fg \otimes \fg\), probably after shrinking \(U\), is of the form 
\begin{equation}\label{eq:intro_standard_form}
    r(x,y) = \frac{\lambda(y)}{x-y}\gamma + r_0(x,y),
\end{equation}
where \(\gamma\) is the Casimir element of \(\fg\) and \(\lambda\colon U \to \fg ,r_0\colon U \times U \to \fg \otimes \fg\) are appropriate holomorphic maps. For skew-symmetric \(r\) this is mentioned in \cite{belavin_drinfeld_diffrence_depending}, for the general case see \cref{prop:complex_rmatrix_standard_form}. The standard form \cref{eq:intro_standard_form} can be used to introduce formal versions of non-degenerate generalized \(r\)-matrices when \(\bC\) is replaced by some field \(\Bbbk\) of characteristic 0. Namely, one can consider elements in \((\fg \otimes \fg)(\!(x)\!)[\![y]\!]\) of the form \cref{eq:intro_standard_form} solving a formal version of the GCYBE. These series will be called formal generalized \(r\)-matrices and are especially interesting for \(\Bbbk = \bR\) from the point of view of integrable systems (see e.g.\ \cite{babelon_bernard_talon}) and for \(\Bbbk = \bC(\!(\hbar)\!)\) from the point of view of quantum groups (see e.g.\ \cite{karolinsky_pianzola_stolin}). This paper is dedicated to the study of formal generalized \(r\)-matrices based on the algebro-geometric approach which was developed in \cite{burban_galinat} as a far reaching expansion of \cite{cherednik_definition_of_tau}. 

\subsection*{Content and results.} Let \(\fg\) be a semi-simple finite-dimensional Lie algebra over a field \(\Bbbk\) of characteristic 0.
In \cref{sec:formal_generalized_rmatrices} we introduce formal generalized \(r\)-matrices and define an appropriate notion of equivalence between them, which allows to treat equivalent formal generalized \(r\)-matrices interchangeably for our purposes. Moreover, we will discuss skew-symmetric formal generalized \(r\)-matrices, which are simply called formal \(r\)-matrices, and we show that a formal generalized \(r\)-matrix \(r\) is determined by a Lie algebra \(\fg(r) \subseteq \fg(\!(z)\!)\) complementary to \(\fg[\![z]\!]\). This construction is already known for generalized \(r\)-matrices; see e.g.\ \cite{cherednik_becklund_darboux,skrypnyk_infinite_dimensional_Lie_algebras}.

In \cref{sec:sheaves_of_algebras} we derive some results in the theory of sheaves of (not necessarily associative) algebras. \Cref{subsec:local_triviality_of_sheaves_of_Lie_algebras} is thereby devoted to the proof of algebro-geometric analogs of the results in \cite{kirangi_lie_algebra_bundles,kirangi_semi_simple} about local triviality of Lie algebra bundles, which may be of independent interest. More precisely, we prove that a sheaf of algebras on a reduced \(\Bbbk\)-scheme of finite type with certain assumptions on its fibers is locally trivial in the \'etale topology; see \cref{thm:etale_trivial_sheaves_of_Lie_algebras}. In addition, we give a sufficient condition in \cref{thm:etale_triviality_due_to_semisimple_fiber} for these assumptions to be at least locally satisfied. In \cref{subsec:geometry_of_lattices} the results of \cite{ostapenko} are combined with  projectification schemes which were already used in e.g.\ \cite{mumford,mulase} to see that certain subalgebras of the algebra of Laurent series with coefficients in a (not necessarily associative) central (\cref{def:central_algebra}), simple, finite-dimensional \(\Bbbk\)-algebra can be geometrized to sheaves of algebras on projective curves over \(\Bbbk\). 

Henceforth, assume that \(\fg\) is central and simple and fix a formal generalized \(r\)-matrix \(r\). In \cref{sec:geometrization} we apply the aforementioned geometrization procedure to \(\fg(r)\) and discuss the consequent geometric properties of \(r\). The results presented in \cref{subsec:local_triviality_of_sheaves_of_Lie_algebras} and \cref{subsec:geometry_of_lattices} state that the affine scheme \(\textnormal{Spec}(O)\), where
\[O \subseteq \textnormal{Mult}(\fg(r)) \coloneqq \{f \in \Bbbk(\!(z)\!)\mid f\fg(r) \subseteq \fg(r)\}\]
is an arbitrary unital subalgebra of finite codimension, admits a completion by a \(\Bbbk\)-rational smooth point \(p\) to an integral projective curve \(X\) over \(\Bbbk\). Furthermore, \(\fg(r)\) defines a coherent cohomology-free sheaf \(\cA\) of Lie algebras on \(X\) which is \'etale \(\fg\)-locally free (\cref{def:local_triviality}) at \(p\).
The most important properties of \(X\) are summarized in the following statement; see \cref{prop:geometry_of_generalized_rmatrices} \& \cref{thm:geometry_of_formal_rmatrices}.

\begin{maintheorem}\label{mainthmA}
Let \(\fg\) be a finite-dimensional central simple Lie algebra over a field \(\Bbbk\) of characteristic 0 and \(r \in (\fg \otimes \fg)(\!(x)\!)[\![y]\!]\) be a formal generalized \(r\)-matrix. Chose a unital subalgebra \(O \subseteq \textnormal{Mult}(\fg(r))\) of finite codimension and let \(X\) be the completion of \(\textnormal{Spec}(O)\) to an integral projective curve mentioned above (see \cref{prop:geometrization_of_klattices} for details). The following results are true.
\begin{enumerate}
    \item \(X\) has geometric genus at most one.
    \item If \(X\) has geometric genus one, \(X\) is elliptic and \(r\) is skew-symmetric up to equivalence.
    \item If \(r\) is skew-symmetric, \(O\) can be chosen in such a way that \(X\) has arithmetic genus one. \hfill \(\Box\)
\end{enumerate}
\end{maintheorem}

\noindent
An integral projective curve over \(\Bbbk\) of arithmetic genus one with a \(\Bbbk\)-rational smooth point is an irreducible cubic plane curve and hence it is either elliptic, rational with a single nodal singularity or rational with a single cuspidal singularity; see \cref{rem:curves_of_arithmetic_genus_one}. Therefore, \cref{mainthmA}.(3) yields a geometric splitting of formal \(r\)-matrices into three categories by the underlying curve, i.e.\ we obtain a geometric trichotomy. 

We can chose a smooth open neighbourhood \(C\) of \(p\) admitting a non-vanishing 1-form \(\eta\) and such that \(\cA|_C\) is \'etale \(\fg\)-locally free. In \cref{subsec:geometric_rmatrix} we will see that \(((X,\cA),(C,\eta))\) is an appropriate refinement of the geometric datum used in \cite{burban_galinat} to construct algebro-geometric solutions of the GCYBE, which is also admissible if \(\Bbbk\) is not algebraically closed. In particular, we can use the properties of \(((X,\cA),(C,\eta))\) to construct a special section \(\rho \in \Gamma(C\times C\setminus\Delta,\cA \boxtimes\cA)\), where \(\Delta\) is the diagonal of \(C\), solving a geometric version of the GCYBE. This section is called geometric \(r\)-matrix of \(((X,\cA),(C,\eta))\) and can be seen as an extension of \(r\) in the following sense; see \cref{thm:geometric_rmatrix_Taylor_series}.

\begin{maintheorem}\label{mainthmB}
Let \(\fg\) be a finite-dimensional central simple Lie algebra over a field \(\Bbbk\) of characteristic 0 and \(r \in (\fg \otimes \fg)(\!(x)\!)[\![y]\!]\) be a formal generalized \(r\)-matrix. If we assign a geometric datum \(((X,\cA),(C,\eta))\) to \(r\) as above and denote by \(\rho\) the associated geometric \(r\)-matrix, the following results are true.
\begin{enumerate}
    \item There is a natural injection \(\Gamma(C\times C\setminus \Delta,\cA \boxtimes \cA) \to (\fg \otimes \fg)(\!(x)\!)[\![y]\!]\), which can be thought of as a Taylor expansion, mapping \(\rho\) to \(\lambda(y)r(x,y)\) for some \(\lambda \in \Bbbk[\![z]\!]^\times\).
    
    \item If \(r\) is skew-symmetric, we may assume \(\lambda = 1\).\hfill \(\Box\)
\end{enumerate}
\end{maintheorem}

\noindent
If \(r\) is skew-symmetric, \(\cA\) is \'etale \(\fg\)-locally free on the smooth locus of \(X\) and admits a global geometric analog of the Killing form (see \cref{thm:geometry_of_formal_rmatrices}), while \(\rho\) turns out to be skew-symmetric in an appropriate sense. We obtain a refinement of the geometric axiomatization of skew-symmetry presented in \cite[Theorem 4.3]{burban_galinat}, which does not assume \(\Bbbk\) to be algebraically closed. 

The \'etale \(\fg\)-local triviality of \(\cA\) in \(p\) combined with \cref{mainthmB} yields the following extension property of generalized \(r\)-matrices; see \cref{subsec:rational_extension_of_formal_generalized_rmatrices}.

\begin{maintheorem}\label{mainthmC}
Let \(\fg\) be a finite-dimensional central simple Lie algebra over a field \(\Bbbk\) of characteristic 0 and \(r \in (\fg \otimes \fg)(\!(x)\!)[\![y]\!]\) be a formal generalized \(r\)-matrix.
There exists a smooth integral projective curve \(Y\) over \(\Bbbk\) such that, up to equivalence, \(r\) extends to a rational section of \((\fg \otimes \fg)\otimes \cO_{Y\times Y}\). In particular, if \(\Bbbk\) is algebraically closed, \(r\) is equivalent to a certain Taylor expansion of a rational function \(Y(\Bbbk) \times Y(\Bbbk) \to \fg \otimes \fg\), where \(Y(\Bbbk)\) is the algebraic set of closed points of \(Y\).\hfill \(\Box\)
\end{maintheorem}
\noindent
In \cref{sec:real_and_complex_generalized_rmatrices} we consider the consequences of the established results for non-degenerate solutions of the GCYBE \cref{eq:intro_two_parameter_GYBE} as well as a similar real analytic notion. In \cref{prop:analytic_and_formal_theories_are_equivalent} we show that the theory of non-degenerate real and complex analytic generalized \(r\)-matrices up to an analytic notion of equivalence coincides with the formal theory over \(\bR\) and \(\bC\) respectively. As already mentioned above, \(\cA\) is \'etale \(\fg\)-locally free on the smooth locus of \(X\), if \(r\) is skew-symmetric. This smooth locus is explicitly known due to \cref{mainthmA}.
Using the results on torsors over smooth affine curves in \cite{pianzola} and studying acyclic sheaves of Lie algebras over complex elliptic curves via factors of automorphy combined with \cref{mainthmB} results in the following version of the Belavin-Drinfeld trichotomy; see \cref{thm:Belavin-Drinfeld_trichotomy}.
\begin{maintheorem}\label{mainthmD}
Let \(\fg\) be a finite-dimensional simple complex Lie algebra and \(r \in (\fg \otimes \fg)(\!(x)\!)[\![y]\!]\) be a formal \(r\)-matrix. Then \(r\) is equivalent to the Taylor series in the second variable in \(0\) of a non-degenerate \(r\)-matrix \(\bC \times \bC \to \fg \otimes \fg\) which is
\begin{itemize}
    \item elliptic in both variables if and only if \(X\) is elliptic,
    \item a rational function of exponentials if and only if \(X\) is nodal and
    \item rational if and only if \(X\) is cuspidal. \hfill \(\Box\)
\end{itemize}
\end{maintheorem}
\noindent
This theorem gives a new implicit proof of the geometrization results for non-degenerate complex analytic \(r\)-matrices, which were carried out, using the Belavin-Drinfeld trichotomy, for elliptic solutions in \cite{burban_heinrich}, for rational solutions in \cite{burban_galinat} and recently for trigonometric solutions independently in \cite{polishchuck_trigonometric_geometrisation} and \cite{abedin_burban}. \Cref{mainthmD} also formalizes the geometric intuition of the Belavin-Drinfeld trichotomy mentioned in \cite[Section 4]{drinfeld_quantum_groups}.

\subsection*{Acknowledgments} 
This work was supported  by the DFG project Bu--1866/5--1. 
I am grateful to Igor Burban for introducing me to the topic and providing useful discussions as well as helpful remarks and advices. Furthermore, I would like to thank Stepan Maximov for taking the time to proofread parts of this document. 

\subsection*{Notation and conventions}
Throughout this document \(\Bbbk\) denotes a field of characteristic \(0\) with algebraic closure \(\overline{\Bbbk}\). By convention the set of natural numbers \(\bN\) excludes \(0\) and  \(\bN_0 = \bN \cup \{0\}\). 

\subsubsection*{Algebra.} For a unital commutative ring \(R\) and \(R\)-modules \(M,N\), the space of \(R\)-linear maps \(M \to N\) (resp. \(M \to M\)) is denoted by \(\Hom_R(M,N)\) (resp. \(\End_R(M)\)), while the tensor product of \(M\) and \(N\) is written as \(M \otimes_R N\). For \(R = \Bbbk\) the indices are omitted. The invertible elements of \(R\) are denoted by \(R^\times\), and \(M^* \coloneqq \Hom_R(M,R)\) is the dual module of \(M\).  If \(R\) is a domain, \(\textnormal{Q}(R) \coloneqq (R\setminus \{0\})^{-1}R\) denotes its quotient field and we write \(\textnormal{Q}(M) \coloneqq M \otimes_R \textnormal{Q}(R)\). Let \(f\colon R \to R'\) be a morphism of unital commutative rings and \(M'\) be an \(R'\)-module. We say that a map \(g \colon M\to M'\) is \(f\)-equivariant if it is a group homomorphism satisfying \(g(r m) = f(r) g(m)\) for all \(r\in R,m\in M\).

In this text, an \(R\)-algebra \(A\) satisfies no additional assumptions, i.e.\ \(A = (A,\mu_A)\) consists of an \(R\)-module \(A\) equipped with a multiplication map \(\mu_A \colon A \otimes_R A \to A\). In particular, a Lie algebra over \(R\) is an \(R\)-algebra. The group of invertible \(R\)-algebra endomorphisms of \(A\), i.e.\ invertible \(R\)-linear maps \(f \colon A \to A\) satisfying \(f\mu_A = \mu_A(f\otimes f)\), will be denoted by \(\Aut_{R\textnormal{-alg}}(A)\). We note that \(\oplus\) will always denote the direct sum of modules and not of algebras. If \(A\) is a Lie algebra with adjoint representation \(\textnormal{ad}\colon A \to \textnormal{End}_R(A)\), we write \([a \otimes 1, t] \coloneqq (\textnormal{ad}(a) \otimes \textnormal{id}_A)t\) for all \(a \in A,t \in A \otimes_R A\) and define \([1 \otimes a, t]\) and \([a \otimes 1 + 1 \otimes a,t]\) in a similar fashion.

\subsubsection*{Algebraic geometry.} Let \(X = (X,\cO_X)\) be a ringed space and \(\cF,\cG\) be two \(\cO_X\)-modules. For a morphism \(f\colon X \to Y = (Y,\cO_Y)\) of ringed spaces, we denote the additional structure morphism by \(f^\flat\colon \cO_Y \to f_*\cO_X\) and write \(f^\sharp \colon f^{-1}\cO_Y \to \cO_X\) for the induced morphism. The set of \(\cO_X\)-module homomorphisms \(\cF \to \cG\) (resp. \(\cF \to \cF\)) is denoted by \(\Hom_{\cO_X}(\cF,\cG)\) (resp. \(\End_{\cO_X}(\cF)\)) while its sheaf counterpart is denoted by \(\sheafHom_{\cO_X}(\cF,\cG)\) (resp. \(\sheafEnd_{\cO_X}(\cF))\). In particular, we write \(\cF^* \coloneqq \sheafHom_{\cO_X}(\cF,\cO_X)\). The tensor product of \(\cF\) and \(\cG\) is written as \(\cF\otimes_{\cO_X}\cG\). By abuse of notation, we write \(f^*= f^*_\cF\colon \cF \to f_*f^*\cF\) for the canonical morphism. 

Assume that \(X\) and \(Y\) are \(S\)-schemes. The fiber product of \(X\) and \(Y\) over \(S\) is denoted by \(X\times_S Y\) and \(\cF|_p\) is the fiber of \(\cF\) in a point \(p\in X\). If \(S = \textnormal{Spec}(\Bbbk)\), the index \(S\) is omitted and \(\textnormal{H}^n(\cF)\) denotes the \(n\)-th global cohomology group of \(\cF\), while \(\textnormal{h}^n(\cF)\) denotes its dimension over \(\Bbbk\), if said space is finite-dimensional. In particular, \(\textnormal{H}^0(\cF) = \Gamma(X,\cF)\) is the space of global sections of \(\cF\).

\subsubsection*{Formal series.} For a module \(M\) over a unital commutative ring \(R\), the module of formal power series in the formal variable \(z\) with coefficients in \(M\) is denoted by \(M[\![z]\!]\). The module \(R[\![z]\!]\) 
is again a unital commutative ring and \(M[\![z]\!]\) is a \(R[\![z]\!]\)-module. Then \(M(\!(z)\!) \coloneqq M[\![z]\!][z^{-1}]\) is the module of formal Laurent series. We note that if \(M\) is an \(R\)-algebra the module \(M[\![z]\!]\) (resp. \(M(\!(z)\!)\)) is naturally an \(R[\![z]\!]\)-algebra  (resp. \(R(\!(z)\!)\)-algebra). Elements \(p\) in \(M(\!(z)\!)\) (resp. \(M(\!(x_1)\!)\dots(\!(x_k)\!)\)) will sometimes be denoted with the formal variable (resp. variables) for convenience, i.e.\ \(p = p(z)\) (resp. \(p = p(x_1,\dots,x_k)\)). Moreover, a generic element \(p \in M(\!(z)\!)\) is written \(p(z) = \sum_{k \gg -\infty}p_kz^k\) and \(p'(z) = \sum_{k \gg -\infty} kp_kz^{k-1}\) denotes the formal derivative of \(p\). Finally, if \(p(z) \in mz^{-k} + z^{-k+1}M[\![z]\!]\), it is said to be of order \(k\) with main part \(mz^{-k}\).

\section{Survey on formal generalized \(r\)-matrices}\label{sec:formal_generalized_rmatrices}

\subsection{Formal generalized \(r\)-matrices}\label{subsec:formal_rmatrices}
Let \(\fg\) be a semi-simple Lie algebra of dimension \(d \in \bN\) over \(\Bbbk\) and \(\{b_i\}_{i = 1}^d \subset \fg\) be an orthonormal basis with respect to the Killing form \(\kappa\) of \(\fg\). The Casimir element of \(\fg\) is given by \(\gamma = \sum_{i = 1}^d b_i \otimes b_i \in \fg \otimes \fg\).

Consider \((x-y)^{-1}\) as the series \(\sum_{k = 0}^\infty x^{-k-1}y^k\in \Bbbk(\!(x)\!)[\![y]\!]\). Then
\begin{align}\label{eq:Yangs_rmatrix}
    r_{\textnormal{Yang}}(x,y)\coloneqq \frac{\gamma}{x-y} = \sum_{k = 0}^\infty\sum_{i = 1}^d x^{-k-1}b_i \otimes y^kb_i,
\end{align}
is an element of 
\((\fg\otimes \fg)(\!(x)\!)[\![y]\!] \cong (\fg(\!(x)\!)\otimes \fg)[\![y]\!] \cong (\fg\otimes \fg)\otimes \Bbbk(\!(x)\!)[\![y]\!]\). It is the formal version of Yang's \(r\)-matrix mentioned in the introduction.

\begin{definition}\label{def:formal_standard_from}
A series \(r \in (\fg\otimes \fg)(\!(x)\!)[\![y]\!]\) is said to be in \emph{standard form} if 
\begin{equation}\label{eq:standard_form}
r(x,y) = \frac{\lambda(y)}{x-y}\gamma +r_0(x,y) = \lambda(y)r_{\textnormal{Yang}}(x,y) + r_0(x,y)
\end{equation}
for some \(\lambda \in \Bbbk[\![z]\!]^\times\) and \(r_0\in (\fg\otimes \fg)[\![x,y]\!]\). In this case 
\begin{itemize}
    \item \(r\) is called \emph{normalized} if \(\lambda = 1\),
    \item \(\overline{r}(x,y)  \coloneqq \lambda(x)r_{\textnormal{Yang}}(x,y) -\tau(r_0(y,x)) \in (\fg \otimes \fg)(\!(x)\!)[\![y]\!]\), where \(\tau\) is the \(\Bbbk[\![x,y]\!]\)-linear extension of the linear automorphism of \(\fg \otimes \fg\) defined by \(a\otimes b \mapsto b\otimes a\), and
    \item \(r\) is called \emph{skew-symmetric} if \(\overline{r} = r\).
\end{itemize}
\end{definition}
\begin{remark}\label{rem:formal_standard_forms}
For a general \(r = r(x,y) \in (\fg\otimes \fg)(\!(x)\!)[\![y]\!]\) there is no appropriate \(\overline{r}\) since switching \(x\) and \(y\) does not define a map of \(\Bbbk(\!(x)\!)[\![y]\!]\) into itself. Therefore, a well-defined notion of skew-symmetry in this context needs additional assumptions, e.g.\ that \(r\) is in standard form.
\end{remark}

\begin{lemma}\label{lem:series_vanishing_at_diagonal}
Let \(V\) be a \(\Bbbk\)-vector space and \(f(x,y) \in V[\![x,y]\!]\) satisfy \(f(z,z) = 0\) in \(V[\![z]\!]\). Then \(f\) is divisible by \(x-y\), i.e.\ \(f(x,y) = (x-y)g(x,y)\) for some \(g \in V[\![x,y]\!]\). Furthermore, if \(f(x,y) = h(x) - h(y)\) for some \(h \in \Bbbk[\![z]\!]\), the identity \(g(z,z) = h'(z)\) holds.
\end{lemma}
\begin{proof}
Let \(V[x,y]_\ell\) denote the homogeneous elements of total degree \(\ell\) in \(\Bbbk[x,y]\) for all \(\ell \in \bN_0\). 
Then \(f = \sum_{\ell = 0}^\infty f_\ell\), where \(f_\ell \in V[x,y]_\ell\). Since \(x-y\) is homogeneous of total degree one, it suffices to prove the claim for \(f = f_\ell \in V[x,y]_\ell\). The polynomial \(x-y \in \Bbbk[x,y] = \Bbbk[y][x]\) is monic, so the polynomial division algorithm provides \(g\in V[x,y]\) and \(r\in V[y]\) such that
\begin{equation}
    f(x,y) = (x-y)g(x,y) + r(y).
\end{equation}
Therefore, \(0 = f(z,z) = r(z)\) proves \(f(x,y) = (x-y)g(x,y)\). Note that in the special case \(f(x,y) = h(x)-h(y)\) for \(h(z) = az^\ell\), a direct calculation verifies that
\begin{equation}
    g(x,y) = a\sum_{k = 0}^{\ell-1}x^k y^{\ell-1-k}.
\end{equation}
Hence, \(g(z,z) = \ell az^{\ell-1} = h'(z)\) in this case, proving the second part of the statement.
\end{proof}

\begin{remark}\label{rem:formal_standard_forms2}
Let \(r=r(x,y) = f(x,y)r_\textnormal{Yang}(x,y) + r_0(x,y) \in (\fg \otimes\fg)(\!(x)\!)[\![y]\!]\) for some \(f \in \Bbbk[\![x,y]\!]^\times\) and \(r_0 \in (\fg \otimes\fg)[\![x,y]\!]\). Then 
\begin{equation}
    r(x,y) = f(y,y)r_\textnormal{Yang}(x,y)+  \frac{f(x,y)-f(y,y)}{x-y}\gamma +r_0(x,y)
\end{equation}
combined with \cref{lem:series_vanishing_at_diagonal} shows that \(r\) is in standard form. In particular, \(\overline{r}\) is also in standard form.
\end{remark}

\begin{notation}\label{not:ij_notations}
Let \(ij \in \{12,13,23\}\), \(\textnormal{U}(\fg)\) be the universal enveloping algebra of \(\fg\) and \(R \coloneqq \Bbbk(\!(x_1)\!)(\!(x_2)\!)[\![x_3]\!]\). For \(t = t_1 \otimes t_2 \in \fg \otimes \fg\) the assignments 
\[t^{12} = t \otimes 1,t^{13} = t_1 \otimes 1 \otimes t_2\textnormal{ and }t^{23} = 1 \otimes t\]
define linear maps \((\cdot)^{ij}\colon \fg \otimes \fg \to \textnormal{U}(\fg) \otimes \textnormal{U}(\fg) \otimes \textnormal{U}(\fg)\)
and the image of \(s \in (\fg\otimes \fg)(\!(x)\!)[\![y]\!]\) under
\begin{align*}
\xymatrix{(\fg\otimes \fg)(\!(x)\!)[\![y]\!] \cong  (\fg\otimes \fg)(\!(x_1)\!)[\![x_3]\!] \subseteq (\fg\otimes \fg)\otimes R  \ar[rr]^-{(\cdot)^{ij}\otimes \textnormal{id}_R} && (\textnormal{U}(\fg) \otimes \textnormal{U}(\fg) \otimes \textnormal{U}(\fg))\otimes R}
\end{align*}
is denoted by \(s^{ij}\).
\end{notation}

\begin{definition}\label{def:formal_generalized_rmatrices}
A series \(r \in (\fg\otimes \fg)(\!(x)\!)[\![y]\!]\) is called a \emph{(normalized) formal generalized \(r\)-matrix}, if it is in (normalized) standard form and solves the \emph{formal generalized classical Yang-Baxter equation} (formal GCYBE)
\begin{equation}\label{eq:formal_GCYBE}
    \textnormal{GCYB}(r) = 0 \textnormal{, where }\textnormal{GCYB}(r) \coloneqq [r^{12},r^{13}] + [r^{12},r^{23}] + [r^{13},\overline{r}^{23}].
\end{equation}
Here we used \cref{not:ij_notations} and note that the brackets in \(\textnormal{GCYB}(r)\) are the usual commutators in the associative \(R\)-algebra \((\textnormal{U}(\fg) \otimes \textnormal{U}(\fg) \otimes \textnormal{U}(\fg))\otimes R\).
\end{definition}

\begin{example}\label{ex:Yangs_rmatrix}
\(r_{\textnormal{Yang}}\) is a skew-symmetric formal generalized \(r\)-matrix since
\begin{equation}
    \begin{split}
        &(x_1-x_2)(x_1-x_3)(x_2-x_3)\textnormal{GCYB}(r_{\textnormal{Yang}})(x_1,x_2,x_3) \\&= ((x_2-x_3)-(x_1-x_3)+(x_1-x_2))[\gamma^{12},\gamma^{13}] = 0.
    \end{split}
\end{equation}
Here we used the fact that \([a\otimes 1 + 1 \otimes a,\gamma] = 0\) for all \(a \in \fg\) implies that
\([\gamma^{12},\gamma^{13}] = - [\gamma^{12},\gamma^{23}] = [\gamma^{13},\gamma^{23}]\) holds. It is easy to see that for all \(\lambda \in \Bbbk[\![z]\!]\) and \(\widetilde{r}(x,y) \coloneqq \lambda(y)r_\textnormal{Yang}(x,y)\) we have:
\begin{equation}
    \textnormal{GCYB}(\widetilde{r})(x_1,x_2,x_3) = \lambda(x_2)\lambda(x_3)\textnormal{GCYB}(r_{\textnormal{Yang}})(x_1,x_2,x_3) = 0.
\end{equation}
Therefore, \(\widetilde{r}\) is a generalized \(r\)-matrix and this series is not skew-symmetric if \(\lambda \notin \Bbbk^\times\).
\end{example}

\begin{remark}\label{rem:formal_GCYBE_domain}
Defining \(r^{32}(x_3,x_2) \coloneqq -\overline{r}^{23}(x_2,x_3)\) in \cref{eq:formal_GCYBE} results in the analog form of the GCYBE used in the introduction. Observe that e.g.\
\begin{align*}
    [(s_1\otimes s_2)^{13},(t_1\otimes t_2)^{23}] = s_1 \otimes t_1 \otimes s_2t_2 - s_1 \otimes t_1 \otimes t_2s_2 = s_1 \otimes t_1 \otimes [s_2,t_2] \in \fg\otimes \fg \otimes \fg
\end{align*}
for all \(s_1,s_2,t_1,t_2 \in \fg\). This and similar calculations show that 
\(\textnormal{GCYB}(r) \in (\fg \otimes \fg \otimes \fg)(\!(x_1)\!)(\!(x_2)\!)[\![x_3]\!]\) for every \(r \in (\fg \otimes \fg)(\!(x)\!)[\![y]\!]\)  in standard form.  This can be further refined. Since \([a \otimes 1 + 1 \otimes a,\gamma] = 0\) for all \(a \in \fg\), the same identity also holds for all \(a \in \fg[\![z]\!]\). We can use this to derive e.g.\ that
\begin{equation}
    [r_0^{12},\gamma^{13}] - [\gamma^{13},\tau(r_0)^{23}] = [r_0^{12}+ \tau(r_0)^{23},\gamma^{13}] \in (\fg \otimes \fg \otimes \fg)[\![x_1,x_2,x_3]\!]
\end{equation}
vanishes for \(x_1 = x_3\), where \(r_0\) is determined by \(r(x,y) = \lambda(y)r_\textnormal{Yang}(x,y) + r_0(x,y)\).
Combining this and similar identities with \cref{lem:series_vanishing_at_diagonal} and the fact that \(\lambda(y)r_\textnormal{Yang}(x,y)\) is a generalized \(r\)-matrix (see \cref{ex:Yangs_rmatrix})  implies that \(\textnormal{GCYB}(r) \in (\fg \otimes \fg \otimes \fg)[\![x_1,x_2,x_3]\!]\) for all \(r \in (\fg \otimes \fg)(\!(x)\!)[\![y]\!]\) in standard form.
\end{remark}

\begin{remark}\label{rem:field_extensions}
Let \(\Bbbk'\) be an arbitrary field extension of \(\Bbbk\) and \(r \in (\fg \otimes \fg)(\!(x)\!)[\![y]\!]\) be a formal generalized \(r\)-matrix. Then \(\fg_{\Bbbk'} \coloneqq \fg \otimes \Bbbk'\) is a semi-simple Lie algebra over \(\Bbbk'\) and the image \(r_{\Bbbk'}\) of \(r\) under the canonical map \[(\fg \otimes \fg)(\!(x)\!)[\![y]\!] \to (\fg_{\Bbbk'}\otimes_{\Bbbk'}\fg_{\Bbbk'})(\!(x)\!)[\![y]\!]\]
is again a formal generalized \(r\)-matrix.
\end{remark}

\begin{notation}
For
\(\varphi \in \textnormal{Aut}_{\Bbbk[\![z]\!]\textnormal{-alg}}(\fg[\![z]\!])\) the image of \(\varphi\) under
\[\textnormal{Aut}_{\Bbbk[\![z]\!]\textnormal{-alg}}(\fg[\![z]\!])\subseteq \textnormal{End}_{\Bbbk[\![z]\!]}(\fg[\![z]\!]) \cong \textnormal{End}(\fg)[\![z]\!]\]
is again denoted by \(\varphi = \varphi(z)\). Furthermore, \[\varphi\otimes \varphi \in \textnormal{End}(\fg)[\![z]\!]\otimes_{\Bbbk[\![z]\!]}\textnormal{End}(\fg)[\![z]\!] \cong (\End(\fg)\otimes \End(\fg))[\![z]\!]\] is also denoted by \(\varphi(z) \otimes \varphi(z)\) while \(\varphi(x) \otimes \varphi(y)\) denotes the image of \(\varphi \otimes \varphi \in \textnormal{End}(\fg)[\![z]\!] \otimes \textnormal{End}(\fg)[\![z]\!]\) under the canonical injection
\(\textnormal{End}(\fg)[\![z]\!] \otimes \textnormal{End}(\fg)[\![z]\!] \to (\textnormal{End}(\fg)\otimes\textnormal{End}(\fg))[\![x,y]\!].\)
\end{notation}

\begin{definition}
A series \(\widetilde{r} \in (\fg \otimes\fg)(\!(x)\!)[\![y]\!]\) is called \emph{equivalent} to \(r \in  (\fg \otimes\fg)(\!(x)\!)[\![y]\!]\) if \begin{equation}
    \widetilde{r}(x,y) = \mu(y)(\varphi(x) \otimes \varphi(y))r(w(x),w(y)),
\end{equation}
where the triple \((\mu,w,\varphi)\) is called an \emph{equivalence} and consists of a series \(\mu \in \Bbbk[\![z]\!]^\times\) called \emph{rescaling}, a formal parameter \(w \in z\Bbbk[\![z]\!]^\times\) called \emph{coordinate transformation} and a map \(\varphi \in \textnormal{Aut}_{\Bbbk[\![z]\!]\textnormal{-alg}}(\fg[\![z]\!])\) called \emph{gauge transformation}. Furthermore, \(\widetilde{r}\) is called \emph{gauge} (resp. \emph{coordinate}) \emph{equivalent} to \(r\) if \(\mu = 1\) and \(w = z\) (resp. \(\varphi = \textnormal{id}_{\fg[\![z]\!]}\)).
\end{definition}

\begin{lemma}\label{lemm:formal_equivalence}
The following results are true.
\begin{enumerate}
    \item Equivalences preserve the property of being a formal generalized \(r\)-matrix.
    \item Equivalences with constant rescaling part preserve skew-symmetry.
    \item Every series in standard form is coordinate equivalent to one in normalized standard form.
\end{enumerate}
\end{lemma}
\begin{proof}
Let \(r(x,y) = \lambda(y)r_\textnormal{Yang}(x,y) + r_0(x,y)\) for \(\mu \in \Bbbk[\![z]\!]^\times,r_0\in (\fg \otimes \fg)[\![x,y]\!]\) and \(\widetilde{r}\in (\fg \otimes \fg)(\!(x)\!)[\![y]\!]\) be equivalent to \(r\) via an equivalence \((\mu,w,\varphi)\).
Using \cref{lem:series_vanishing_at_diagonal} and \((\varphi(z)\otimes \varphi(z))\gamma = \gamma\) (see e.g.\ \cref{rem:Laurent_Killing_form} below), we can deduce that 
\begin{equation}\label{eq:rYang_coordinate_transform}
  (\varphi(x)\otimes \varphi(y))r_\textnormal{Yang}(w(x),w(y)) - w'(y)^{-1}r_\textnormal{Yang}(x,y) \in (\fg \otimes \fg)[\![x,y]\!],  
\end{equation}
hence \(\widetilde{r}\) is in standard form. 
It is easy to see that
\[\textnormal{GCYB}(\widetilde{r})(x_1,x_2,x_3) = \mu(x_2)\mu(x_3)(\varphi(x_1)\otimes \varphi(x_2)\otimes \varphi(x_3))\textnormal{GCYB}(r)(w(x_1),w(x_2),w(x_3)) = 0,\] 
since \(\textnormal{GCYB}(r) = 0\). This proves \emph{(1)}. The proof of \emph{(2)} is clear.

Let us put \(\mu = 1\) and \(\varphi = \textnormal{id}_{\fg[\![z ]\!]}\). There exists a unique \(u \in z\Bbbk[\![z]\!]^\times\) such that \(u'(z) = \lambda(u(z))\) and setting \(w = u\) proves \emph{(3)} under consideration of \cref{eq:rYang_coordinate_transform}. Indeed, if we write \(u(z) = \sum_{k = 1}^\infty u_kz^k\) and \(\lambda(z) = \sum_{k = 0}^\infty \lambda_kz^k\), it holds that \(\lambda(u(z)) = \sum_{k = 0}^\infty c_k z^k\), where \(c_0 = \lambda_0\) and
\begin{equation}
    c_k = \sum_{\ell = 1}^k \sum_{\substack{
     (j_1,\dots,j_\ell)\in\bN^\ell\\j_1 + \dots + j_\ell = k}}\lambda_\ell u_{j_1}\dots u_{j_\ell}.
\end{equation}
for all \(k \in \bN\) (note that we exclude \(0\) from \(\bN\)), hence the coefficients of \(u\) can be determined inductively by \(u_{k + 1} = \frac{c_k}{k+1}\) for all \(k \in \bN_0\).
\end{proof}

\begin{remark}
The study of formal generalized \(r\)-matrices will be pursued up to equivalence, i.e.\ equivalent \(r\)-matrices will be treated as interchangeable, where in light of \cref{lemm:formal_equivalence}.\emph{(2)} only equivalences with constant rescaling part will be used if skew-symmetry is relevant to the context. In particular, \cref{lemm:formal_equivalence}.\emph{(3)} would permit us to restrict our attention to normalized generalized \(r\)-matrices. Nevertheless, this will be done only if necessary, since rescalings appear naturally in the algebro-geometric context; see e.g.\ \cref{thm:geometric_rmatrix_Taylor_series} below.
\end{remark}

\subsection{Lie subalgebras of \(\fg(\!(z)\!)\) complementary to \(\fg[\![z]\!]\)}\label{subsec:lie_subalgebras_of_rmatrices} It was noted in \cite{cherednik_becklund_darboux} (see also \cite[Proposition 6.2]{etingof_schiffmann}) that it is possible to assign a Lie subalgebra of \(\fg(\!(z)\!)\) complementary to \(\fg[\![z]\!]\) to each non-degenerated solution of \cref{eq:intro_one_parameter_CYBE}. In the following, we will discuss the formal analog of this construction.
For a series 
\[s = s(x,y) = \sum_{k = 0}^\infty\sum_{i = 1}^d s_{k,i}(x) \otimes b_iy^k \in (\fg\otimes \fg)(\!(x)\!)[\![y]\!]\]
it is always possible to define the vector subspace 
\begin{align}
    \fg(s) \coloneqq \textnormal{Span}_{\Bbbk}\{s_{k,i}(z) \mid k \in \bN_0, i \in \{1,\dots,d\}\}
\end{align}
of \(\fg(\!(z)\!)\). It is the smallest subspace of \(\fg(\!(z)\!)\) satisfying \(s \in (\fg(s) \otimes \fg)[\![y]\!]\) and does not depend on the choice of basis \(\{b_i\}_{i = 1}^d\).

\begin{proposition}\label{prop:Lie_subalgebra_of_generalized_rmatrix} The assignment
\(r \mapsto \fg(r)\) gives a bijection between normalized formal generalized \(r\)-matrices and Lie subalgebras \(W \subseteq \fg(\!(z)\!)\) satisfying \(\fg(\!(z)\!) = \fg[\![z]\!] \oplus W\). 
\end{proposition}
\begin{proof}
Let \(r(x,y) = \sum_{k = 0}^\infty \sum_{i = 1}^dr_{k,i}(x)\otimes y^kb_i\)
be a normalized formal generalized \(r\)-matrix. The identity \(\fg(\!(z)\!) = \fg[\![z]\!] \oplus \fg(r)\) is a direct consequence of \(r_{k,i} - z^{-k-1}b_i \in \fg[\![z]\!]\) for all \(k \in \bN_0,i \in \{1,\dots,d\}\). It remains to show that \(\fg(r)\) is a subalgebra of \(\fg(\!(z)\!)\). The fact that
\([z^ka\otimes 1 + 1 \otimes z^ka,\gamma] = 0\) holds for all \(a \in \fg,k\in \bN_0\)
forces 
\begin{equation}
    [r^{12}+r^{13},\gamma^{23}](x_1,x_2,x_3) = \sum_{k = 0}^\infty \sum_{i = 1}^d r_{k,i}(x_1) \otimes [x_2^kb_i \otimes 1 + 1 \otimes x_3^kb_i,\gamma] \in (\fg(r) \otimes \fg \otimes \fg)[\![x_2,x_3]\!]
\end{equation}
to vanish for \(x_2 = x_3\). Therefore, 
\cref{lem:series_vanishing_at_diagonal}  implies that 
\begin{equation}\label{eq:second_half_of_GCYB}
    [r^{12},r^{23}] + [r^{13},\overline{r}^{23}]\in (\fg(r) \otimes \fg \otimes \fg)[\![x_2,x_3]\!].
\end{equation}
This combined with \(0 = \textnormal{GCYB}(r) = [r^{12},r^{13}] + [r^{12},r^{23}] + [r^{13},\overline{r}^{23}]\) shows that
\begin{align}\label{eq:first_half_of_GCYB}
    &\sum_{k,\ell = 0}^\infty \sum_{i,j=1}^d [r_{k,i}(x_1),r_{\ell,j}(x_1)] \otimes x_2^k b_i \otimes x_3^\ell b_j = [r^{12},r^{13}](x_1,x_2,x_3) 
\end{align}
is an element of \((\fg(r) \otimes \fg \otimes \fg)[\![x_2,x_3]\!]\).
We conclude that \([r_{k,i},r_{\ell,j}] \in \fg(r)\) for all \(k,\ell \in \bN_0, i,j \in \{1,\dots,d\}\). In particular, \(\fg(r)\) is a subalgebra of \(\fg(\!(z)\!)\).

Let us consider now a Lie subalgebra  \(W \subset \fg(\!(z)\!)\) satisfying \(\fg(\!(z)\!) = \fg[\![z]\!] \oplus W\). For every \(k \in \bN_0, i \in \{1,\dots,d\}\) there is an unique element \(r^W_{k,i} \in W\) such that \(r^W_{k,i}-b_iz^{-k-1} \in \fg[\![z]\!]\). By construction,
\begin{equation}
    r^W=r^W(x,y) \coloneqq \sum_{k = 0}^\infty \sum_{i=1}^d r^W_{k,i}(x) \otimes b_iy^k \in (\fg \otimes \fg)(\!(x)\!)[\![y]\!]
\end{equation}
is in normalized standard form and satisfies \(\fg(r^W) = W\). Furthermore, we can see that \(r^{\fg(r)} = r\). Thus, it remains to show that \(\textnormal{GCYB}(r^W) = 0\). In \cref{rem:formal_GCYBE_domain} it was noted that
\(\textnormal{GCYB}(r^W) \in (\fg \otimes \fg \otimes \fg)[\![x_1,x_2,x_3]\!]\).
Since \(\fg(r^W) = W\) is closed under the Lie bracket, \cref{eq:second_half_of_GCYB} and \cref{eq:first_half_of_GCYB} show that \(\textnormal{GCYB}(r^W) \in (W \otimes \fg \otimes \fg)[\![x_2,x_3]\!]\).
Summarized, we obtain
\begin{equation}
    \textnormal{GCYB}(r^W) \in(W \otimes \fg \otimes \fg)[\![x_2,x_3]\!]\cap  (\fg \otimes \fg \otimes \fg)[\![x_1,x_2,x_3]\!] = \{0\},
\end{equation}
since \(\fg[\![z]\!] \cap W = \{0\}\), concluding the proof.
\end{proof}
\begin{remark}
There are different methods to assign Lie algebras to certain classes of \(r\)-matrices, which should not be confused with the universal method described here. For example, in \cite{KPSST} the authors assign subalgebras of \(\fg(\!(z^{-1})\!)\times \fg\) to so-called quasi-trigonometric \(r\)-matrices and in \cite{abedin_burban} trigonometric solutions of the CYBE \cref{eq:intro_two_parameter_CYBE} are related to subalgebras of \(\mathfrak{L} \times \mathfrak{L}\), where \(\mathfrak{L}\) is a twisted loop algebra. 
\end{remark}

\begin{lemma}\label{lem:generators_of_g(r)}
Let \(r\) be a normalized formal generalized \(r\)-matrix. Then \(\fg(r)\) is generated as a Lie algebra by \(\fg(r)\cap z^{-1}\fg[\![z]\!] = \{(1 \otimes \alpha)r(z,0)\mid \alpha\in \fg^*\}\).
\end{lemma}
\begin{proof}
Let \(W\) be the Lie subalgebra of \(\fg(r)\) generated by \(\fg(r)\cap z^{-1}\fg[\![z]\!]\) and assume that for some \(m \in \bN\) we have  \(\fg(r)\cap z^{-m}\fg[\![z]\!] \subseteq W\). For every pair \(a_1,a_2 \in \fg\) exist unique \(\widetilde{a}_1,\widetilde{a}_2 \in W\) and \(s \in \fg(r)\) such that
\begin{align}
  \widetilde{a}_1(z) -a_1z^{-1},
\widetilde{a}_1(z) - a_2z^{-m}, s - [a_1,a_2]z^{-m-1} \in \fg[\![z]\!].
\end{align}
Since \([\widetilde{a}_1,\widetilde{a}_2] \in W\) and  \(s-[\widetilde{a}_1,\widetilde{a}_2] \in \fg(r)\cap z^{-m}\fg[\![z]\!]\subseteq W\), we see that \(s\in W\). Therefore, \([\fg,\fg] = \fg\) implies that \(\fg(r)\cap z^{-m-1}\fg[\![z]\!] \subseteq W\) and \(W = \fg(r)\) is verified by induction on \(m\).
\end{proof}

\begin{lemma}\label{lem:equivalence_on_subalgebra_level}
Let \(\widetilde{r} \in (\fg \otimes \fg)(\!(x)\!)[\![y]\!]\) be equivalent to a formal generalized \(r\)-matrix \(r\) via an equivalence \((\mu,w,\varphi)\). Then \(\fg(\widetilde{r})\) is the image of \(\fg(r)\) under the map \(\varphi_w \in \Aut_{\Bbbk\textnormal{-alg}}(\fg(\!(z)\!))\) defined by \(a(z) \mapsto \varphi(z)a(w(z))\).   
\end{lemma}
\begin{proof}
First note that for any \(s(x,y) = \sum_{k = 0}^\infty \sum_{i = 1}^ds_{k,i}(x) \otimes b_iy^k \in (\fg \otimes \fg)(\!(x)\!)[\![y]\!]\) and \(\lambda(z) = \sum_{k = 0}^\infty \lambda_kz^k \in \Bbbk[\![z]\!]^\times\) we have
\[\widetilde{s}(x,y) \coloneqq \lambda(y)s(x,y) = \sum_{k = 0}^\infty \sum_{i = 1}^d \left( \sum_{\ell = 0}^k\lambda_\ell s_{k-\ell,i}(x)\right) \otimes b_iy^k\]
and hence \(\fg(\widetilde{s}) = \fg(s)\). Therefore, we may assume that \(r\) is normalized and \(\mu(z) = w'(z)\). Then \(\widetilde{r}\) is also a normalized generalized \(r\)-matrix, as can be seen in the proof of \cref{lemm:formal_equivalence}.

Since \(\varphi_w\colon \fg(\!(z)\!) \to \fg(\!(z)\!)\) defined by \(a(z) \mapsto \varphi(z)a(w(z))\) is a \(\Bbbk\)-linear automorphism of Lie algebras and 
\(\fg(r) \cap z^{-1}\fg[\![z]\!] = \{(1 \otimes \alpha)r(z,0)\mid \alpha \in \fg^*\}\)
generates \(\fg(r)\) by \cref{lem:generators_of_g(r)}, \(\varphi_w(\fg(r))\) is generated by \(\varphi_w(\fg(r) \cap z^{-1}\fg[\![z]\!])\). We have
\begin{align*}
    \varphi_w((1\otimes \alpha)r(z,0)) = (\varphi(z)\otimes \alpha)r(w(z),0)= \left(1\otimes \left(\mu(0)^{-1}\alpha\varphi(0)^{-1}\right)\right)\widetilde{r}(z,0),
\end{align*}
where \(w(0) = 0\), \(\mu(0) \in \Bbbk^\times\) and \(\varphi(0)\in\Aut_{\Bbbk\textnormal{-alg}}(\fg)\) was used. Since \(\alpha \mapsto \mu(0)^{-1}\alpha\phi(0)^{-1}\) defines an linear automorphism of \(\fg^*\), we see that 
\begin{align*}
    \varphi_w(\fg(r) \cap z^{-1}\fg[\![z]\!]) = \fg(\widetilde{r}) \cap z^{-1}\fg[\![z]\!]
\end{align*}
and hence \(\phi(\fg(r)) = \fg(\widetilde{r})\) by applying \cref{lem:generators_of_g(r)}.
\end{proof}

\begin{remark}
\Cref{lem:equivalence_on_subalgebra_level} implies that for a formal generalized \(r\)-matrix \(r \in (\fg \otimes \fg)(\!(x)\!)[\![y]\!]\), \(\lambda \in \Bbbk[\![z]\!]^\times\) and \(\widetilde{r}(x,y) = \lambda(y)r(x,y)\), we have \(\fg(r) = \fg(\widetilde{r})\). In particular, \(\fg(r)\) is a Lie subalgebra of \(\fg(\!(z)\!)\) complementary to \(\fg[\![z]\!]\) for all not necessarily normalized formal generalized \(r\)-matrices \(r\) by virtue of \cref{prop:Lie_subalgebra_of_generalized_rmatrix}. 
\end{remark}

\begin{example}\label{ex:homogeneous_Lie_subalgebras}
It is easy to see that \(\fg(r_\textnormal{Yang}) = z^{-1}\fg[z^{-1}]\). This subalgebra is stable under multiplication by \(z^{-1}\). More generally, a subalgebra \(W \subseteq \fg(\!(z)\!)\) satisfying \(\fg(\!(z)\!) = \fg[\![z]\!] \oplus W\)  is called \emph{homogeneous} if \(z^{-1}W \subseteq W\). Such a subalgebra is automatically a deformation of \(\fg(r_\textnormal{Yang})\) in the sense that there exists \(A \in \End(\fg)[\![z]\!]\) such that \(A(0) = \textnormal{id}_\fg\) and 
\[W = A\fg(r_{\textnormal{Yang}}) = \textnormal{Span}_\Bbbk\{z^{-k-1}A(z)b_i\mid k \in \bN_0,i \in \{1,\dots,d\}\}.\]
The series \(A\) is thereby uniquely determined by \(z^{-1}A(z)b_i \in W\) for all \(i \in \{1,\dots,d\}\). 
The condition on \(A\) in order for \(W\) to be a Lie algebra is examined in \cite{golubchik_sokolov_compatible_Lie_brackets} and turns out to be related to the notion of \emph{compatible Lie brackets}.

Let \(r\) be the normalized formal generalized \(r\)-matrix such that \(W = \fg(r)\).
Since \(\overline{r}\) is also in normalized standard form, we can see that \(\fg(\!(z)\!) = \fg[\![z]\!] \oplus \fg(\overline{r})\), hence there also exists a unique invertible series \(\overline{A} \in \End(\fg)[\![z]\!]\) such that \(\overline{A}(0)=\textnormal{id}_\fg\) and \(\fg(\overline{r}) = \overline{A}\fg(r_{\textnormal{Yang}})\). We will show in \cref{ex:formula_for_rmatrices_of_homogeneous_Lie_algebras} below that
\begin{align}\label{eq:formula_for_rmatrices_of_homogeneous_Lie_algebras}
    r(x,y) = \frac{A(x)\otimes \overline{A}(y)}{x-y}\gamma,
\end{align}
where \(A(x)\otimes \overline{A}(y)\) is viewed as an element of \((\End(\fg) \otimes\End(\fg))[\![x,y]\!]\). 
In particular, \(r\) is skew-symmetric if and only if \(A = \overline{A}\).
\end{example}

\subsection{Skew-symmetry and formal \(r\)-matrices} This section is dedicated to the study of skew-symmetric formal generalized \(r\)-matrices. These obviously satisfy a formal version of the classical Yang-Baxter equation \cref{eq:intro_two_parameter_CYBE} and in fact turn out to be exactly solutions of this equation.

\begin{definition}\label{def:formal_rmatrix}
A series \(r \in (\fg\otimes \fg)(\!(x)\!)[\![y]\!]\) is called a
\emph{(normalized) formal \(r\)-matrix} if it is in (normalized) standard form and solves the \emph{formal classical Yang-Baxter equation} (formal CYBE)
\begin{align}\label{eq:formal_CYBE}
    \textnormal{CYB}(r) = 0\textnormal{, where }\textnormal{CYB}(r) \coloneqq [r^{12},r^{13}] + [r^{12},r^{23}] + [r^{13},r^{23}].
\end{align}
Here \cref{not:ij_notations} were used 
and the brackets in \(\textnormal{CYB}(r)\) are the usual commutators in the associative \(R\)-algebra \((\textnormal{U}(\fg) \otimes \textnormal{U}(\fg) \otimes \textnormal{U}(\fg))\otimes R\).
\end{definition}

\begin{proposition}\label{prop:formal_rmatrices_are_skew-symmetric}
A series \(r \in (\fg\otimes \fg)(\!(x)\!)[\![y]\!]\) is a formal \(r\)-matrix if and only if it is a skew-symmetric formal generalized \(r\)-matrix.
\end{proposition}
\begin{proof}
It is obvious that a formal generalized \(r\)-matrix \(r\) solves the formal CYBE \cref{eq:formal_CYBE} if \(\overline{r} = r\), so we have to prove the contrary, i.e.\ that each formal \(r\)-matrix solves the formal GCYBE \cref{eq:formal_GCYBE} and satisfies \(\overline{r} = r\). The equations \cref{eq:second_half_of_GCYB}, where the \(\overline{r}\) is replaced by \(r\), and \cref{eq:first_half_of_GCYB} imply that \(\fg(r) \subseteq \fg(\!(z)\!)\) is a Lie subalgebra since \(\textnormal{CYB}(r) = 0\). Therefore, \cref{prop:Lie_subalgebra_of_generalized_rmatrix} states that \(\textnormal{GCYB}(r) = 0\). In particular, we have:
\begin{align}\label{eq:proof_of_skew_symmetry}
    0 = \textnormal{CYB}(r)-\textnormal{GCYB}(r) = [r^{13},r^{23}- \overline{r}^{23}].
\end{align}
Multiplying \cref{eq:proof_of_skew_symmetry} with \(x_1-x_3\), setting \(x_1 = x_3\) and probably multiplying with an element of \(\Bbbk[\![x_3]\!]^\times\) results in
\([\gamma^{13},r^{23}- \overline{r}^{23}] = 0\).
Application of the map \(a_1 \otimes a_2 \otimes a_3 \longmapsto a_2 \otimes [a_1,a_3]\) gives the desired \(\overline{r} = r\). Here we used the following fact:
\begin{equation}\label{eq:fact_about_Casimir_element}
    \textnormal{The image of }\gamma\textnormal{ under }\mu\colon \fg \otimes \fg \to \End(\fg) \textnormal{ defined by }a_1\otimes a_2 \mapsto \textnormal{ad}(a_1)\textnormal{ad}(a_2)\textnormal{ is } \textnormal{id}_\fg.
\end{equation}
Indeed, if \(\fg = \bigoplus_{i = 1}^k\fg_i\) is the decomposition of \(\fg\) into simple ideals, we have  \(\gamma = \sum_{i = 1}^k\gamma_i\), where \(\gamma_i\) is the Casimir element of \(\fg_i\), so we may assume that \(\fg\) is simple. Furthermore, an element \(f \in \End(\fg)\) with the property \(f \otimes \textnormal{id}_{\overline{\Bbbk}} = \textnormal{id}_{\fg \otimes \overline{\Bbbk}}\), where \(\overline{\Bbbk}\) is the algebraic closure of \(\Bbbk\), already satisfies \(f = \textnormal{id}_\fg\), hence we may assume \(\Bbbk = \overline{\Bbbk}\). The endomorphism \(\mu(\gamma) = \sum_{i = 1}^d \textnormal{ad}(b_i)^2\) is the quadratic Casimir operator of the adjoint representation, which is a multiple of the identity due to Schur's Lemma and equals the identity since \(\textnormal{Tr}(\textnormal{id}_\fg) = d =\sum_{i = 1}^d\kappa(b_i,b_i) = \textnormal{Tr}(\mu(\gamma)).\)
\end{proof}

\begin{notation}
Let us denote the \(\Bbbk(\!(z)\!)\)-bilinear extension \(\fg(\!(z)\!)\times \fg(\!(z)\!)\to \Bbbk(\!(z)\!)\) of the Killing form \(\kappa\) of \(\fg\) with the same symbol. Then \(\fg(\!(z)\!)\) is equipped with the \(\Bbbk\)-bilinear form \(\kappa_0\) defined by
\begin{align}
    \kappa_0(s,t) \coloneqq \textnormal{res}_0\kappa(s,t)\textnormal{d}z = \sum_{k+\ell = -1}\kappa(s_k,t_\ell) \textnormal{ for all }s = \sum_{k \gg -\infty}s_kz^k,t = \sum_{k \gg -\infty}t_kz^k \in \fg(\!(z)\!),
\end{align}
where \( \textnormal{res}_0 f \textnormal{d}z = f_{-1}\) for any series \(f = \sum_{k \gg -\infty}f_kz^k \in  \Bbbk(\!(z)\!)\).  
\end{notation}

\begin{remark}\label{rem:Laurent_Killing_form}
The \(\Bbbk(\!(z)\!)\)-bilinear extension of \(\kappa\) is the Killing form of \(\fg(\!(z)\!)\) as a Lie algebra over \(\Bbbk(\!(z)\!)\). Therefore, \(\gamma\) can also be understood as the Casimir element of \(\fg(\!(z)\!)\). In particular, \([a \otimes 1 + 1 \otimes a,\gamma] = 0\) for all \(a \in \Bbbk(\!(z)\!)\) and \((\varphi(z) \otimes \varphi(z))\gamma = \gamma\) for all 
\begin{equation}
    \varphi \in \textnormal{Aut}_{\Bbbk[\![z]\!]\textnormal{-alg}}(\fg[\![z]\!]) \subseteq \textnormal{Aut}_{\Bbbk(\!(z)\!)\textnormal{-alg}}(\fg(\!(z)\!)).
\end{equation}
Moreover, \(\kappa_0\) is symmetric, non-degenerate and invariant, i.e.\ \(\kappa_0([a,b],c) = \kappa_0(a,[b,c])\) for all \(a,b,c\in\fg(\!(z)\!)\).
\end{remark}

\begin{lemma}\label{lem:orthogonal_complement_of_g(r)}
Let \(r\) be a normalized formal generalized \(r\)-matrix. Then \(\fg(r)^\bot = \fg(\overline{r})\) and \(r\) is skew-symmetric if and only if \(\fg(r)^\bot = \fg(r)\). 
\end{lemma}
\begin{proof}
Let us write 
\begin{equation}
    r(x,y) = \sum_{k = 0}^\infty \sum_{i = 1}^d r_{k,i}(x)\otimes y^kb_i\textnormal{ and }\overline{r}(x,y) = \sum_{k = 0}^\infty \sum_{i = 1}^d \overline{r}_{k,i}(x)\otimes y^kb_i
\end{equation}
as well as \(r_{k,i}(z) = z^{-k-1}b_i + \sum_{\ell = 0}^\infty\sum_{j = 1}^d r_{k,i}^{\ell,j}z^\ell b_j\). Then we have
\begin{align}
    r(x,y)-r_{\textnormal{Yang}}(x,y) = \sum_{k,\ell = 0}^\infty \sum_{i,j=1}^d r_{k,i}^{\ell,j} x^\ell b_j \otimes y^kb_i
\end{align}
and hence \(\overline{r}_{\ell,j}(z) = z^{-\ell-1}b_j - \sum_{k = 0}^\infty\sum_{i = 1}^d r_{k,i}^{\ell,j}z^k b_i\). Therefore, we can deduce
\begin{equation}
    \kappa_0(r_{k,i},\overline{r}_{\ell,j}) = r_{k,i}^{\ell,j} -r_{k,i}^{\ell,j} = 0.
\end{equation}
This implies that \(\fg(\overline{r}) \subseteq \fg(r)^\bot\). Now \(0 = \fg[\![z]\!] \cap \fg(r)^\bot = (\fg[\![z]\!] \oplus \fg(r))^\bot\) and \(\fg(\!(z)\!) = \fg[\![z]\!] \oplus \fg(\overline{r})\), since \(\overline{r}\) is in normalized standard form, show that \(\fg(r)^\bot = \fg(\overline{r})\). In particular, we see that \(r = \overline{r}\) implies \(\fg(r)^\bot = \fg(r)\). On the other hand, since both \(r\) and \(\overline{r}\) are of normalized standard form, \(\fg(r) = \fg(r)^\bot = \fg(\overline{r})\) forces 
\begin{equation}
    r- \overline{r} \in (\fg(r) \otimes \fg)[\![y]\!] \cap (\fg \otimes \fg)[\![x,y]\!] = \{0\}.
\end{equation}
We can conclude that \(\fg(r)^\bot = \fg(r)\) if and only if \(\overline{r} = r\).
\end{proof}

\begin{example}\label{ex:formula_for_rmatrices_of_homogeneous_Lie_algebras}
We now have all ingredients to prove \cref{eq:formula_for_rmatrices_of_homogeneous_Lie_algebras}. Recall the setting in \cref{ex:homogeneous_Lie_subalgebras}: we have invertible series \(A,\overline{A} \in \textnormal{End}(\fg)[\![z]\!]\) satisfying \(A(0) = \overline{A}(0) = \textnormal{id}_\fg\) and \(\fg(r) = z^{-1}A\fg[z^{-1}]\), \(\fg(\overline{r}) = z^{-1}\overline{A}\fg[z^{-1}]\) for a normalized formal generalized \(r\)-matrix \(r\). It is easy to see that \(\fg(\overline{r})= \fg(r)^\bot\) implies that
\[\textnormal{res}_0z^k\kappa(Aa_1,\overline{A}a_2)\textnormal{d}z = \begin{cases}\kappa(a_1,a_2)& k = 1\\ 0&k\neq 1 \end{cases}\]
for all \(a_1,a_2 \in \fg\). From this we can deduce that \(\kappa(Aa_1,\overline{A}a_2) = \kappa(a_1,a_2)\) for all \(a_1,a_2 \in \fg(\!(z)\!)\) and as a consequence \((A(z) \otimes \overline{A}(z))\gamma = \gamma\). Therefore,
\[\widetilde{r}(x,y) \coloneqq \frac{A(x) \otimes \overline{A}(y)}{x-y}\gamma \in (\fg \otimes \fg)(\!(x)\!)[\![y]\!]\]
is in normalized standard form by virtue of \cref{lem:series_vanishing_at_diagonal}. Furthermore, it is straight forward to verify that \(\widetilde{r} \in (\fg(r) \otimes \fg)[\![y]\!]\) and hence
\[r-\widetilde{r} \in (\fg(r) \otimes \fg)[\![y]\!]\cap (\fg \otimes \fg)[\![x,y]\!] = \{0\},\]
where we used that \(r\) and \(\widetilde{r}\) are both of normalized standard form.
\end{example}

Using \cref{lem:orthogonal_complement_of_g(r)} it is possible to equip the Lie algebra associated to normalized formal \(r\)-matrix with additional structure, namely a dual bracket defining a Lie bialgebra structure. Let us recall what this means.
\begin{definition}\label{def:Lie_bialgebra}
A Lie algebra \(\mathfrak{l}\) over \(\Bbbk\) equipped with a skew-symmetric map \(\delta \colon \mathfrak{l} \to \mathfrak{l} \otimes \mathfrak{l}\), called \emph{Lie cobracket}, is called \emph{Lie bialgebra} if \(\delta^*\colon (\mathfrak{l} \otimes \mathfrak{l})^* \to \mathfrak{l}^*\) restricted to \(\mathfrak{l}^*\otimes \mathfrak{l}^* \subseteq (\mathfrak{l} \otimes \mathfrak{l})^*\) defines a Lie algebra structure on \(\mathfrak{l}^*\) and \(\delta\) is a 1-cocycle, i.e.\
\begin{align*}
    \delta([a,b]) = [a\otimes 1 + 1 \otimes a, \delta(b)] - [b\otimes 1 + 1 \otimes b,\delta(a)]
\end{align*}
for all \(a,b\in\mathfrak{l}\).
\end{definition}

The following statement is a reformulation of \cite[Proposition 6.2]{etingof_schiffmann} for normalized formal \(r\)-matrices.

\begin{proposition}\label{prop:Lie_bialgebra_of_formal_rmatrix}
For a normalized formal \(r\)-matrix \(r\in (\fg \otimes \fg)(\!(x)\!)[\![y]\!]\) the Lie algebra \(\fg(r)\) equipped with the linear map \(\delta\colon \fg(r) \to \fg(r) \otimes \fg(r)\) defined by 
\begin{align*}
    &\delta(a)(x,y) \coloneqq [a(x) \otimes 1 + 1 \otimes a(y), r(x,y)]
\end{align*}
for all \(a \in \fg(r)\) is a Lie bialgebra. 
\end{proposition}
\begin{proof}
If \(r(x,y) = \sum_{k = 0}^\infty \sum_{i = 1}^d r_{k,i}(x)\otimes b_iy^k\), \(\textnormal{CYB}(r) = 0\) can be written as \([r^{13}+r^{23},r^{12}]= [r^{13},r^{23}]\), which is equivalent to
\begin{align}\label{eq:cybedeterminesdelta}
    \sum_{k= 0}^\infty \sum_{i = 1}^d \delta(r_{k,i})(x_1,x_2) \otimes b_ix^k= \sum_{k,\ell = 0}^\infty \sum_{i,j = 1}^d r_{k,i}(x_1) \otimes r_{\ell,j}(x_2) \otimes [b_ix_3^{k},b_jx_3^{\ell}].
\end{align}
This shows that \(\delta\colon \fg(r) \to \fg(r) \otimes \fg(r)\) is well-defined and the Jacobi-identity in \(\fg(\!(z)\!)\) implies that \(\delta\) is a 1-cocycle. Combining \(\kappa_0(r_{k,i}(z),b_jz^\ell) = \delta_{ij}\delta_{k\ell}\) for all \(i,j \in \{1,\dots,d\},k,\ell \in \bN_0\) with the fact that \(\fg(\!(z)\!) = \fg[\![z]\!]\oplus \fg(r)\) shows that \(a \mapsto \kappa_0(a,\cdot)\) defines an isomorphism \(\kappa_0^\textnormal{a}\colon \fg[\![z]\!]\to \fg(r)^*\).
Applying \(\kappa_0^{\otimes 3}(\cdot,b_{i}x_1^{k} \otimes b_{j}x_2^{\ell} \otimes r_{m,n}(x_3))\) to \cref{eq:cybedeterminesdelta} yields
\begin{align}\label{eq:Lie_bialgebra_of_formal_rmatrix}
    \kappa_0^{\otimes2}(\delta(r_{m,n})(x,y),b_{i}x^{k} \otimes b_{j}y^{\ell}) = \kappa_0(r_{m,n}(z),[b_{i}z^{k},b_{j}z^{\ell}]).
\end{align}
Here we wrote 
\[B^{\otimes n}(v_1\otimes \dots \otimes v_n,w_1\otimes \dots \otimes w_n) = B(v_1,w_1) \dots B(v_n,w_n)\]
for a bilinear form \(B\) on a vector space \(V\) and \(v_1,\dots,v_n,w_1,\dots,w_n \in V\).
Equation \cref{eq:Lie_bialgebra_of_formal_rmatrix} implies that \(\delta^*\) is identified with the standard Lie bracket of \(\fg[\![z]\!]\) after identifying \(\fg(r)^*\) with \(\fg[\![z]\!]\) via \(\kappa_0^\textnormal{a}\), hence \(\fg(r)\) is a Lie bialgebra if equipped with \(\delta\).
\end{proof}
\begin{remark}
For a normalized formal \(r\)-matrix \(r\) the triple \((\fg(\!(z)\!),\fg[\![z]\!],\fg(r))\) is a  so-called \emph{Manin triple} and the property  \cref{eq:Lie_bialgebra_of_formal_rmatrix} implies that this Manin triple \emph{determines} the Lie bialgebra structure of \(\delta\). 
\end{remark}

The Lie bialgebra structure on the Lie algebra associated to a normalized formal \(r\)-matrix can be used to derive the following version of the result of \cite{belavin_drinfeld_diffrence_depending}.

\begin{proposition}\label{prop: formal_rmatrix_depends_on_difference}
Let \(r \in (\fg \otimes \fg)(\!(x)\!)[\![y]\!]\) be a normalized formal \(r\)-matrix. There exists \(s \in z^{-1}(\fg\otimes\fg)[\![z]\!]\) and \(\varphi \in \textnormal{Aut}_{\Bbbk[\![z]\!]\textnormal{-alg}}(\fg[\![z]\!])\) such that \(s(x-y) = (\varphi(x) \otimes \varphi(y))r(x,y)\). Furthermore, \(\fg(s) \coloneqq \varphi(\fg(r))\) is closed under the formal derivative \(\textnormal{d}/\textnormal{d}z\colon\fg(\!(z)\!) \to \fg(\!(z)\!)\), \(a(z) \mapsto a'(z)\).
\end{proposition}
\begin{proof}
\textbf{Step 1.} \emph{Setup. }
Let \(\delta \colon \fg(r) \to \fg(r) \otimes \fg(r)\) be as in \cref{prop:Lie_bialgebra_of_formal_rmatrix}. Consider the \emph{canonical derivation} \(D\colon \fg(r) \to \fg(r)\) of the Lie bialgebra \(\fg(r)\), i.e.\ the composition of \(\delta\) with the Lie bracket \([,]\colon\fg(r) \otimes \fg(r) \to \fg(r)\). As the name suggests, the 1-cocycle condition of \(\delta\) implies that \(D\) is a derivation of \(\fg(r)\). \\
\textbf{Step 2.} \emph{\(D = \textnormal{d}/\textnormal{d}z - \textnormal{ad}(h(z,z))\) for an appropriate \(h(x,y)\in(\fg\otimes \fg)[\![x,y]\!]\). }
The linear map \(\fg \otimes \fg \to \textnormal{End}(\fg)\) defined by \(a_1 \otimes a_2 \mapsto \textnormal{ad}(a_1)\textnormal{ad}(a_2)\) maps \(\gamma\) to \(\textnormal{id}_\fg\); see \cref{eq:fact_about_Casimir_element}. Combining this with the fact that \([a \otimes 1 + 1 \otimes a,\gamma] = 0\) for all \(a \in \fg[\![z]\!]\) and \cref{lem:series_vanishing_at_diagonal} results in
\begin{align}
    \left[a(x)\otimes 1 + 1\otimes a(y),\frac{\gamma}{x-y} \right] = \left[\frac{a(x)-a(y)}{x-y} \otimes 1,\gamma\right] \stackrel{[,]}{\longmapsto} a'(z).
\end{align}
If we write \(h(x,y) \in \fg[\![x,y]\!]\) for the image of \(r(x,y) - r_\textnormal{Yang}(x,y) \in (\fg \otimes \fg)[\![x,y]\!]\) under the \(\Bbbk[\![x,y]\!]\)-linear extension of the Lie bracket \(\fg \otimes \fg \to \fg\) and use the fact that
\begin{align}
    [a \otimes 1 + 1 \otimes a, c\otimes d] = [a,c]\otimes d + c \otimes [a,d]\stackrel{[,]}\longmapsto [[a,c],d] + [c,[a,d]] = [a,[c,d]]
\end{align}
holds for all \(a,b,c \in \fg\), we can conclude that
\(D(a)(z) = a'(z) - [h(z,z),a(z)]\). \\
\textbf{Step 3.} \emph{There exists \(\varphi\in\textnormal{Aut}_{\Bbbk[\![z]\!]\textnormal{-alg}}(\fg[\![z]\!])\) such that \(\varphi(\fg(r))\) is closed under the formal derivative. }
It is not hard to see by induction on the coefficients that there exists a unique \(\psi \in \textnormal{End}(\fg)[\![z]\!]\) satisfying \(\psi'(z) = \textnormal{ad}(h(z,z))\psi(z)\) and \(\psi(0) = \textnormal{id}_\fg\). For every \(a_1,a_2 \in \fg\) the series
\begin{equation}
    c_1(z)\coloneqq \psi(z)[a_1,a_2], c_2(z)\coloneqq [\psi(z)a_1,\psi(z)a_2] \in \fg[\![z]\!]
\end{equation}
satisfy \(c_i'(z) = [h(z,z),c_i(z)]\) and \(c_i(0) = [a_1,a_2]\) for \(i \in \{1,2\}\). Therefore, a coefficient comparison forces \(c_1 = c_2\) and thus \(\psi\) defines an element of \(\textnormal{Aut}_{\Bbbk[\![z]\!]\textnormal{-alg}}(\fg[\![z]\!])\). Let \(\varphi(z) \coloneqq \psi(z)^{-1}\in\textnormal{End}(\fg)[\![z]\!]\), i.e.\ \(\varphi(z)\psi(z) = \textnormal{id}_{\fg}\), and note that \(\varphi(z)\) also defines an element of \(\textnormal{Aut}_{\Bbbk[\![z]\!]\textnormal{-alg}}(\fg[\![z]\!])\). Consider the normalized formal \(r\)-matrix \(\widetilde{r}(x,y) \coloneqq (\varphi(x)\otimes \varphi(y))r(x,y)\) and let \(\widetilde{\delta}\) be the Lie cobracket from \cref{prop:Lie_bialgebra_of_formal_rmatrix} of \(\fg(\widetilde{r})\). Clearly \(\widetilde{\delta}(a)(x,y) = (\varphi(x) \otimes \varphi(y))\delta(\psi(a))(x,y)\) for all \(a \in \fg(\widetilde{r})\). Therefore, the canonical derivation \(\widetilde{D}\) of \(\fg(\widetilde{r})\) reads
\begin{align*}
    \widetilde{D}(a)(z) = \varphi(z)D(\psi(a))(z) &= \varphi(z)(\psi'(z)a(z) + \psi(z)a'(z)- [h(z,z),\psi(z)a(z)]) \\&= a'(z) + \varphi(z)([h(z,z),\psi(z)a(z)]-[h(z,z),\psi(z)a(z)]) = a'(z).
\end{align*}
In particular, \(\widetilde{D}\colon \fg(\widetilde{r}) \to \fg(\widetilde{r})\) is the restriction of the formal derivative to \(\fg(\widetilde{r})\). \\
\textbf{Step 4.} \emph{\(s(z) \coloneqq \widetilde{r}(z,0)\) concludes the proof. }
Since \(\fg(\widetilde{r})\) is closed under \(\widetilde{D}\), which is the restriction of the formal derivative to \(\fg(\widetilde{r})\), we have
\[(1 \otimes \kappa(b_i,\cdot))\frac{(-1)^k}{k!}\widetilde{r}^{(k)}(z,0) \in \fg(\widetilde{r}) \cap (b_iz^{-k-1} + \fg[\![z]\!])\]
for all \(i \in \{1,\dots,d\}\), \(k \in \bN_0\), where \(\widetilde{r}^{(k)}(z,0) \coloneqq \widetilde{D}^k \widetilde{r}(z,0)\). The proof of \cref{prop:Lie_subalgebra_of_generalized_rmatrix} and the expansion of \(\widetilde{r}(x-y,0)\) as a series imply that
\[\widetilde{r}(x,y) = \sum_{k = 0}^\infty\frac{(-1)^k}{k!}\widetilde{r}^{(k)}(x,0)y^k = \widetilde{r}(x-y,0)\]
and we can see that putting \(s(z) \coloneqq \widetilde{r}(z,0)\) concludes the proof.
\end{proof}

 \section{Some results about sheaves of algebras}\label{sec:sheaves_of_algebras}

\subsection{Definition and generalities}
In \cref{sec:geometrization} we will assign a geometric datum to any formal generalized \(r\)-matrix. The main ingredient of this datum consists of a sheaf of Lie algebras on a projective curve constructed from the subalgebra associated to that formal generalized \(r\)-matrix. In this section we will discuss all definitions and properties related to sheaves of Lie algebras that are required in the subsequent sections. Although we are interested in sheaves of Lie algebras, we will consider general sheaves of algebras as long as it does not obscure the constructions or restricts the scope of the statements. 

\begin{remark}
In this text, an \(R\)-algebra \(A\) over some unital commutative ring \(R\) does not necessarily satisfy any additional axioms, i.e.\ \(A\) is simply a pair \((A,\mu_A)\) consisting of a \(R\)-module \(A\) equipped with an \(R\)-linear map \(\mu_A\colon A\otimes_R A \to A\) called multiplication. In particular, a Lie algebra over \(R\) is a \(R\)-algebra.
\end{remark}

\begin{definition}\label{def:sheaf_of_algebras}
Let \(X = (X,\cO_X)\) be a ringed space.
\begin{itemize}
    \item A \emph{sheaf of algebras} \(\cA\) on \(X\) is a pair \((\cA,\mu_\cA)\) consisting of an \(\cO_X\)-module \(\cA\) equipped with an \(\cO_X\)-linear morphism \(\mu_\cA\colon\cA \otimes_{\cO_X} \cA \to \cA\) called \emph{multiplication of \(\cA\)}. 
    \item A \emph{morphism  of sheaves of algebras} \(f\colon\cA \to \cB\) is an \(\cO_X\)-linear morphism satisfying \(f\mu_\cA = \mu_{\cB} (f\otimes f)\).
    \item A sheaf of algebras \(\cA\) on \(X\) is called \emph{sheaf of Lie algebras} if \( [,]_\cA \coloneqq\mu_\cA\) defines a Lie algebra structure on all local section. A \emph{morphism of sheaves of Lie algebras} is simply a morphism of sheaves of algebras. The adjoint representations of the local sections of \(\cA\) induce an \(\cO_X\)-linear morphism \(\textnormal{ad}_\cA \colon \cA \to \sheafEnd_{\cO_X}(\cA)\) called \emph{adjoint representation of \(\cA\)}.
\end{itemize}
\end{definition}

\begin{remark}
Let \(f \colon Y \to X\) be a morphism of locally ringed spaces, \(\cA\) (resp. \(\mathcal{B}\)) be a sheaf of algebras on \(X\) (resp. \(Y\)). The sheaf of algebras \(f^*\cA\) (resp. \(f_*\mathcal{B}\)) is naturally a sheaf of algebras on \(Y\) (resp. \(X\)) with the multiplication defined through
\begin{align}
    &f^*\cA \otimes_{\cO_Y} f^*\cA \stackrel{\cong}\longrightarrow f^*(\cA \otimes_{\cO_X}\cA)\stackrel{f^*\mu_\cA}\longrightarrow \cA &(\textnormal{resp. }
    f_*\mathcal{B} \otimes_{\cO_X} f_*\mathcal{B} \longrightarrow f_*(\mathcal{B} \otimes_{\cO_Y}\mathcal{B})\stackrel{f_*\mu_\mathcal{B}}\longrightarrow \mathcal{B}),
\end{align}
respectively, where the unlabeled arrows are the canonical ones. If \(\cA\) (resp. \(\mathcal{B}\)) is a sheaf of Lie algebras, it is easy to see that \(f^*\cA\) (resp. \(f_*\cB\)) is a sheaf of Lie algebras.

In particular, for any \(p \in X\) with residue field \(\Bbbk(p)\), the stalk \(\cA_p\) (resp. the fiber \(\cA|_p\)) is naturally an \(\cO_{X,p}\)-algebra (resp. \(\Bbbk(p)\)-algebra), which is a Lie algebra if \(\cA\) is a sheaf of Lie algebras. Indeed, this can be seen as a special case of the inverse image construction by choosing \(Y = (\{p\},\cO_{X,p})\) (resp. \(Y = (\{p\},\Bbbk(p))\)) and considering the inverse image with respect to the canonical morphism \(f \colon Y \to X\).
\end{remark}

It will turn out to be useful that the Killing form of a Lie algebra admits a geometric analog.

\begin{definition}\label{def:sheaf_Killing_form}
The Killing form \(K_\mathcal{A}\) of a finite locally free sheaf \(\mathcal{A}\) of Lie algebras on a ringed space \(X = (X,\cO_X)\) is the \(\cO_X\)-bilinear morphism
\begin{align}\label{eq:definition_sheaf_Killing_form}
    \xymatrix{\mathcal{A} \times \mathcal{A} \ar[rr]^-{\textnormal{ad}_{\mathcal{A}} \times \textnormal{ad}_{\mathcal{A}}}&& \sheafEnd_{\cO_X}(\mathcal{A}) \times \sheafEnd_{\cO_X}(\mathcal{A}) \ar[r]& \sheafEnd_{\cO_X}(\mathcal{A}) \ar[r]^-{\textnormal{Tr}_\cA}& \cO_X},
\end{align}
where \(\textnormal{ad}_\cA\) is the adjoint representation of \(\cA\), the second map is given by composition and \(\textnormal{Tr}_\cA\) is the sheaf trace of \(\cA\).
\end{definition}

\begin{lemma}\label{lem:fiber_of_sheaf_Killing_form}
Let \(f \colon Y \to X\) be a morphism of locally ringed spaces and \(\cA\) be a finite locally free sheaf of Lie algebras on \(\cA\) with Killing form \(K_\cA\). Then \(f^*K_\cA\) is the Killing form of the finite locally free sheaf of algebras \(f^*\cA\) on \(Y\). In particular, for any point \(p \in X\) the fiber \(K_{\cA}|_p\) coincides with the Killing form of \(\cA|_p\). 
\end{lemma}
\begin{proof}
The canonical map \(\chi \colon f^*\sheafEnd_{\cO_X}(\cA) \longrightarrow \sheafEnd_{\cO_Y}(f^*\cA)\)  coincides with the isomorphism
\begin{equation}
    \End_{\cO_{X,f(q)}}(\cA_{f(q)})\otimes_{\cO_{X,f(q)}}\cO_{Y,q} \cong \End_{\cO_{Y,q}}(\cA_{f(q)}\otimes_{\cO_{X,f(q)}}\cO_{Y,q})
\end{equation}
in the stalk in any point \(q\in Y\), where we used that \(\cA\) is finite locally free. This shows that
\begin{align}
    f^*\cA \stackrel{f^*\textnormal{ad}_\cA}\longrightarrow f^*\sheafEnd_{\cO_X}(\cA) \stackrel{\chi}\longrightarrow \sheafEnd_{\cO_Y}(f^*\cA)
    \textnormal{ and }
    f^*\sheafEnd_{\cO_X}(\cA) \stackrel{\chi}\longrightarrow \sheafEnd_{\cO_Y}(f^*\cA) \stackrel{\textnormal{Tr}_{f^*\cA}}\longrightarrow \cO_Y,
\end{align}
coincide with \(\textnormal{ad}_{f^*\cA}\) and \(f^*\textnormal{Tr}_\cA\) respectively and that \(\chi\) is compatible with the composition of endomorphisms of sheaves. Therefore, applying \(f^*\) to \cref{eq:definition_sheaf_Killing_form} and using \(\chi\) implies that \(f^*K_\cA\) coincides with the Killing form of \(f^*\cA\). The observation that the functor \((\cdot)|_p\) can be realized by the inverse image via \((\{p\},\Bbbk(p))\to X\), where \(\Bbbk(p)\) is the residue field of \(p\), concludes the proof.
\end{proof}

\subsection{Local triviality of sheaves of algebras}\label{subsec:local_triviality_of_sheaves_of_Lie_algebras}
Let \(\cA\) be a coherent sheaf of algebras on a \(\Bbbk\)-scheme \(X\) of finite type. There are several different notions of local triviality for \(\cA\).

\begin{definition}\label{def:local_triviality}
Let \(A\) be a finite-dimensional \(\Bbbk\)-algebra.
\begin{itemize}
    \item \(\cA\) is called \emph{weakly \(A\)-locally free in \(p \in X\)} if \(\cA|_p \cong A \otimes \Bbbk(p)\) as \(\Bbbk(p)\)-algebras, where \(\Bbbk(p)\) is the residue field of \(p\).
    \item \(\cA\) is called \emph{formally \(A\)-locally free in \(p \in X\)} if \(\widehat{\cA}_p \cong A \otimes \compO_{X,p}\) as \(\compO_{X,p}\)-algebras.
    \item \(\cA\) is called \emph{\'etale \(A\)-locally free in \(p \in X\)} if there exists an \'etale morphism \(f\colon U\to X\) of \(\Bbbk\)-schemes such that \(p \in f(U)\) and \(f^*\cA \cong A \otimes \cO_U\) as sheaves of algebras.
    \item \(\cA\) is called \emph{Zarsiki \(A\)-locally free in \(p \in X\)} if there exists an open neighbourhood \(U\) of \(p\) such that \(\cA|_U \cong A \otimes \cO_U\) as sheaves of algebras.
    \item \(\cA\) is called \emph{weakly} (resp. \emph{formally, \'etale, Zariski}) \emph{\(A\)-locally free} if it has this property in all \(p \in X\).
\end{itemize}
\end{definition}

\begin{remark}\label{rem:local_triviality_of_sheaves_of_algebras}
Let \(p \in X\) and \(A\) be a finite-dimensional \(\Bbbk\)-algebra.
Obviously, \(\cA\) is weakly \(A\)-locally free in \(p \in X\) if it is formally \(A\)-locally free in \(p\). 
Since open immersions are \'etale, \(\cA\) is both formally and \'etale \(A\)-locally free in \(p\) if it is Zariski \(A\)-locally free in \(p\).
Furthermore, if \(\Bbbk = \overline{\Bbbk}\) and \(p\) is closed, \(\cA\) is formally \(A\)-locally free in \(p\) if it is \'etale \(A\)-locally free in \(p\).
\end{remark}

\begin{remark}\label{rem:etale_morphisms_wlog_surjective}
Let \(A\) be a finite-dimensional \(\Bbbk\)-algebra.
Note that if \(\cA\) is \'etale \(A\)-locally free in a point \(p \in X\) it is already \'etale \(A\)-locally free in an open neighbourhood of \(p\), since \'etale morphisms are open.  
Assume we have a set of \'etale morphisms \(f_i\colon U_i \to X\) for \(i \in \{1,\dots,n\}\) such that \(X = \bigcup_{i = 1}^n f_i(U_i)\). Then \(f\coloneqq \coprod_{i = 1}^nf_i \colon Y \coloneqq \coprod_{i = 1}^nU_i \to X\) is surjective and \'etale. If \(U_i\) is affine for all \(i \in \{1,\dots,n\}\), \(Y\) is also affine. This shows that if \(\cA\) is \'etale \(A\)-locally free, there exists a surjective \'etale morphism \(f \colon Y \to X\) of \(\Bbbk\)-schemes such that \(Y\) is affine and \(f^*\cA \cong \fg \otimes \cO_Y\). Here we used that \(X\) is quasi-compact.
\end{remark}

Let \(\overline{X} \coloneqq X \times \textnormal{Spec}(\overline{\Bbbk})\) and \(\pi \colon \overline{X} \to X\) be the canonical projection, where we recall that \(\overline{\Bbbk}\) denotes the algebraic closure of \(\Bbbk\). Fix a finite-dimensional \(\Bbbk\)-algebra \(A\) and write \(\overline{A} \coloneqq A \otimes \overline{\Bbbk}\).

\begin{lemma}\label{lem:etale_locally_triviallity_wlog_over_algebraic_closure}
The sheaf \(\cA\) is \'etale \(A\)-locally free in a point \(p \in X\) if and only if \(\pi^*\cA\) is \'etale \(\overline{A}\)-locally free in all \(q \in \pi^{-1}(p)\).
\end{lemma}

\begin{proof}
\textbf{Step 1. }\emph{Setup. } The "only if" part follows directly from the fact that the property of being \'etale is stable under base change; see \cite[Proposition 17.3.3.(iii)]{egaIV_part_4}. It remains to prove the "if" part. \'Etale \(A\)-local triviality is a local property, so we can assume that \(M = \Gamma(X,\cA)\) is a free \(R\)-algebra for \(X = \textnormal{Spec}(R)\), where \(R \coloneqq \Bbbk[x_1,\dots,x_m]/I\) for some ideal \(I \subseteq \Bbbk[x_1,\dots,x_m]\). Then \(\overline{X} \cong \textnormal{Spec}(\overline{R})\) for \(\overline{R} \coloneqq \overline{\Bbbk}[x_1,\dots,x_m]/\overline{\Bbbk}I\) and the natural injective morphism \(\iota\colon R \to \overline{R}\) induces \(\pi\colon \overline{X} \to X\). By definition, \(p\) is a prime ideal of \(R\). Fix \(\overline{q} \in \pi^{-1}(p)\), i.e.\ \(\overline{q} \subset \overline{R}\) is a prime ideal such that \(\iota^{-1}(\overline{q}) = p\). The \'etale \(A\)-local triviality of \(\pi^*\cA\) in \(\overline{q}\) can now be formulated as: there exists an \'etale morphism \(\overline{f}\colon \overline{R} \to \overline{S}\) of \(\overline{\Bbbk}\)-algebras, such that \(\overline{q} = \overline{f}^{-1}(\overline{r})\) for some prime ideal \(\overline{r}\subset S\), and an isomorphism
\[\psi \colon M \otimes_R \overline{S} \cong (M\otimes_R \overline{R}) \otimes_{\overline{R}} \overline{S} \to \overline{A}\otimes_{\overline{\Bbbk}} \overline{S} \cong A \otimes \overline{S}\]
of \(\overline{S}\)-algebras. We may assume that \(\overline{S} = \overline{\Bbbk}[x_1,\dots,x_n]/(s_1,\dots,s_{k})\) for some \(s_1,\dots,s_{k}\in\overline{\Bbbk}[x_1,\dots,x_n]\). The morphism \(\overline{f}\) is completely determined by the images \(f_i \coloneqq \overline{f}(x_i + \overline{\Bbbk}I) \in \overline{S}\) of \(x_i + \overline{\Bbbk}I \in\overline{R}\) for \(i \in \{1,\dots,m\}\). We can describe \(\psi\) 
by a matrix \(\widetilde{\psi} \in \textnormal{Mat}_{d\times d}(\overline{S})\) after choosing a basis of \(A\). Here \(d = \dim(A)\). \\
\textbf{Step 2. }\emph{\(\overline{\Bbbk}\) can be replaced by some finite field extension \(\Bbbk'\) of \(\Bbbk\). } Since \(\overline{\Bbbk}\) is algebraic over \(\Bbbk\), we can choose a finite field extension \(\Bbbk'\) of \(\Bbbk\) such that \(s_1,\dots,s_{k} \in \Bbbk'[x_1,\dots,x_n]\), \(f_1,\dots,f_m,\det(\tilde{\psi})^{-1} \in S\) and \(\widetilde{\psi} \in \textnormal{Mat}_{d\times d}(S)\) for
\[R' \coloneqq \Bbbk'[x_1,\dots,x_m]/\Bbbk'I \textnormal{ and } S \coloneqq \Bbbk'[x_1,\dots,x_n]/(s_1,\dots,s_{k}).\]
Here we used that only finitely many elements of \(\overline{\Bbbk}\) appeared in these constructions. The assignment \(x_i + \Bbbk'I \mapsto f_i\) for \(i \in \{1,\dots,m\}\) defines a \(\Bbbk\)-algebra morphism \(f\colon R' \to S\) making the diagram
\begin{equation}\label{eq:diagram_etale_tiviality_wlog_algebraic_closure}
    \xymatrix{R' \ar[r]^f \ar[d]_{\iota'} & S \ar[d]^{\jmath}\\\overline{R}\ar[r]_{\overline{f}}&\overline{S}}
\end{equation}
commutative, where the vertical maps are the canonical ones. In particular, it holds that \(f^{-1}(r) = q\) for \(q \coloneqq \iota^{\prime,-1}(\overline{q})\) and \(r  \coloneqq \jmath^{-1}(\overline{r})\). \\
\textbf{Step 3. }\emph{Concluding the proof. } Since  \(\overline{f}\) can be identified with \(f\otimes_{\Bbbk'} \textnormal{id}_{\overline{\Bbbk}}\), \(f\) is \'etale since \(\overline{f}\) is; see \cite[Proposition 17.7.1.(ii)]{egaIV_part_4}. Furthermore, \(R \to R'\) is \'etale since \(\Bbbk \to \Bbbk'\) is finite, the characteristic of \(\Bbbk\) is \(0\) and \(R' \cong R \otimes \Bbbk'\). The composition \(g \colon R \to S\) of the canonical morphism \(\iota'' \colon R \to R'\) with \(f\) is thus \'etale and satisfies \(g^{-1}(r) = \iota^{\prime\prime,-1}(f^{-1}(r)) = \iota^{-1}(\overline{f}^{-1}(\overline{r})) = p\), where we used \(\iota = \iota'\iota''\) and \cref{eq:diagram_etale_tiviality_wlog_algebraic_closure}.
The matrix \(\widetilde{\psi}\) defines a morphism
\(M \otimes_R S \to A \otimes S\)
of \(S\)-algebras, which is bijective due to \(\textnormal{det}(\tilde{\psi})^{-1} \in S\). This is equivalent to the fact that \(\cA\) is \'etale \(A\)-locally free in the point \(p\).
\end{proof}

The following result can be seen as an algebro-geometric version of the result on local triviality of Lie algebra bundles in \cite{kirangi_lie_algebra_bundles}. 
    
\begin{theorem}\label{thm:etale_trivial_sheaves_of_Lie_algebras} Let \(A\) be a finite-dimensional \(\Bbbk\)-algebra and \(\cA\) be a finite locally free sheaf of algebras on a reduced \(\Bbbk\)-scheme \(X\) of finite type. Then
\(\cA\) is \'etale \(A\)-locally free if and only if \(\pi^*\cA\) is weakly \(\overline{A}\)-locally free in all closed points \(p\in X\).
\end{theorem}
\begin{proof}
\textbf{Step 1.} \emph{Setup. }
By \cref{lem:etale_locally_triviallity_wlog_over_algebraic_closure} we may assume that \(\Bbbk = \overline{\Bbbk}\) and so we have to show that \(\cA\) is \'etale \(A\)-locally free if and only if \(\cA|_p \cong A\) for all \(p \in X\) closed. The "only if" part was already discussed in \cref{rem:local_triviality_of_sheaves_of_algebras}. It remains to prove the "if" part.
Recall that an algebraic prevariety is a locally ringed space associated to the closed points of a reduced \(\Bbbk\)-scheme of finite type and that the category of reduced \(\Bbbk\)-schemes of finite type is equivalent to the category of algebraic prevarieties. We are therefore permitted to work in the latter category.

\'Etale local triviality is local. Therefore, we can assume that \(X\) is an affine variety, i.e.\ the ringed space associated to the closed points of an reduced affine \(\Bbbk\)-scheme. Let us identify \(A\) with a \(\Bbbk\)-algebra of the form \((\Bbbk^d,\mu_A)\). The \(\Gamma(X,\cO_X)\)-algebra \(\Gamma(X,\cA)\) can be identified with the \(\Gamma(X,\cO_X)\)-module of all regular maps \(X \to \Bbbk^d\) equipped with the multiplication map \(\mu_{\cA}\colon\cA \otimes_{\cO_X} \cA \to \cA\) defined by some regular map \(\theta \colon X \to M = \Hom(\Bbbk^d\otimes \Bbbk^d,\Bbbk^d)\) via \(\mu_{\cA}(a \otimes b)(p) = \theta(p)(a(p) \otimes b(p))\)
for all \(a,b\colon X \to \Bbbk^d\) regular and \(p \in X\).
By definition, \(\cA|_{p} = (\Bbbk^d,\theta(p))\) for all \(p \in X\). The group \(G = \textnormal{GL}(d,\Bbbk)\) acts on \(M\) by 
\begin{align}\label{eq:actionofG}
    &(L \cdot \vartheta)(v \otimes w) = L^{-1}\vartheta(Lv \otimes Lw) &\forall L \in G, \vartheta \in M, v,w \in \Bbbk^d.
\end{align}
The orbit \(G\cdot\mu_A\) coincides with the set of multiplications on \(\Bbbk^d\) determining an algebra structure isomorphic to \(A\). Therefore, \(\theta(X) \subseteq G\cdot \mu_A\) by assumption.\\
\textbf{Step 2.} \emph{The canonical map \(o \colon G \to G\cdot \mu_A\) is a surjective smooth morphism of algebraic prevarieties. }
Consider the stabiliser \(H\) of \(\mu_A\) in \(G\). The canonical map
\(o\colon G \to G/H \cong G\cdot \mu_A\) defined by \(L \mapsto L \cdot \mu_A\) is a faithfully flat morphism of algebraic prevarieties and the induced morphism \(G \times_{G\cdot \mu_A} G \to G \times H\) is an isomorphism; see e.g. \cite{milne_algebraic_groups}.
Note that the pull-back diagram
\begin{align}
    \xymatrix{G \times_{G\cdot \mu_A} G \cong G \times H \ar[r]\ar[d]& G\ar[d]\\G \ar[r]_o& G\cdot \mu_A},
\end{align}
combined with the fact that \(H\) and hence \(G \times H \to G\) is smooth and \(o\) is flat, implies that \(o\) is smooth; see \cite[Proposition 17.7.4]{egaIV_part_4}.  \\
\textbf{Step 3.} \emph{For all \(p \in X\) there exists an \'etale morphism \(f\colon Y \to X\) and a morphism \(s\colon Y \to G\) such that \(p \in f(Y)\) and \(o s = \theta f\). }
Consider the pull-back diagram
\begin{align}
    \xymatrix{G \times_{G\cdot \mu_A}X \ar[r]\ar[d]_g &G \ar[d]^o \\X \ar[r]_-\theta& G\cdot \mu_A},
\end{align}
The morphism \(g\) is surjective and smooth since \(o\) is; see e.g.\ \cite[Proposition 6.15.(3)]{goertz_wedhorn}.
Let \(p \in X\) be an arbitrary point.
Using the construction in \cite[Corollaire 17.16.3]{egaIV_part_4}, we see that there exists a locally closed affine subvariety \(Y \subseteq G \times_{G\cdot \mu_A}X\) such that \(f \coloneqq g|_Y\) is \'etale and \(p \in f(Y)\).  Let \(s\) be the restriction of the canonical projection \(G \times_{G\cdot \mu_A}X \to G\) to \(Y\). By construction \(o s = \theta f\).\\   
\textbf{Step 4.} \emph{\(s\) induces an isomorphism \(\psi \colon f^*\cA \to A \otimes \cO_Y\). } We can identify \(\Gamma(Y,f^*\cA)\)
with the  \(\Gamma(Y,\cO_Y)\)-module of all regular maps \(Y \to \Bbbk^d\) equipped with the multiplication \(\mu_{f^*\cA}\) determined by \(\mu_{f^*\cA}(a\otimes b)(q) = \theta(f(q))(a(q)\otimes b(q))\) for all \(a,b \colon Y \to \Bbbk^d\) regular and \(q \in Y\). Evaluating \(o s = \theta f\) in an arbitrary \(q \in Y\) results in \(s(q) \cdot \mu_A = \theta(f(q))\), hence 
\[s(q)\theta(f(q))(a(q)\otimes b(q)) = s(q)(s(q)\cdot \mu_A)(a(q)\otimes b(q)) = \mu_A(s(q)a(q)\otimes s(q)b(q))\]
for all \(a,b \colon Y \to \Bbbk^d\) regular and \(q\in Y\). This shows that the \(\Gamma(Y,\cO_Y)\)-linear automorphism \(\psi\) of \(\{a \colon Y \to \Bbbk^d\mid a \textnormal{ is regular}\}\) defined by \(\psi(a)(q) = s(q)a(q)\) for all \(a \colon Y \to \Bbbk^d\) regular and \(q \in Y\), induces an isomorphism \(\psi \colon f^*\cA \to A \otimes \cO_Y\) of sheaves of algebras.
\end{proof}

The following result can be seen as an algebro-geometric analog of \cite[Lemma 2.1]{kirangi_semi_simple}.

\begin{theorem}\label{thm:etale_triviality_due_to_semisimple_fiber}
Let \(\cA\) be a finite locally free sheaf of Lie algebras on a reduced \(\Bbbk\)-scheme \(X\) of finite type and \(\cA|_p\) be semi-simple for some closed point \(p \in X\). Then \(\cA\) is \'etale \(\cA|_p\)-locally free in \(p \in X\).
\end{theorem}
\begin{proof}
\Cref{lem:fiber_of_sheaf_Killing_form} and Cartan's criterion for semi-simplicity imply  that \(\pi^*\cA|_q\) is semi-simple for all \(q \in \pi^{-1}(p)\) if and only if \(\cA|_{p}\) is. Therefore,
using \cref{lem:etale_locally_triviallity_wlog_over_algebraic_closure}, we may assume that \(\Bbbk\) is algebraically closed. As in Step 1 in the proof of \cref{thm:etale_trivial_sheaves_of_Lie_algebras}, we can proceed to work in the category of algebraic prevarieties, assume that \(X\) is an affine variety and that \(\cA\) is free of rank \(d\).
Let \(B\subseteq M = \textnormal{Hom}(\Bbbk^d\otimes \Bbbk^d,\Bbbk^d)\) be the affine subvariety of all of possible Lie brackets on \(\Bbbk^d\).
Then \(\Gamma(X,\cA)\) can be identified with the \(\Gamma(X,\cO_X)\)-module of all regular maps \(Y \to \Bbbk^d\) equipped with the Lie bracket \(\mu_\cA\) defined by a regular map \(\theta \colon X \to B\) via \(\mu_{\cA}(a \otimes b)(q) = \theta(q)(a(q) \otimes b(q))\)
for all \(a,b\colon X \to \Bbbk^d\) regular and \(q \in X\). 

The action of \(G = \textnormal{GL}(n,\Bbbk)\) on \(M\), introduced in Step 1 of the proof of \cref{thm:etale_trivial_sheaves_of_Lie_algebras}, restricts to an action on \(B\).
In particular, the orbit \(G\cdot\theta(p)\) coincides with the set of Lie brackets on \(\Bbbk^d\) determining an Lie algebra structure isomorphic to \(\cA|_p = (\Bbbk^d,\theta(p))\). Combining \cite[Theorem 7.2]{nijenhuis_richardson} with Whitehead's Lemma and the fact that \(\cA|_p\) is semi-simple, we see that \(G\cdot \theta(p) \subseteq B\) is open and as a consequence \(U \coloneqq \theta^{-1}(G\cdot \theta(p))\) is an open neighbourhood of \(p\). For all \(q \in U\), we have \(\cA|_q \cong (\Bbbk^d,\theta(p)) = \cA|_p\), so \cref{thm:etale_trivial_sheaves_of_Lie_algebras} asserts that \(\cA|_U\) is \'etale \(A\)-locally free.
\end{proof}

\subsection{Sheaves of algebras and lattices}\label{subsec:geometry_of_lattices} In this subsection we will combine the results of \cite{ostapenko} with a version of the geometrization scheme presented in \cite{mumford,mulase} to derive a connection between certain subalgebras of Laurent series algebras and sheaves of algebras on projective curves which are formally locally free in some distinguished rational smooth point. The application of this procedure to the subalgebra associated to a formal generalized \(r\)-matrix will result in algebro-geometric properties of said \(r\)-matrix in \cref{sec:geometrization}.
\begin{definition}\label{def:lattices}
    For a  finite-dimensional \(\Bbbk\)-algebra \(A\), a subspace \(W \subseteq A(\!(z)\!)\) is called \emph{\(A\)-lattice} if it is a subalgebra of \(A(\!(z)\!)\) satisfying
\[\dim(A[\![z]\!]\cap W) \eqqcolon h_0 < \infty\textnormal{ and }\dim(A(\!(z)\!)/(A[\![z]\!] + W)) \eqqcolon h_1 < \infty.\]
The pair \((h_0,h_1)\) is called \emph{index} of \(W\).
\end{definition}

\begin{example}
Let \(r \in (\fg \otimes \fg)(\!(x)\!)[\![y]\!]\) be a formal generalized \(r\)-matrix for some finite-dimensional semi-simple Lie algebra \(\fg\) over \(\Bbbk\). The Lie algebra \(\fg(r)\) is a \(\fg\)-lattice of index \((0,0)\).
\end{example}

\begin{definition}
For a vector space \(V\) over \(\Bbbk\), a subset \(M \subseteq V(\!(z)\!)\) and an integer \(k\), we put \(M_k \coloneqq M\cap z^{-k}V[\![z]\!]\). An element \(a(z) = vz^{-k}+\dots \in M_k\setminus M_{k-1}\) is said to have \emph{order} \(k\) and \emph{main part} \(vz^{-k}\).
\end{definition}

\begin{proposition}\label{prop:geometrization_of_klattices}
Let \(O\) be an unital \(\Bbbk\)-lattice of index \((h_0,h_1)\). 
\begin{enumerate}
    \item \(O\) has Krull dimension 1 and \(h_0 = 1\), i.e.\ \(O_0 = O \cap \Bbbk[\![z]\!] = \Bbbk\).
    \item \(\textnormal{gr}(O) \coloneqq \bigoplus_{k \in \bN_0}t^kO_k \subseteq O[t]\) is a graded unital \(\Bbbk\)-subalgebra and
    \(X\coloneqq\textnormal{Proj}(\textnormal{gr}(O))\) is an integral projective curve over \(\Bbbk\) of arithmetic genus \(h_1\).
    \item There exists a \(\Bbbk\)-rational smooth point \(p\) and a  \(\Bbbk\)-algebra isomorphism \(c\colon \compO_{X,p} \to \Bbbk[\![z]\!]\) inducing an isomorphism \(\textnormal{Spec}(O) \to X\setminus\{p\}\). More precisely, the extension of \(c\) to an isomorphism \( \textnormal{Q}(\widehat{\cO}_{X,p}) \to \Bbbk(\!(z)\!)\), which will again be denoted by \(c\), satisfies  \(c(\Gamma(X\setminus\{p\},\cO_X)) = O\).
\end{enumerate}
\end{proposition}
\begin{proof}
\textbf{Step 1.} \emph{\(O_0 = \Bbbk\), 
\(\textnormal{gr}(O) \subseteq O[t]\) is a graded unital \(\Bbbk\)-subalgebra and \(X\) is integral. } That \(\textnormal{gr}(O)\) is a graded \(\Bbbk\)-subalgebra of \(O[t]\) follows directly from \(O_kO_\ell \subseteq O_{k+\ell}\) for all \(k,\ell \in \bZ\). 
The space \(O_0 = O \cap \Bbbk[\![z]\!]\) contains \(\Bbbk\) since \(O\) is unital. If \(f \in O_0\setminus\Bbbk\) would exist, \(O_0\) would contain the infinite linearly independent set \(\left\{(f(z)-f(0))^n\mid n \in \bN\right\}\), contradicting \(\dim(O_0) < \infty\). Finally, \(X\) is integral, since \(O[t]\), and thus \(\textnormal{gr}(O)\), is an integral domain.\\
\textbf{Step 2.} \emph{\(O\) has Krull dimension 1, i.e.\ (1) is true. } The condition
\(\dim(\Bbbk(\!(z)\!)/(\Bbbk[\![z]\!] + O)) < \infty\) implies that for sufficiently large \(r \in \bN\) there exist elements \(f\) and \(g\) of \(O\) with main parts \(z^{-r}\) and \(z^{-r-1}\) respectively. \(O_0 = \Bbbk\) implies that the canonical projection \(O \to \Bbbk[z^{-1}]\) is injective. Consequently, \(\Bbbk[f,g] \subseteq O\) is of finite codimension, since for every \(\ell_1 \ge r\) and \(0\le \ell_2\le r-1\) the element \(f^{\ell_1-\ell_2}g^{\ell_2}\) has main part \(z^{-\ell_1 r-\ell_2}\). Therefore, the Krull dimension of \(O\) and the Krull dimension of \(\Bbbk[f,g]\) coincide. The latter is one, since, if \(h_1,\dots, h_k\) is a basis of \(\Bbbk[f,g]_{r(r+1)-1}\), we have
    \begin{align}
       \Bbbk[f,g]_{r(r+1)-1}\ni f^{r+1}-g^r=c_1h_1+\dots+c_kh_k
    \end{align}
for some \(c_1,\dots,c_k \in \Bbbk\), which is a polynomial relation of \(f\) and \(g\).\\
\textbf{Step 3.} \emph{Construction of \(p\) and \(c\). } 
By definition of \(X = \textnormal{Proj}(\textnormal{gr}(O))\), the homogeneous prime ideal \(p \coloneqq (t) = \bigoplus_{k \in\bN_0}t^{k+1}O_k\) generated by \(t \in \textnormal{gr}(O)\) is a point of \(X\). 
Observe that \(t^kh\) is an element of the homogeneous elements \(S\) in \(\textnormal{gr}(O)\setminus (t)\) if and only if \(h\) has order \(k\). This shows
\begin{equation}\label{eq:localization_of_X_at_p}
  \cO_{X,p} = (S^{-1}\textnormal{gr}(O))_0 = \{a/h\mid a,h \in O, a/h \in \Bbbk[\![z]\!]\} = \textnormal{Q}(O)\cap \Bbbk[\![z]\!].
\end{equation}
Choosing \(f,g\) as in Step 2 yields \(u \coloneqq f/g \in \textnormal{Q}(O)\cap z\Bbbk[\![z]\!]^{\times}\). Therefore, \(\Bbbk[u] \subseteq  \textnormal{Q}(O)\cap \Bbbk[\![z]\!] \subseteq \Bbbk[\![z]\!]\) and \(\Bbbk[\![u]\!] = \Bbbk[\![z]\!]\) results in an isomorphism \(c \colon \compO_{X,p} \to \Bbbk[\![z]\!]\). We conclude that \(p\) is \(\Bbbk\)-rational and smooth.  \\
\textbf{Step 4. }\emph{The affine open subset \(\textnormal{D}_+(t) \subseteq X\) is isomorphic to \(\textnormal{Spec}(O)\) and \(c(\Gamma(\textnormal{D}_+(t),\cO_X)) = O\). }
Since \(\Bbbk[\![z]\!]_{(u)} = \Bbbk(\!(z)\!)\) for \(u\) from Step 3, we can see from \cref{eq:localization_of_X_at_p} that the rational functions on \(X\) can be identified with \(\textnormal{Q}(O)\) via the isomorphism \(c \colon \textnormal{Q}(\compO_{X,p}) \to \Bbbk(\!(z)\!)\). More precisely, we can deduce that \(c(\Gamma(\textnormal{D}_+(t^kh),\cO_X)) = \textnormal{gr}(O)[(t^kh)^{-1}]_0 \subseteq \textnormal{Q}(O)\) for all \(t^kh \in \textnormal{gr}(O)\), i.e.\ \(c\) induces the natural isomorphisms \(\Gamma(\textnormal{D}_+(t^kh),\cO_X) \cong \textnormal{gr}(O)[(t^kh)^{-1}]_0\). In particular, we see that
\begin{equation}
    c(\Gamma(\textnormal{D}_+(t),\cO_X)) = \textnormal{gr}(O)[t^{-1}]_0 = O.
\end{equation}
Therefore, \(c\) defines an isomorphism \(\textnormal{Spec}(O) \to \textnormal{D}_+(t)\).\\
\textbf{Step 5.} \emph{\(X\) is an integral projective curve over \(\Bbbk\). }
Using the Step 1,2 and 4, it remains to show that \(X\) is a \(\Bbbk\)-scheme of finite type, since then \(\textnormal{dim}(X) = \dim(\textnormal{D}_+(t)) = 1\). Thus, we have to show that \(\textnormal{gr}(O)\) is a finitely-generated \(\Bbbk\)-algebra; see e.g.\ \cite[Lemma 13.9.(2) and Proposition 13.12]{goertz_wedhorn}. We will prove that each basis \(B\) of the finite dimensional space \(\bigoplus_{k = 0}^{r^2}t^kO_k\), containing \(t,t^rf\) and \(t^{r+1}g\), generates \(\textnormal{gr}(O)\), where \(f,g\) and \(r\) are as in Step 2. Let us write \(R\) for the \(\Bbbk\)-subalgebra of \(\textnormal{gr}(O)\) generated by \(B\). We prove by induction on \(k \ge r^2\) that \(t^kO_k \subseteq R\), which is obvious for \(k = r^2\). By induction assumption \(t^{k-1}O_{k-1} \subseteq R\), and thus \(t^{k}O_{k-1} \subseteq R\). Let \(h \in O\) have main part \(az^{-k}\) and \(\ell_1\ge r, 0\le \ell_2 \le r-1\) be such that \(k = \ell_1 r + \ell_2\). Then
\[t^kh-a(t^rf)^{\ell_1-\ell_2}(t^{r+1}g)^{\ell_2} = t^kh- at^kf^{\ell_1-\ell_2}g^{\ell_2} \in t^kO_{k-1} \subseteq R,\]
proving \(t^kh \in R\) and this concludes the induction.   \\
\textbf{Step 6.} \emph{\(\textnormal{D}_+(t) = X\setminus\{p\}\) and \(\textnormal{h}^1(\cO_X) = h_1\). } Since \(O \cap \Bbbk[\![z]\!] = \Bbbk\) by Step 1, the open subscheme \(D_+(t)\cup\textnormal{D}_+\left(t^{r+1} g\right)=D_+(t) \cup \{p\}\), where \(g\) was chosen in Step 2, of \(X\) is not affine and hence a proper \(\Bbbk\)-scheme by \cite[Chapter IV,Exercise 1.4]{hartshorne}. Therefore, the embedding \(\iota(\textnormal{Spec}(O)) \cup \{p\} \to X\) is closed, so \(X = \iota(\textnormal{Spec}(O)) \cup \{p\}\), since \(X\) is integral.
For every coherent sheaf \(\mathcal{F}\) on \(X\), we have an exact sequence
\begin{align}\label{eq:adelicsequence}
    0 \longrightarrow \tH^0(\mathcal{F}) \longrightarrow \Gamma(X\setminus\{p\},\mathcal{F})\oplus \widehat{\mathcal{F}}\longrightarrow \textnormal{Q}(\widehat{\mathcal{F}})\longrightarrow\tH^1(\mathcal{F})\longrightarrow 0;
\end{align}
see e.g.\ \cite[Proposition 3]{parshin}. Setting \(\mathcal{F} = \cO_X\) and applying \(c\) results in the desired \(\textnormal{h}^1(\cO_X) = \dim(\Bbbk(\!(z)\!)/(\Bbbk[\![z]\!] + O)) = h_1\).
\end{proof}

\begin{theorem}\label{thm:geometrization_of_Alattice}
Let \(O\) be an unital \(\Bbbk\)-lattice, \(X\), \(p \) and \(c\) be the associated geometric datum from \cref{prop:geometrization_of_klattices} and \(W\) be some \(A\)-lattice of index \((h_0,h_1)\) satisfying \(OW\subseteq W\) for some finite-dimensional \(\Bbbk\)-algebra \(A\).
\begin{enumerate}
    \item \(\textnormal{gr}(W) \coloneqq \bigoplus_{k\in\bZ}t^kW_k \subseteq W[t,t^{-1}]\) is a graded \(\textnormal{gr}(O)\)-subalgebra and the associated sheaf of algebras \(\cA\) on \(X\) is coherent and torsion-free.
    \item There is a natural \(c\)-equivariant isomorphism \(\zeta\colon \widehat\cA_p \to A[\![z]\!]\) such that the induced map \(\zeta \colon \textnormal{Q}(\widehat{\cA}_p) \to A(\!(z)\!)\) satisfies \(\zeta(\Gamma(X\setminus\{p\},\cA)) = W\).
    \item \(\textnormal{h}^0(\cA) = h_0 \) and \(\textnormal{h}^1(\cA) = h_1\).
\end{enumerate} 
\end{theorem}
\begin{proof}
\textbf{Step 1.} \emph{\(\textnormal{gr}(W) \subseteq W[t,t^{-1}]\) is a graded \(\textnormal{gr}(O)\)-subalgebra and \(\cA\) is torsion-free. } The fact that \(O_kW_\ell, \mu_A(W_k \otimes W_\ell) \subseteq W_{k + \ell}\) for all \(k,\ell \in \bZ\) immediately implies that \(\textnormal{gr}(W) \subseteq W[t,t^{-1}]\) is a graded \(\textnormal{gr}(O)\)-subalgebra. Here \(\mu_A\colon A(\!(z)\!) \otimes_{\Bbbk(\!(z)\!)} A(\!(z)\!) \to A(\!(z)\!)\) denotes the multiplication map of \(A(\!(z)\!)\), which can be identified with the \(\Bbbk(\!(z)\!)\)-linear extension of the multiplication map of \(A\). It is obvious that \(\cA\) is torsion-free since \(\textnormal{gr}(W)\) is.\\
\textbf{Step 2.} \emph{\(\cA\) is coherent. } Since \(X\) is noetherian, we have to prove that \(\textnormal{gr}(W)\) is finitely-generated by \cite[Proposition 5.11.(c)]{hartshorne}. Choose \(r \in \bN\) such that for all \(k \ge r\) and for all \(v \in A\) there exists an element in \(W\) with main part \(vz^{-k}\) as well as an element in \(O\) with main part \(z^{-k}\).
Since \(W_0 = W \cap A[\![z]\!]\) is finite-dimensional, \(W_{k}\) is too, for all \(k \in \bZ\) and \(W_{-k}=\{0\}\) for \(k\) sufficiently large. Let \(B\) be any basis of the finite-dimensional vector space \(\bigoplus_{k = -\infty}^{2r}t^kW_{k} \subseteq \textnormal{gr}(W)\) and \(M\) be the \(\textnormal{gr}(O)\)-submodule of \(\textnormal{gr}(W)\) spanned by \(B\). 
From \(\Bbbk = O_0\) we see that \(t^kW_k \subseteq M\) for all \(k \le 2r\). We show \(\textnormal{gr}(W) = M\) through proving \(t^kW_k \subseteq M\) by induction on \(k \ge 2r\), which has already been verified for \(k = 2r\). By induction assumption \(t^{k-1}W_{k-1} \subseteq M\), which immediately implies that \(t^kW_{k-1} \subseteq M\) since \(t \in \textnormal{gr}(O)\). Let \(a \in W\) have main part \(v z^{-k}\). There exists \(b \in W\) with main part \(vz^{-r}\) and \(h \in O\) with main part \(z^{-r-\ell}\) for \(\ell = k-2r\). Therefore, \(t^ka - t^{r+\ell}ht^{r}a \in t^kW_{k-1}\subseteq M\). The observations \(t^rb \in M\) and \(t^{r+\ell}h \in \textnormal{gr}(O)\) show that \(t^ka \in M\), which concludes the induction.\\
\textbf{Step 3.} \emph{Construction of \(\zeta\) and proof of \(\zeta(\Gamma(X\setminus\{p\},\cA)) = W\). } The same reasoning as in Step 3 of \cref{prop:geometrization_of_klattices} yields \(\cA_p \cong \textnormal{Q}(W) \cap A[\![z]\!]\).
For each \(v \in A\) exists \(k \in \bN\) such that there is some \(a \in W\) with main part \(vz^{-k}\) and \(h \in O\) with main part \(z^{-k}\). Therefore, \(a/h \in \textnormal{Q}(W) \cap A[\![z]\!]\) satisfies \(a(0) = v\). This shows that \(\Bbbk[\![z]\!](Q(W)\cap A[\![z]\!]) = A[\![z]\!]\). Using \(c\) we get \(\zeta\) as the composition
\[\widehat{\cA}_p \cong \cA_p \otimes_{\cO_{X,p}} \compO_{X,p} \cong (\textnormal{Q}(W) \cap A[\![z]\!]) \otimes_{\textnormal{Q}(O) \cap \Bbbk[\![z]\!]} \Bbbk[\![z]\!] \cong A[\![z]\!],\]
where the last isomorphism is given by multiplication. The same arguments as in Step 4 of the proof of \cref{prop:geometrization_of_klattices} imply that \(\zeta(\Gamma(\textnormal{D}_+(t),\cA)) = \textnormal{gr}(W)[t^{-1}]_0 =  W\), so \(\textnormal{D}_+(t) = X\setminus\{p\}\) concludes the proof.\\
\textbf{Step 4. }\emph{\(\textnormal{h}^0(\cA) = h_0, \textnormal{h}^1(\cA) = h_1\). }
Applying \(\zeta\) to \cref{eq:adelicsequence} for \(\mathcal{F} = \cA\) results in \(\textnormal{h}^0(\cA) = h_0\) and \( \textnormal{h}^1(\cA) = h_1\).
\end{proof}

\begin{notation}
Let \(A\) be a finite-dimensional \(\Bbbk\)-algebra, \(O\) be a unital \(\Bbbk\)-lattice and \(W\) be some \(A\)-lattice satisfying \(OW \subseteq W\). We write \(\textnormal{GD}(O,W) \coloneqq ((X,\cA),(p,c,\zeta))\) for the geometric datum associated to the pair \((O,W)\) by virtue of \cref{thm:geometrization_of_Alattice}.
\end{notation}

\begin{remark}\label{rem:glattice_yields_sheaf_of_Lie_algebras}
Let \(A\) be a finite-dimensional \(\Bbbk\)-algebra, \(O\) be an unital \(\Bbbk\)-lattice, \(W\) be some \(A\)-lattice satisfying \(OW \subseteq W\) and \(\textnormal{GD}(O,W) = ((X,\cA),(p,c,\zeta))\).
\begin{enumerate}
    \item If \(A\) is a Lie algebra, \(\cA\) is a sheaf of Lie algebras, since the sections of \(\cA\) can be identified with subalgebras of \(A(\!(z)\!)\) using \(\zeta\).
    \item \(\cA\) is formally \(A\)-locally free in \(p\), hence weakly \(A\)-locally free in \(p\). Therefore, \cref{thm:etale_triviality_due_to_semisimple_fiber} implies that \(\cA\) is \'etale \(A\)-locally free in \(p\) if \(A\) is a semi-simple Lie algebra.
\end{enumerate}
\end{remark}

\begin{remark}\label{rem:Alattice_from_geometry}
The construction of geometric data from lattices can be inverted: let \(A\) be a finite-dimensional \(\Bbbk\)-algebra and \(\cA\) be a torsion-free coherent sheaf of algebras over a projective curve \(X\) over \(\Bbbk\) equipped with a \(\Bbbk\)-rational point \(p\) and a formal trivialization \(c \colon \compO_{X,p} \to \Bbbk[\![z]\!]\) as well as a \(c\)-equivariant \(\Bbbk\)-algebra isomorphism \(\zeta \colon \compA_p \to A[\![z]\!]\). Then the exact sequence \cref{eq:adelicsequence} for \(\mathcal{F} \in \{\cO_X,\cA\}\) implies that \(O\coloneqq c(\Gamma(X\setminus\{p\},\cO_X)\) is a \(\Bbbk\)-lattice and \(W \coloneqq \zeta(\Gamma(X\setminus\{p\},\cA)\) is an \(A\)-lattice such that \(O W \subseteq W\).

Note that if \(A\) is \emph{analytically rigid} as algebra, i.e.\ there exists no non-trivial formal deformation of \(A\), then \(\cA|_p\cong A\) is already sufficient for the existence of an isomorphism \(\zeta\) for an arbitrary choice of \(c\). A basic example of an analytically rigid algebra is a semi-simple Lie algebra. This follows from the vanishing of the second Lie algebra cohomology with values in \(A\) due to Whitehead's Lemma; see \cite{hazewinkel_gerstenhaber}.
\end{remark}

The problem of the existence of a \(\Bbbk\)-lattice stabilizing some \(A\)-lattice is considered in \cite{ostapenko} for a finite-dimensional simple \(\Bbbk\)-algebra \(A\) in the case that \(\Bbbk\) is algebraically closed. However, the methods actually apply to simple Lie algebras over arbitrary fields of characteristic 0 under an additional assumption.

\begin{definition}\label{def:central_algebra}
Let \(A\) be a finite-dimensional \(\Bbbk\)-algebra and \(C\) be the centralizer of the subspace of \(\textnormal{End}(A)\) generated by all left and right multiplication maps of \(A\), i.e.\ \(L(ab) = aL(b) = L(a)b\) for all \(L \in C, a,b \in A\). \(C\) is called \emph{centroid} of \(A\) and \(A\) is said to be \emph{central} if \(C = \Bbbk \textnormal{id}_A\).
\end{definition}

\begin{remark}\label{rem:simple_is_central_over_algebraically_closed_field}
Every finite-dimensional simple \(\Bbbk\)-algebra is central if \(\Bbbk\) is algebraically closed. Indeed, \cite[Chapter X,Theorem 1]{jacobson} tells us that the centroid of a  finite-dimensional simple \(\Bbbk\)-algebra is a field and as a consequence a finite field extension of \(\Bbbk\). 
\end{remark}

\begin{theorem}[\cite{ostapenko}]\label{thm:multipliers_are_klattice}
Let \(A\) be a finite-dimensional, central, simple \(\Bbbk\)-algebra and \(W\) be some \(A\)-lattice. The unital \(\Bbbk\)-algebra 
\(\textnormal{Mult}(W)\coloneqq \{f \in \Bbbk(\!(z)\!)\mid fW \subseteq W\}\) 
is a \(\Bbbk\)-lattice.
\end{theorem}
\begin{proof}
\textbf{Step 1.} \emph{Setup. }
Let us denote by \(r_a\) (resp. \(\ell_a\)) 
the right (resp. left) multiplication by an element \(a \in A(\!(z)\!)\) considered as an element in \(\textnormal{End}_{\Bbbk(\!(z)\!)}(A(\!(z)\!)) \cong \textnormal{End}(A)(\!(z)\!)\), i.e.\ \(\ell_a(b) = ab = r_b(a)\) for all \(a,b \in A(\!(z)\!)\). Note that \(r_a,\ell_a \in \End(A)\) for \(a \in A\). The fact that \(A\) is central combined with \cite[Chapter X,Theorem 4.]{jacobson} implies that the subalgebra of \(\textnormal{End}(A)\) generated by \(\{r_a,\ell_a\mid a\in A\}\) equals \(\textnormal{End}(A)\). Let \(J \subseteq \textnormal{End}(A)(\!(z)\!)\) be the subalgebra generated by \(\{r_{a},\ell_{a}\mid a\in W\}\).\\
\textbf{Step 2.} \emph{For every \(L \in \End(A)\) there exists \(r(L) \in \bN_0\) such that for all \(k\in \bN_0\) there is an element of \(J\) with main part \(z^{-k-r(L)}L\). }
There exists a non-commutative polynomial \(f = f(x_1,\dots,x_q)\) and \(a_1,\dots,a_q\in A\) such that \(f(m_1,\dots,m_q) = L\), where \(m_i \in \{r_{a_i},\ell_{a_i}\}\) for all \(i\in\{1,\dots,q\}\). Let \(\{f_i\}_{i = 1}^q\) be the unique polynomials defined by the following inductive process: \(f_q\) is the sum of all monomials of \(f\) depending on \(x_q\) and, if \(f_{i+1},\dots,f_q\) are given, \(f_i\) is the sum of all monomials of \(f-f_{i+1}-\dots -f_q\) which depend on \(x_i\). By construction, \(f_i\) depends only on \(x_1,\dots,x_i\) and every monomial of \(f_i\) contains a factor \(x_i\). Let  \(f_{ij}\) denote the homogeneous component of \(f_i\) of degree \(j\) and \(g_{ij}(x_1,\dots,x_i;y_i)\)
be the polynomial where the left most \(x_i\) appearing in every monomial of \(f_{ij}\) is changed to \(y_i\). Since \(W\) is an \(A\)-lattice, we can chose \(s \in \bN\) such that for all \(i\in\{1,\dots,q\}\) and \(k \in \bN_0\) there exists \(b_i^k\) in \(W\) with main part \(z^{-s-k}a_i\). Let \(\widetilde{m}^k_i\) be the left (resp. right) multiplication by \(b_i^k\) if \(m_i\) is the left (resp. right) multiplication by \(a_i\). By construction
\(g_{ij}\left(\widetilde{m}_1^0,\dots,\widetilde{m}_i^0;\widetilde{m}_i^{k}\right)\) has main part \(z^{-k-sj}f_{ij}(m_1,\dots,m_i)\). Thus,
\begin{equation}
    \sum_{i = 1}^q \sum_{j = 1}^{\textnormal{deg}(f_i)}g_{ij}\left(\widetilde{m}_1^0,\dots,\widetilde{m}_i^0;\widetilde{m}_i^{k+s(\deg(f)-j)}\right)
\end{equation}
has main part \(z^{-k-s\deg(f)}L\) and putting \(r(L) \coloneqq s\deg(f)\) proves the statement.\\
\textbf{Step 3.} \emph{There exists \(r \in \bN_0\) such that for all \(L \in \End(A)\) and \(k \in \bN_0\) there is an element of \(J\) with main part \(z^{-r-k}L\). }
Let \(\{L_i\}_{i = 1}^n\) be a basis of \(\End(A)\). Using Step 2, the result follows by choosing \(r \coloneqq \textnormal{max}\{r(L_i)\mid i \in \{1,\dots,q\}\}\). \\
\textbf{Step 4. }\emph{\(\textnormal{Mult}(W)\) is a \(\Bbbk\)-lattice. }
In \cite{amitsur} the author constructs a non-commutative homogeneous polynomial \(P = P(x_1,\dots,x_q)\) in \(q \coloneqq 2\textnormal{dim}(A)^2\)-variables which is linear in each variable and non-vanishing with values in \(\Bbbk \textnormal{id}_A\) if evaluated on \(\textnormal{End}(A)\). In particular, we may choose \(L_1,\dots,L_q\in \textnormal{End}(A)\) such that \(P(L_1,\dots,L_q) = \textnormal{id}_A\). Let \(r\in\bN\) be the integer from Step 3, i.e.\ for all \(k \in \bN_0\) and \(i \in \{1,\dots,q\}\) we may chose  \(\widetilde{L}^k_i \in J\) with main part \(z^{-r-k}L_i\). We have
\begin{equation}\label{eq:constructing_multipliers}
    P\left(\widetilde{L}^0_1,\dots,\widetilde{L}^0_{q-1},\widetilde{L}^{\ell-rq}_q\right) = f_\ell\textnormal{id}_A
\end{equation}
for all \(\ell \ge rq\), where \(f_\ell \in \Bbbk(\!(z)\!)\) has main part \(z^{-\ell}\). Since the left-hand side of \cref{eq:constructing_multipliers} is an element of \(J\), we can see that \(f_\ell \in \textnormal{Mult}(W)\). This concludes the proof.
\end{proof}

\section{Algebraic geometry of formal generalized \(r\)-matrices}\label{sec:geometrization}

\subsection{Geometrization of generalized \(r\)-matrices}\label{subsec:geometrization_of_generalized_rmatrices}
Let \(\fg\) be a finite-dimensional, central, simple Lie algebra over \(\Bbbk\) and \(r \in (\fg \otimes \fg)(\!(x)\!)[\![y]\!]\) be a formal generalized \(r\)-matrix. The algebra \(\fg(r)\) is a \(\fg\)-lattice of index \((0,0)\) by \cref{prop:Lie_subalgebra_of_generalized_rmatrix}, so \Cref{thm:multipliers_are_klattice} implies that every unital subalgebra \(O\) of \(\textnormal{Mult}(\fg(r)) = \{f\in \Bbbk(\!(z)\!)\mid f\fg(r) \subseteq \fg(r)\}\) of finite codimension is a \(\Bbbk\)-lattice satisfying \(O\fg(r) \subseteq \fg(r)\). Therefore, \cref{thm:geometrization_of_Alattice} provides a geometric datum \(\textnormal{GD}(O,\fg(r))= ((X,\cA),(p,c,\zeta))\), where:
\begin{itemize}
    \item \(X\) is an integral projective curve over \(\Bbbk\) of arithmetic genus \(\textnormal{dim}(\Bbbk(\!(z)\!)/(\Bbbk[\![z]\!] + O))\) that comes equipped with a \(\Bbbk\)-algebra isomorphism \(c\colon \compO_{X,p} \to \Bbbk[\![z]\!]\) at a \(\Bbbk\)-rational smooth point \(p\in X\), which induces an isomorphism \(\textnormal{Spec}(O) \to X\setminus \{p\}\);
    \item \(\cA\) is a coherent sheaf of Lie algebras on \(X\) which satisfies \(\textnormal{h}^0(\cA) = 0 = \textnormal{h}^1(\cA)\), is
    \'etale \(\fg\)-locally free in \(p\) (see \cref{rem:glattice_yields_sheaf_of_Lie_algebras}) and admits a natural \(c\)-equivariant Lie algebra isomorphism \(\zeta \colon \widehat\cA_p \to \fg[\![z]\!]\) such that \(\zeta(\Gamma(X\setminus\{p\},\cA)) =\fg(r)\).
\end{itemize}

\begin{example}\label{ex:generalized_rmatrices_over_P^1}
Assume \(X \cong \mathbb{P}^1_\Bbbk\), then \(O = \textnormal{Mult}(\fg(r))\), since \(O\) is integrally closed, and it holds that \(\Bbbk(\!(z)\!) = \Bbbk[\![z]\!] + O\). This and \(O \cap \Bbbk[\![z]\!] = \Bbbk\) can be used to deduce that \(O = \Bbbk[u^{-1}]\) for an arbitrary \(u\in z\Bbbk[\![z]\!]^\times\) such that \(u^{-1} \in O\). Let \(w \in z\Bbbk[\![z]\!]^\times\) be the compositional inverse of \(u\), i.e.\ \(w(u(z)) = z\). Then \(\widetilde{r}(x,y) \coloneqq r(w(x),w(y))\) satisfies \(\textnormal{Mult}(\fg(\widetilde{r})) = \Bbbk[z^{-1}]\). In other words, \(\fg(\widetilde{r})\) is homogeneous in the sense of  \cref{ex:homogeneous_Lie_subalgebras}. In particular, the formal generalized \(r\)-matrices \(r\) for which \(X \cong \mathbb{P}^1_\Bbbk\) are exactly those equivalent to the ones described in \cref{ex:homogeneous_Lie_subalgebras}.
\end{example}

\begin{remark}\label{rem:field_extension_geometry}
Let \(\Bbbk'\) be a field extension of \(\Bbbk\) and let \(r_{\Bbbk'} \in (\fg_{\Bbbk'}\otimes_{\Bbbk'}\fg_{\Bbbk'})(\!(x)\!)[\![y]\!]\) be the generalized \(r\)-matrix given in \cref{rem:field_extensions}, where we note that \(\fg_{\Bbbk'} \coloneqq \fg \otimes_{\Bbbk}\Bbbk'\) is a central simple Lie algebra over \(\Bbbk'\) by \cite[Chapter X, Theorem 3]{jacobson}. If \(O_{\Bbbk'}\) denotes the image of \(O\otimes \Bbbk'\) under the multiplication map \(\Bbbk(\!(z)\!)\otimes \Bbbk' \to \Bbbk'(\!(z)\!)\), the geometric datum \(\textnormal{GD}(O_{\Bbbk'},\fg(r_{\Bbbk'})) = ((X_{\Bbbk'},\cA),(p_{\Bbbk'},c_{\Bbbk'},\zeta_{\Bbbk'}))\) satisfies the following compatibilities:
\begin{itemize}
    \item \(O_{\Bbbk'} \cong O \otimes \Bbbk'\) induces an isomorphism \(X_{\Bbbk'} \cong X \times \textnormal{Spec}(\Bbbk')\) such that the canonical map \(O \to O_{\Bbbk'}\) is compatible with the canonical morphism \(\pi \colon X_{\Bbbk'} \to X\). Furthermore, \(\pi(p_{\Bbbk'}) = p\) and the following diagram commutes 
    \[\xymatrix{\compO_{X,p} \ar[r]^{c}\ar[d]_{\widehat{\pi}^\sharp_p}& \Bbbk[\![z]\!]\ar[d] \\ \compO_{X_{\Bbbk'},p_{\Bbbk'}} \ar[r]_{c_{\Bbbk'}} & \Bbbk'[\![z]\!]}.\]
    \item The multiplication map \(\fg(r) \otimes \Bbbk' \cong \fg_{\Bbbk'}(r_{\Bbbk'})\) induces an isomorphism \(\pi^*\cA \cong \cA_{\Bbbk'}\) of sheaves of Lie algebras such that the following diagram commutes
    \[\xymatrix{\compA_{p} \ar[r]^{\zeta}\ar[d]& \fg[\![z]\!]\ar[d] \\ \compA_{\Bbbk',p_{\Bbbk'}} \ar[r]_{\zeta_{\Bbbk'}} & \fg_{\Bbbk'}[\![z]\!]}.\]
\end{itemize}
In particular, if \(\Bbbk \to \Bbbk'\) is Galois with Galois group \(G\), then \(G\) acts on \(X_{\Bbbk'}\) by automorphisms of \(X\)-schemes and on \(\cA_{\Bbbk'}\) by \(\cO_X\)-linear automorphisms of sheaves of Lie algebras in such a way that \(X\) and \(\cA\) are the respective fixed objects. 
\end{remark}

\begin{lemma}\label{lem:geometry_of_equivalences}
For \(i \in \{1,2\}\), let \(r_i \in (\fg \otimes \fg)(\!(x)\!)[\![y]\!]\) be a generalized \(r\)-matrix with associated geometric datum \(\textnormal{GD}(\textnormal{Mult}(\fg(r_i)),\fg(r_i)) = ((X_i,\cA_i),(p_i,c_i,\zeta_i))\). The series \(r_1\) and \(r_2\) are equivalent if and only if there exists an isomorphism \(f\colon X_2 \to X_1\) mapping \(p_2\) to \(p_1\) as well as an isomorphism \(\phi\colon \cA_1\to f_*\cA_2\) of sheaves of Lie algebras.
\end{lemma}
\begin{proof}
\("\implies "\) Let \((\lambda,w,\varphi)\) be an equivalence of \(r_1\) to \(r_2\). The assignment \(p(z) \mapsto p_w(z) \coloneqq p(w(z))\) defines an isomorphism \(O_1 \coloneqq \textnormal{Mult}(\fg(r_1)) \to \textnormal{Mult}(\fg(r_2)) \eqqcolon O_2\) resulting in a graded isomorphism 
\[O_1[t] \supseteq \textnormal{gr}(O_1) \stackrel{f^\flat}\longrightarrow \textnormal{gr}(O_2) \subseteq O_2[t]\]
which maps \(t\) to \(t\). We can see from e.g.\ \cite[Section 13.2]{goertz_wedhorn} that this induces an isomorphism \(f \colon X_2 \to X_1\). This isomorphism is seen to satisfy \(f(p_2) = p_1\) after recalling the construction of \(p_1\) and \(p_2\) in Step 3 of the proof of \cref{prop:geometrization_of_klattices}.
\Cref{lem:equivalence_on_subalgebra_level} provides an isomorphism \(\fg(r_1) \to \fg(r_2)\) defined by \(a(z) \mapsto \varphi(z)a(w(z))\). This induces a graded \(f^\flat\)-equivariant isomorphism \(\textnormal{gr}(\fg(r_1)) \to \textnormal{gr}(\fg(r_2))\). The process of associating a quasi-coherent sheaf to a graded module is functorial (see e.g.\ \cite[Section 13.4]{goertz_wedhorn}) and the sheaf associated to \(\textnormal{gr}(\fg(r_2))\) equipped with the \(\textnormal{gr}(O_1)\)-module structure induced by \(f^\flat\) is exactly \(f_*\cA_2\). Thus, we obtain an isomorphism \(\phi\colon \cA_1\to f_*\cA_2\).\\
\("\Longleftarrow"\) Let \(w\) be the image of \(z\) under 
\[\Bbbk[\![z]\!] \stackrel{c_1^{-1}}{\longrightarrow} \compO_{X_1,p_1} \stackrel{\widehat{f}^\sharp_p}\longrightarrow \compO_{X_2,p_2}\stackrel{c_2}{\longrightarrow} \Bbbk[\![z]\!].\]
Then the \(\Bbbk\)-linear isomorphism
\[\fg[\![z]\!] \stackrel{\zeta_1^{-1}}{\longrightarrow} \compA_{1,p_1} \stackrel{\widehat{\phi}_{p_1}}\longrightarrow \compA_{2,p_2}\stackrel{\zeta_2}{\longrightarrow} \fg[\![z]\!],\]
where we implicitly used \(\big(\widehat{f_*\cA_2}\big)_{p_1} \cong \compA_{2,p_2}\), takes the form \(a(z)\mapsto \varphi(z)a(w(z))\) for some \(\varphi \in \Aut_{\Bbbk[\![z]\!]\textnormal{-alg}}(\fg)\). \cref{lem:equivalence_on_subalgebra_level} and the calculation
\[\zeta_2\widehat{\phi}_{p_1}\zeta_1^{-1}(\fg(r_1)) = \zeta_2\widehat{\phi}_{p_1}(\Gamma(X_1\setminus\{p_1\},\cA_1)) = \zeta_2(\Gamma(X_2\setminus\{p_2\},\cA_2)) = \fg(r_2)\] 
verify that \(r_1\) is equivalent to \(r_2\).
\end{proof}

\begin{remark}
A special case of \cref{lem:geometry_of_equivalences} was used in \cite{abedin_stepan_classical_twists} to classify twists of the standard bialgebra structure on twisted loop algebras by reducing it to the classification of trigonometric \(r\)-matrices given in \cite{belavin_drinfeld_solutions_of_CYBE_paper}.
\end{remark}

\begin{lemma}\label{lem:normalization_of_rmatrix_geometry}
Let \(\widetilde{O}\) be the integral closure of \(O\) and put \(\widetilde{\fg(r)} \coloneqq \widetilde{O}\fg(r)\). Then \(\widetilde{O}\) is naturally a \(\Bbbk\)-lattice and  \(\widetilde{\fg(r)}\) is a \(\fg\)-lattice. Let \(\textnormal{GD}(\widetilde{O},\widetilde{\fg(r)}) = ((\widetilde{X},\widetilde{\cA}),(\widetilde{p},\widetilde{c},\widetilde{\zeta}))\) be the geometric datum associated to \(\widetilde{O}\) and \(\widetilde{\fg(r)}\) by virtue of \cref{thm:geometrization_of_Alattice}.
\begin{enumerate}
    \item The canonical map \(O \subseteq \widetilde{O}\) induces a morphism \(\nu \colon \widetilde{X} \to X\) such that \(\nu(\widetilde{p}) = p\) and \(\nu\) is the normalization of \(X\).
    \item The canonical inclusion \(\fg(r) \subseteq \widetilde{\fg(r)}\) induces an injective morphism \(\iota \colon \cA \to \nu_*\widetilde{\cA}\) of sheaves of Lie algebras having a torsion cokernel.
    \item The morphisms \(\nu\) and \(\iota\) induce isomorphisms \(\cO_{X,p} \cong \cO_{\widetilde{X},p}\) and \(\cA_p \cong \widetilde{\cA}_{\widetilde{p}}\) identifying \(\widetilde{c}\) and \(\widetilde{\zeta}\) with \(c\) and \(\zeta\).
\end{enumerate}  
\end{lemma}
\begin{proof}
We have \(\widetilde{O}\subseteq \textnormal{Q}(\widetilde{O}) = \textnormal{Q}(O)  \subseteq  \Bbbk(\!(z)\!)\) is a unital subalgebra and it is well-known that \(O \to \widetilde{O}\) is finite. Thus, \(\widetilde{O}\) is a \(\Bbbk\)-lattice and \( \widetilde{\fg(r)}\) is a \(\fg\)-lattice. Let \(\textnormal{GD}(\widetilde{O},\widetilde{\fg(r)}) = ((\widetilde{X},\widetilde{\cA}),(\widetilde{p},\widetilde{c},\widetilde{\zeta}))\). Then \(\cO_{X,p} \cong \textnormal{Q}(O)\cap \Bbbk[\![z]\!] = \textnormal{Q}(\widetilde{O}) \cap \Bbbk[\![z]\!] \cong \cO_{\widetilde{X},\widetilde{p}}\) is compatible with \(O \to \widetilde{O}\), resulting in the morphism \(\nu \colon \widetilde{X} \to X\) mapping \(\widetilde{p}\) to \(p\).  The morphism \(\nu\) is finite and \(\widetilde{X}\) is smooth, so \(\nu \colon \widetilde{X} \to X\) is the normalization of \(X\). Now \(\fg(r) \to \widetilde{\fg(r)}\) induces an injective graded morphism \(\textnormal{gr}(\fg(r)) \to \textnormal{gr}\big(\widetilde{\fg(r)}\big)\) which results in the injective morphism \(\iota \colon\cA \to \nu_*\widetilde{\cA}\). Clearly, \(\iota\) is an isomorphism locally around \(p\), thus we see that the cokernel of \(\iota\) is a torsion sheaf. The canonical maps \(\nu^\sharp_p\colon \cO_{X,p} \to \cO_{\widetilde{X},p}\) and \(\iota_p\colon \cA_p \to \big(\nu_*\cA\big)_{\widetilde p} \cong \widetilde{\cA}_{\widetilde{p}}\) are isomorphisms and are easily seen to identify \(\widetilde{c}\) and \(\widetilde{\zeta}\) with \(c\) and \(\zeta\).
\end{proof}

\begin{theorem}\label{prop:geometry_of_generalized_rmatrices}
The following results are true.
\begin{enumerate}
    \item The genus \(\widetilde{g} \coloneqq \textnormal{h}^1(\cO_{\widetilde{X}})\) of \(\widetilde{X}\) is at most one.
    \item Let \(\Bbbk(\!(z)\!)\) be equipped with the \(\Bbbk\)-bilinear form defined by \((f,g)\mapsto \textnormal{res}_0fg\textnormal{d}z\). Then 
    \[\widetilde{O}^\bot\fg(r)\subseteq \widetilde{O}^\bot\widetilde{\fg(r)} \subseteq \widetilde{\fg(r)}^\bot \subseteq \fg(r)^\bot.\]
    \item If \(\widetilde{g}=1\) and \(O = \textnormal{Mult}(\fg(r))\), we have: \(\fg(r) = \widetilde{\fg(r)}\), \(X = \widetilde{X}\) and the Killing form of \(\widetilde\cA\) (see \cref{def:sheaf_Killing_form}) is perfect. Furthermore, there exists \(\lambda \in \Bbbk[\![z]\!]\) and \(w \in z\Bbbk[\![z]\!]^\times\) such that \(\widetilde{r}(x,y) = ´\lambda(y)r(w(x),w(y))\) is normalized and skew-symmetric. In particular, \(r\) is equivalent to the normalized formal \(r\)-matrix \(\widetilde{r}\).
\end{enumerate}
\end{theorem}
\begin{proof}
\textbf{Step 1.} \emph{\(\textnormal{h}^1(\widetilde{\cA}) = 0 = \textnormal{h}^1(\widetilde{\cA}^*)\). } 
Using \(\textnormal{h}^1(\textnormal{Cok}(\iota)) = 0 = \textnormal{h}^1(\cA)\) in the long exact sequence in cohomology of
\begin{equation}\label{eq:normalizationshortexactseq}
0 \longrightarrow \cA \stackrel{\iota}\longrightarrow \nu_*\widetilde\cA \longrightarrow \textnormal{Cok}(\iota)\longrightarrow 0,
\end{equation}
implies that \(\textnormal{h}^1(\widetilde\cA) = 0\). Here we used \(\iota\) from \cref{lem:normalization_of_rmatrix_geometry} and the fact that \(\textnormal{Cok}(\iota)\) is a torsion sheaf.
Let \(K^\textnormal{a}\colon \widetilde\cA \to \widetilde\cA^*\) be the canonical \(\cO_{\widetilde X}\)-linear morphism induced by the Killing form \(K\) of \(\widetilde{\cA}\). \Cref{lem:fiber_of_sheaf_Killing_form} implies that \(K|_p\) is the Killing form of \(\widetilde\cA|_p \cong \cA|_p \cong \fg\). In particular, \(K|_p\) is non-degenerate, so \(K^\textnormal{a}|_p\) is an isomorphism. Therefore, \(\textnormal{Ker}(K^\textnormal{a})\) and \(\textnormal{Cok}(K^\textnormal{a})\) are  torsion sheaves. Since \(\widetilde\cA\) is locally free, this forces \(\textnormal{Ker}(K^\textnormal{a})\) to vanish. Using \(\textnormal{h}^1(\textnormal{Cok}(K^\textnormal{a})) = 0 = \textnormal{h}^1(\widetilde{\cA})\) in the long exact sequence in cohomology of
\begin{equation}\label{eq:tildeKexactseq}
0 \longrightarrow \widetilde\cA \longrightarrow \widetilde\cA^* \longrightarrow \textnormal{Cok}(K^\textnormal{a})\longrightarrow 0
\end{equation}
results in \(\textnormal{h}^1(\widetilde\cA^*) = 0\). \\
\textbf{Step 2.} \emph{\(\widetilde{g}\le 1\), i.e.\ (1) holds. } The Riemann-Roch theorem for \(\widetilde\cA\) and \(\widetilde\cA^*\) (e.g.\ in the version of \cite[Chapter 7, Exercise 3.3]{liu}) reads
\begin{align*}
    &0 \le \textnormal{h}^0(\widetilde{\cA})-\textnormal{h}^1(\widetilde{\cA}) = \textnormal{deg}(\textnormal{det}(\widetilde{\cA})) + (1-\widetilde g)\textnormal{rank}(\widetilde{\cA})\\&0 \le \textnormal{h}^0(\widetilde{\cA}^*)-\textnormal{h}^1(\widetilde{\cA}^*) = -\textnormal{deg}(\textnormal{det}(\widetilde{\cA})) + (1-\widetilde g)\textnormal{rank}(\widetilde{\cA}),
\end{align*}
where we used that \(\textnormal{det}(\widetilde{\cA}^*) = \textnormal{det}(\widetilde{\cA})^*\) implies \(\textnormal{deg}(\textnormal{det}(\widetilde{\cA}^*)) = -\textnormal{deg}(\textnormal{det}(\widetilde{\cA}))\). We conclude \(\widetilde g \le 1\).\\
\textbf{Step 3.} \emph{The statement in (2) holds. } The diagram
\begin{align}\label{eq:kappaKdiagram}
    \xymatrix{\Gamma(U,\widetilde\cA) \times \Gamma(U,\widetilde\cA) \ar[r]^-{K_U}\ar[d]_{\widetilde{\zeta} \times \widetilde{\zeta}}&\Gamma(U,\cO_{\widetilde{X}})\ar[d]^{\widetilde{c}}\\\fg(\!(z)\!)\times\fg(\!(z)\!)\ar[r]_-\kappa&\Bbbk(\!(z)\!)}
\end{align}
commutes for all affine open \(U \subseteq X\) such that \(\widetilde\cA|_U\) is free and consequently for all \(U\subseteq X\) open by a gluing argument. In particular, \(\kappa(a,b) \in \widetilde{c}(\Gamma(\widetilde{X}\setminus\{\widetilde{p}\},\cO_{\widetilde{X}})) = \widetilde{O}\) for all \(a,b \in \widetilde{\fg(r)}\). Hence, we see that 
\[\kappa_0(fa,b) = \textnormal{res}_0\kappa(fa,b)\textnormal{d}z = \textnormal{res}_0f\kappa(a,b)\textnormal{d}z = 0\]
for all \(f \in \widetilde{O}^\bot\). Therefore, \(fa \in \widetilde{\fg(r)}^\bot\) and we can complete the chain of inclusions by observing that \(\fg(r) \subseteq \widetilde{\fg(r)}\) implies \(\widetilde{\fg(r)}^\bot \subseteq \fg(r)\).\\
\textbf{Step 4.} \emph{If \(\widetilde{g} = 1\), we have  \(\textnormal{Mult}(\fg(r)) = \widetilde{O}\) and \(\textnormal{GD}(\widetilde{O},\fg(r)) =  ((\widetilde{X},\widetilde\cA),(\widetilde{p},\widetilde{c},\widetilde{\zeta}))\). } Let \(\Omega^1_{\widetilde{X}}\) be the sheaf of regular 1-forms on \(\widetilde{X}\). We have \(\textnormal{H}^1(\Omega^1_{\widetilde{X}}) = \Bbbk \cdot \eta\) for some global 1-form \(\eta\) on \(\widetilde{X}\) and this choice defines an isomorphism \(\Omega^1_{\widetilde{X}} \cong \cO_{\widetilde{X}}\). Serre duality (e.g.\ in the version \cite[Chapter 6, Remark 4.20 and Theorem 4.32]{liu}) and Step 1 provide that \(0 = \textnormal{h}^1(\widetilde{\cA}^*) = \textnormal{h}^0(\widetilde{\cA})\). Combined with the fact that \(\textnormal{h}^1(\cA) = 0\) and \(\textnormal{Cok}(\iota)\) is torsion, the long exact sequence of \cref{eq:normalizationshortexactseq} in cohomology implies that \(\iota \colon \cA \to \nu_*\widetilde{\cA}\) is an isomorphism. In particular, \(\widetilde{\fg(r)} = \fg(r)\), \(\textnormal{Mult}(\fg(r)) = \widetilde O\) and \(X = \widetilde{X}\) in case that \(O = \textnormal{Mult}(\fg(r))\).\\
\textbf{Step 5.} \emph{Concluding the proof of (3). }
Let \(\textnormal{d}u(z) = u'(z)\textnormal{d}z = c^*(\eta)\) for some global 1-form \(\eta\) on \(X\) (see \cref{rem:c*_construction_for_1_forms} below) and \(a,b \in \Gamma(X\setminus\{p\},\cA)\). If \(w \in z\Bbbk[\![z]\!]^\times\) is the series uniquely determined by \(u(w(z)) = z\) and \(\widetilde{a} \coloneqq \zeta(a),\widetilde{b} \coloneqq \zeta(b)\), we may calculate
\begin{align*}
    &\kappa_0\left(\widetilde{a}(w(z)),\widetilde{b}(w(z))\right) = \textnormal{res}_0 \kappa\left(\widetilde{a}(w(z)),\widetilde{b}(w(z))\right)\textnormal{d}z\\&= \textnormal{res}_0 \kappa\left(\widetilde{a}(z),\widetilde{b}(z)\right)\textnormal{d}u(z) = \textnormal{res}_p K(a,b)\eta = 0.
\end{align*}
Here the second to last equality uses \cref{eq:kappaKdiagram} and \cref{rem:c*_construction_for_1_forms} below while the last equality is due to the residue theorem \cite[Corollary of Theorem 3]{tate_residue_paper} under consideration of \(K(a,b)\eta \in \Gamma(X\setminus\{p\},\Omega_X)\). Thus, the image \(W\) of \(\fg(r)\) under \(a(z) \mapsto a(w(z))\) satisfies \(W^\bot = W\). \Cref{lem:equivalence_on_subalgebra_level} states that  \(W = \fg(\widetilde{r})\) for \(\widetilde{r}(x,y) = \lambda(y)r(w(x),w(y))\), where \(\lambda \in\Bbbk[\![z]\!]\) is arbitrary. This shows that \(\widetilde{r}\) is a formal \(r\)-matrix if we chose \(\lambda\) in such a way that \(\widetilde{r}\) is normalized; see \cref{lem:orthogonal_complement_of_g(r)}. That \(K\) is indeed perfect follows from \(\textnormal{Cok}(K^\textnormal{a}) = 0\), which is a consequence of using \(\textnormal{h}^1(\widetilde{\cA}) = \textnormal{h}^0(\widetilde{\cA}^*) = 0\) in the long exact sequence of \cref{eq:tildeKexactseq} in cohomology. Here we used Serre duality again.
\end{proof}
\begin{remark}\label{rem:c*_construction_for_1_forms}
Let \(\omega_X\) be the dualizing sheaf of \(X\). Since \(p\) is smooth, \(\omega_{X,p}\) can be identified with the K\"ahler differentials \(\Omega^1_{\cO_{X,p}/\Bbbk}\). Therefore, using e.g.\ \cite[Corollary 12.5 and Example 12.7]{kunz_kaehler_differentials}, we obtain a \(c\)-equivariant isomorphism
\(c^*\colon \widehat{\omega}_{X,p} \to \Bbbk[\![z]\!]\textnormal{d}z\). More precisely, the differential \(\cO_{X,p} \to \Omega^1_{\cO_{X,p}/\Bbbk} \cong \omega_{X,p}\) induces a continuous differential \(\textnormal{d} \colon \widehat{\cO}_{X,p}\to \widehat{\omega}_{X,p}\) whose image generates \(\widehat{\omega}_{X,p}\) and \(c^*(\textnormal{d}f) = c(f)'\textnormal{d}z\). The isomorphism \(c^*\) respects residues by \cite[Theorem 2]{tate_residue_paper}, i.e.\ \(\textnormal{res}_p \eta = \textnormal{res}_0 c^*\eta\) for all \(\eta \in \textnormal{Q}(\widehat{\omega}_{X,p})\). Here we extended \(c^*\) to a map \(\textnormal{Q}(\widehat{\omega}_{X,p}) \to \Bbbk(\!(z)\!)\textnormal{d}z\).
\end{remark}

\subsection{Geometric trichotomy of formal generalized \(r\)-matrices}
As in the previous subsection, let \(\fg\) be a finite-dimensional, central, simple Lie algebra over \(\Bbbk\). 

\begin{theorem}\label{thm:geometry_of_formal_rmatrices}
Let \(r \in (\fg \otimes \fg)(\!(x)\!)[\![y]\!]\) be a normalized formal \(r\)-matrix. There exists an unital \(\Bbbk\)-subalgebra \(O \subseteq \textnormal{Mult}(\fg(r)) = \{f\in\Bbbk(\!(z)\!)\mid f\fg(r) \subseteq \fg(r)\}\) of finite codimension such that the associated geometric datum \(\textnormal{GD}(O,\fg(r)) = ((X,\cA),(p,c,\zeta))\) has the following properties:
\begin{enumerate}
    \item \(X\) has arithmetic genus one.
    \item Let \(\breve{X}\) be the smooth locus of \(X\). The locally free sheaf \(\cA|_{\breve{X}}\) of Lie algebras is \'etale \(\fg\)-locally free and its Killing form extends uniquely to an invariant perfect pairing \(\cA \times \cA \to \cO_X\).
    \item There exists a global generator \(\eta\) of the dualizing sheaf \(\omega_X\) of \(X\) such that \(c^*(\widehat{\eta}_p) = \textnormal{d}z\) in the notation of \cref{rem:c*_construction_for_1_forms}.
\end{enumerate}
\end{theorem}
\begin{proof}
\textbf{Step 1. }\emph{Setup. } Let \(\widetilde{O}\) be the integral closure of \(\textnormal{Mult}(\fg(r))\) and \(O \subseteq\textnormal{Mult}(\fg(r))\) be an, for the moment arbitrary, unital subalgebra of finite codimension. Note that \(\widetilde{O}\) is also the integral closure of \(O\) and the genus of the normalization of \(X\) is \(\widetilde{g}\coloneqq \dim(\Bbbk(\!(z)\!)/(\Bbbk[\![z]\!]+\widetilde{O}))\). \cref{prop:geometry_of_generalized_rmatrices} implies that \(\widetilde{g} \in \{0,1\}\).\\
\textbf{Step 2.} \emph{It is possible to chose \(O\) such that \(X\) has arithmetic genus one. } For \(\widetilde{g} = 1\) this is already proven in \cref{prop:geometry_of_generalized_rmatrices}, so we may assume \(\widetilde{g}=0\), i.e.\ \(\Bbbk[\![z]\!]+\widetilde{O} = \Bbbk(\!(z)\!)\). Since \(\widetilde{O} \cap \Bbbk[\![z]\!]=\Bbbk\), the canonical projection \(O \to \Bbbk[z^{-1}]\) is injective and as a consequence \(\widetilde{O} = \Bbbk[u]\) for an arbitrary \(u \in z^{-1}\Bbbk[\![z]\!]^\times \cap \widetilde{O}\). \cref{prop:geometry_of_generalized_rmatrices}.\emph{(2)} states that
\begin{align}\label{eq:complement_of_tilde_O_is_multiplier}
    \widetilde{O}^\bot\fg(r) \subseteq \fg(r)^\bot = \fg(r),
\end{align}
where \(\Bbbk(\!(z)\!)\) is equipped with the bilinear form \((f,g) \mapsto \textnormal{res}_0fg\textnormal{d}z\) and the last equality follows from \cref{prop:formal_rmatrices_are_skew-symmetric} and  \cref{lem:orthogonal_complement_of_g(r)}. 
In particular, \(\widetilde{O}^\bot \subseteq \textnormal{Mult}(\fg(r)) \subseteq \widetilde{O}\). Since for all \(k \in \bN_0\)
\begin{equation}
    \textnormal{res}_0 u^ku'\textnormal{d}z = \textnormal{res}_0\frac{1}{k+1}\left(u^{k+1}\right)'\textnormal{d}z = 0
\end{equation}
holds, where \((\cdot)'\) denotes the formal derivative with respect to \(z\), we see that \(u'\widetilde{O} \subseteq \widetilde{O}^\bot \subseteq \widetilde{O}\). Thus, \cref{eq:complement_of_tilde_O_is_multiplier} implies that \(\Bbbk + u'\widetilde{O} \subseteq \textnormal{Mult}(\fg(r))\). The fact that \(u' \in \widetilde{O}=\Bbbk[u]\) has order two can be used to see that \(\Bbbk + u'\widetilde{O} = \Bbbk[u',u'u]\) is an unital \(\Bbbk\)-subalgebra of \(\textnormal{Mult}(\fg(r))\) such that \(\dim(\Bbbk(\!(z)\!)/(\Bbbk[\![z]\!] + \Bbbk[u',u'u])) = 1\). Choosing \(O = \Bbbk[u',u'u]\) implies that \(X\) has arithmetic genus 1.\\
\textbf{Step 3.} \emph{Choosing \(O\) as in Step 2, the Killing form of \(\cA|_{\breve{X}}\) extends to a paring \(\cA \times \cA \to \cO_X\). } We have \(\widetilde{X} = X = \breve{X}\) for \(\widetilde{g} = 1\) by virtue of \cref{prop:geometry_of_generalized_rmatrices}, so we may assume \(\widetilde{g}=0\), i.e.\ \(\widetilde{O} = \Bbbk[u]\) and \(O = \Bbbk[u',u'u]\).  Since \(X = \breve{X} \cup (X \setminus \{p\})\), we have to define the paring \(\cA \times \cA \to \cO_X\) on the affine open set \(X\setminus \{p\}\) and show that it is compatible with the Killing form of \(\cA|_{\breve{X}}\). The diagram \cref{eq:kappaKdiagram} implies that \(\kappa\colon \fg(\!(z)\!) \times \fg(\!(z)\!)\to \Bbbk(\!(z)\!)\) restricts to a mapping \(\widetilde{\fg(r)} \times \widetilde{\fg(r)}\to \widetilde{O}\) and it suffices to show that this pairing restricts further to \(\fg(r) \times \fg(r) \to O\). Observe that \(\fg(r)^\bot = \fg(r)\) implies that \(\kappa\) restricts to a bilinear map
\begin{equation}
    \fg(r) \times \fg(r) \to P \coloneqq \{f \in \widetilde{O}\mid \textnormal{res}_0f\textnormal{d}z = 0\}.
\end{equation}
We have \( O \subseteq \widetilde{O}^\bot \subseteq P\), since \(1 \in \widetilde{O}\). The codimension of \(P\) in \(\widetilde{O}\) is one, i.e.\ \(\widetilde{O}/P\) is spanned by \(u + P\). The same is true for \(\widetilde{O}/O\) proving \(P = O\) and concluding the proof.\\
\textbf{Step 4.} \emph{The pairing \(\cA \times \cA \to \cO_X\) constructed in Step 3 is perfect. }
Similarly as in Step 1 of the proof of \cref{prop:geometry_of_generalized_rmatrices}, the morphism \(\cA \to \cA^*\) induced by \(\cA \times \cA \to \cO_X\) is injective with a torsion cokernel \(\mathcal{C}\). The long exact sequence in cohomology induced by
\[0\longrightarrow \cA \longrightarrow \cA^* \longrightarrow \mathcal{C}\longrightarrow 0\]
combined with \(\textnormal{h}^0(\cA) = 0 = \textnormal{h}^1(\cA)\) implies \(\textnormal{h}^0(\mathcal{C}) = \textnormal{h}^0(\cA^*) = \textnormal{h}^1(\cA) = 0\). Here we used the Serre duality (in e.g.\ the version \cite[Chapter 6, Remark 4.20 and Theorem 4.32]{liu}), where we note that \(\textnormal{h}^1(\cO_X) = 1\) and the fact that \(X\) is locally a complete intersection (see e.g.\ \cref{rem:curves_of_arithmetic_genus_one} below) imply that the dualizing sheaf of \(X\) is trivial. We conclude that \(\mathcal{C} = 0\). Thus, \(\cA \to \cA^*\) is an isomorphism. This is equivalent to saying that  \(\cA \times \cA \to \cO_X\) is perfect.\\
\textbf{Step 5.} \emph{Reduction of the remaining statements to the case that \(\Bbbk\) is algebraically closed. } It remains to prove that \(\cA|_{\Breve{X}}\) is \'etale \(\fg\)-locally free and that there exists \(\eta \in \Gamma(X,\omega_X)\) such that \(c^*(\widehat{\eta}_p)= \textnormal{d}z\). The latter statement is equivalent to \(c^*(\widehat{\eta}_p) \in \Bbbk \textnormal{d}z\) for all \(\eta \in\Gamma(X,\omega_X)\), since \(\omega_X\) is trivial. Let
\begin{equation}
    \textnormal{GD}(O_{\overline{\Bbbk}},\fg(r_{\overline{\Bbbk}})) = ((X_{\overline{\Bbbk}},\cA_{\overline{\Bbbk}}),(p_{\overline{\Bbbk}},c_{\overline{\Bbbk}},\zeta_{\overline{\Bbbk}})) \textnormal{ and } \pi \colon X_{\overline{\Bbbk}}\to X
\end{equation}
be as in \cref{rem:field_extension_geometry}, where we recall that \(\overline{\Bbbk}\) is the algebraic closure of \(\Bbbk\). Note that \(r_{\overline{\Bbbk}}\) is a normalized formal \(r\)-matrix, \(\pi\) factors as \(X_{\overline{\Bbbk}}\cong X\times \textnormal{Spec}(\overline{\Bbbk}) \to X\), where the second map is the canonical projection, and \(\cA_{\overline{\Bbbk}} \cong \pi^*\cA\). \cref{lem:etale_locally_triviallity_wlog_over_algebraic_closure} states that \(\cA|_{\breve{X}}\) is \'etale \(\fg\)-locally free if \(\cA_{\overline{\Bbbk}}|_{\breve{X}_{\overline{\Bbbk}}}\) is \'etale \(\fg \otimes \overline{\Bbbk}\)-locally free, where \(\pi^{-1}(\breve{X}) = \breve{X}_{\overline{\Bbbk}}\) is the smooth locus of \(X_{\overline{\Bbbk}}\). Furthermore, \(\pi^*\omega_X \cong \omega_{X_{\overline{\Bbbk}}}\) is the dualizing sheaf of \(X_{\overline{\Bbbk}}\) (see e.g.\ \cite[Theorem 3.6.1]{conrad}) and the image of \(c^*(\widehat{\eta}_p)\) under \(\Bbbk[\![z]\!]\textnormal{d}z \to \overline{\Bbbk}[\![z]\!]\textnormal{d}z\) equals \(c^*_{\overline{\Bbbk}}\big(\widehat{\pi^*\eta}\big)_{p_{\overline{\Bbbk}}}\), hence \(c^*(\widehat{\eta}_p) \in \Bbbk \textnormal{d}z\) if and only if \(c^*_{\overline{\Bbbk}}\big(\widehat{\pi^*\eta}\big)_{p_{\overline{\Bbbk}}} \in \overline{\Bbbk} \textnormal{d}z\). Summarized, we may assume that \(\Bbbk = \overline{\Bbbk}\).\\
\textbf{Step 6.} \emph{\(\cA|_{\breve{X}}\) is \'etale \(\fg\)-locally free. } By Step 5 we may assume that \(\Bbbk\) is algebraically closed. Using Step 4 and \cref{lem:fiber_of_sheaf_Killing_form}, we see that \(\cA|_q\) is semi-simple for all closed \(q \in \breve{X}\). Therefore, \(\cA\) is \'etale \(\cA|_q\)-locally free in each 
closed point \(q \in \breve{X}\). Consequently, \(\cA\) is weakly \(\cA|_q\)-locally free in all closed points in some open neighbourhood of \(q\), since \'etale maps are open; see also \cref{rem:local_triviality_of_sheaves_of_algebras}. This forces \(\cA|_q \cong \cA|_p \cong \fg\) for all \(q \in \breve{X}\) closed, since \(\breve{X}\) is connected. \Cref{thm:etale_trivial_sheaves_of_Lie_algebras} states that \(\cA|_{\breve{X}}\) is \'etale \(\fg\)-locally free.\\
\textbf{Step 7.} \emph{There exists \(\eta \in \Gamma(X,\omega_X)\) such that \(c^*(\widehat{\eta}_p) = \textnormal{d}z\). }
By Step 5 we may assume that \(\Bbbk\) is algebraically closed.
Let \(\eta\) be a non-zero global section of the dualizing sheaf \(\omega_X\) of \(X\) and \(c^*(\widehat{\eta}_p) = \textnormal{d}w(z) = w'(z)\textnormal{d}z\) for some \(w(z) \in z\Bbbk[\![z]\!]\). The dualizing sheaf can be identified with the sheaf of Rosenlicht regular 1-forms, i.e.\ \(\eta\) is a rational 1-form on \(\widetilde{X}\), where we recall that \(\nu \colon \widetilde{X}\to X\) is the normalization of \(X\), which is regular on \(\nu^{-1}(\breve{X})\) and satisfies 
\begin{align*}
    &\sum_{q \in \nu^{-1}(s)} \textnormal{res}_q f \eta = 0&\forall s \in X\textnormal{ singular and closed}, f\in \cO_{X,s};
\end{align*}
see e.g.\ \cite[Theorem 5.2.3]{conrad}.
The residue theorem on \(\widetilde{X}\) implies that \(\textnormal{res}_p f \eta = 0\) for all \(f \in \Gamma(X\setminus\{p\},\cO_X)\). Combining this with \cref{rem:c*_construction_for_1_forms} and using diagram \cref{eq:kappaKdiagram} results in
\[\textnormal{res}_0\kappa(\zeta(a),\zeta(b))w'\textnormal{d}z = \textnormal{res}_p K(a,b)\eta = 0\]
for all \(a,b\in \Gamma(X\setminus\{p\},\cA)\). This implies that \(w'\fg(r) \subseteq \fg(r)^\bot = \fg(r)\). In other words, we obtain \(w' \in \textnormal{Mult}(\fg(r)) \cap \Bbbk[\![z]\!]^\times = \Bbbk^\times\). We can conclude the proof by replacing \(\eta\) with \((w')^{-1}\eta \in \textnormal{H}^0(X,\omega_X)\).\\
\end{proof}

\begin{remark}\label{rem:curves_of_arithmetic_genus_one}
Every integral projective curve \(X\) over \(\Bbbk\) of arithmetic genus one with a \(\Bbbk\)-rational smooth point \(p\) is a plane cubic curve, i.e.\ determined by one cubic equation. This can be seen for example from \cref{eq:adelicsequence}: fix an isomorphism \(c\colon \compO_{X,p} \to \Bbbk[\![z]\!]\) write \(O \coloneqq c(\Gamma(X\setminus\{p\},\cO_X)\) and note that the codimension of \(\Bbbk[\![z]\!]+O\) in \(\Bbbk(\!(z)\!)\) is one. Then \(O \cap \Bbbk[\![z]\!] = c(\Gamma(X,\cO_X)) = \Bbbk\) implies that \(O = \Bbbk[f,g]\), where \(f\) has order 2 and \(g\) has order 3. After probably adjusting \(f\) and \(g\), we get \(g^2 = f^3 + af + b\) for some \(a,b \in \Bbbk\). This is a minimal polynomial relation between \(f\) and \(g\), i.e.\ \(O \cong \Bbbk[x,y]/(y^2 - x^3-ax-b)\). 

The curve \(X\) is smooth if and only if \(4a^3+27b^2 \neq 0\), in which case it is elliptic, and has a unique nodal (resp. cuspidal) singularity if \(4a^3 = -27b^2 \neq 0\) (resp. \(a = b = 0\)). In the singular cases \(X\) is rational, i.e.\ \(X\) has the normalization \(\nu \colon \mathbb{P}^1_\Bbbk \to X\).
\end{remark}

\subsection{The geometric \(r\)-matrix}\label{subsec:geometric_rmatrix}
Let \(\fg\) be a finite-dimensional, central, simple Lie algebra over \(\Bbbk\). Fix a generalized \(r\)-matrix \(r \in (\fg\otimes\fg)(\!(x)\!)[\![y]\!]\), an unital subalgebra \(O\subseteq \textnormal{Mult}(\fg(r))\) of finite codimension and write \(\textnormal{GD}(O,\fg(r)) = ((X,\cA),(p,c,\zeta))\) for the associated geometric datum.
The sheaf \(\cA\) is \'etale \(\fg\)-locally free in \(p\) by \cref{thm:etale_triviality_due_to_semisimple_fiber}. Since \'etale morphisms are open, we can chose a smooth open neighbourhood \(C\) of \(p\) such that \(\cA|_C\) is \'etale \(\fg\)-locally free. After probably shrinking \(C\), there exists a non-vanishing 1-form \(\eta\) on \(C\).
We have obtained a geometric datum \(((X,\cA),(C,\eta))\) consisting of
\begin{itemize}
    \item an integral projective curve \(X\) over \(\Bbbk\), a non-empty smooth open subscheme \(C \subseteq X\) and a non-vanishing 1-form \(\eta\) on C as well as
    \item a coherent sheaf of Lie algebras \(\cA\) on \(X\) such that \(\textnormal{h}^0(\cA) = 0 = \textnormal{h}^1(\cA)\) and \(\cA|_C\) is \'etale \(\fg\)-locally free.
\end{itemize}

\begin{remark}
For \(\Bbbk = \overline{\Bbbk}\) the geometric datum \(((X,\cA),(C,\eta))\) satisfies the axioms used in \cite{burban_galinat} to construct a geometric analog of a generalized \(r\)-matrix called \emph{geometric \(r\)-matrix}. Indeed, \cref{thm:etale_trivial_sheaves_of_Lie_algebras} implies that \(\cA|_C\) is \'etale \(\fg\)-locally free if and only if it is weakly \(\fg\)-locally free for \(\Bbbk = \overline{\Bbbk}\). Therefore, the above conditions can be seen as an appropriate generalization of the axioms used in \cite{burban_galinat} if one works over a not necessarily algebraically closed ground field. In the following we recall the construction of the geometric \(r\)-matrix and observe that it works in our generalized setting.
\end{remark}

\begin{remark}
If \(r\) is normalized and skew-symmetric and \(X,\eta\) are chosen as in \cref{thm:geometry_of_formal_rmatrices}, we may assume that \(C\) is the smooth locus of \(X\). If \(\Bbbk = \overline{\Bbbk}\), we can see that \(((X,\cA),(C,\eta))\) satisfies the geometric axiomatization of skew-symmetry given in \cite[Theorem 4.3]{burban_galinat}, where the third condition in said theorem can be seen as a consequence of the fact that the Killing form of \(\cA|_C\) extends to a pairing \(\cA \times \cA \to \cO_X\) by \cref{thm:geometry_of_formal_rmatrices}.\emph{(2)}; see \cite[Theorem 1.2.(2)]{galinat_thesis}. 
\end{remark}

Let \(\Delta\) denote the image of the diagonal embedding \(\delta \colon C \to X \times C\). The choice of non-vanishing 1-form \(\eta\) induces the so-called \emph{diagonal residue sequence}
\begin{align}\label{eq:resseq}
    0 \longrightarrow \cO_{X \times C} \longrightarrow \cO_{X\times C}(\Delta) \stackrel{\textnormal{res}^\eta_\Delta}{\longrightarrow} \delta_*\cO_C\longrightarrow 0.
    \end{align}
The map \(\textnormal{res}^\eta_\Delta\) is thereby determined as follows: for a closed point \(q\in C\) with local parameter \(u\) defined on an affine open subset \(U\) of \(C\), the sheaves \(\Omega^1_C\) and \(\cO_{X\times C}(-\Delta)\) are locally generated by \(\textnormal{d}u\) and
\[u-v \coloneqq u \otimes 1 - 1 \otimes u \in \Gamma(U,\cO_X) \otimes \Gamma(U,\cO_X) \cong \Gamma(U \times U,\cO_{X \times X})\]
around \(q\) and \((q,q)\) respectively; \(\textnormal{res}_\Delta^\eta\) maps \((u-v)^{-1}\) to \(\mu\), where \(\mu\) is defined by \(\eta_q = \mu \textnormal{d}u\).
Tensoring \cref{eq:resseq} with \(\cA \boxtimes \cA|_C \coloneqq \textnormal{pr}_1^*\cA \otimes_{\cO_{X\times C}}\textnormal{pr}_2^*\cA|_C,\), where \(X \stackrel{\textnormal{pr}_1}{\longleftarrow} X\times C \stackrel{\textnormal{pr}_2}{\longrightarrow} C\) are the canonical projections, gives rise to a short exact sequence
\begin{align}\label{eq:Resseq}
    0 \longrightarrow \cA \boxtimes \cA|_C \longrightarrow \cA \boxtimes \cA|_C(\Delta) \longrightarrow \delta_*(\cA|_C \otimes_{\cO_C} \cA|_C)\longrightarrow 0.
\end{align}
The K\"unneth formula implies that
\begin{equation}
    \begin{split}
        &\textnormal{H}^0(\cA \boxtimes \cA|_C) = \textnormal{H}^0(\cA) \otimes \textnormal{H}^0(\cA|_C) =  0 \textnormal{ and }
        \\&\textnormal{H}^1(\cA \boxtimes \cA|_C) = \textnormal{H}^1(\cA) \otimes \textnormal{H}^0( \cA|_C) \oplus \textnormal{H}^0(\cA) \otimes \textnormal{H}^1(\cA|_C) = 0,
    \end{split}
\end{equation}
where we used \(\textnormal{h}^0(\cA) = 0 = \textnormal{h}^1(\cA)\). The long exact sequence in cohomology induced by
\cref{eq:Resseq} yields an isomorphism \(R\colon\textnormal{H}^0(\cA \boxtimes \cA|_C(\Delta)) \to \textnormal{H}^0(\cA|_C \otimes \cA|_C)\). 

\Cref{rem:etale_morphisms_wlog_surjective} states that there exists a surjective \'etale morphism \(f \colon Y \to C\) such that \(f^*\cA|_C \cong \fg \otimes\cO_Y\), while \cref{lem:fiber_of_sheaf_Killing_form} asserts that the inverse image \(f^*K\) of the Killing form \(K\) of \(\cA|_C\) can be identified with the Killing form of \(\fg \otimes \cO_Y\). The pairing \(f^*K\) is prefect due to the simplicity of \(\fg\). Thus, the surjectivity of \(f\) implies that \(K\) is also perfect. This implies that the morphism 
\[\widetilde{K}\colon\cA|_C\otimes_{\cO_C}\cA|_C \to \sheafEnd_{\cO_C}(\cA|_C),\]
defined by \(a \otimes b \mapsto K_U(b,-)a\) for all affine open \(U \subseteq C\) and \(a,b\in\Gamma(U,\cA)\), is an isomorphism.
Summarized, we obtain an isomorphism \(\Phi\coloneqq \widetilde{K}R\colon\textnormal{H}^0(\cA \boxtimes \cA|_C(\Delta)) \to \textnormal{End}_{\cO_C}(\cA|_C)\).

\begin{definition}
The section \(\rho\coloneqq \Phi^{-1}(\textnormal{id}_{\cA|_C}) \in \textnormal{H}^0(\cA \boxtimes \cA|_C(\Delta))\) is called \emph{geometric \(r\)-matrix} of \(((X,\cA),(C,\eta))\).
\end{definition}

Let us begin the discussion of the geometric \(r\)-matrix by observing that geometric \(r\)-matrices, constructed from equivalent formal generalized \(r\)-matrices in the way presented above, are equivalent in an geometric sense.

\begin{lemma}\label{lem:equivalence_of_geometric_rmatrices}
For \(i \in \{1,2\}\), let \(r_i \in (\fg \otimes \fg)(\!(x)\!)[\![y]\!]\) be a generalized \(r\)-matrix and denote the associated geometric datum by
\(\textnormal{GD}(\textnormal{Mult}(\fg(r_i)),\fg(r_i)) = ((X_i,C_i),(p_i,c_i,\zeta_i))\). Assume that \(r_1\) and \(r_2\) are equivalent and \(f\colon X_2\to X_1\) and \(\phi \colon \cA_1 \to f_*\cA_2\) are the isomorphisms provided by \cref{lem:geometry_of_equivalences}. We can chose a geometric datum \(((X_i,\cA_i),(C_i,\eta_i))\) for \(i \in \{1,2\}\) as in the beginning of this section in such a way that \(f^{-1}(C_1) = C_2\) and \(f^*\eta_1 = \eta_2\). In this case, the associated geometric \(r\)-matrices \(\rho_1, \rho_2\) satisfy
\(\rho_2 = (f^*(\phi) \boxtimes f^*(\phi))(f\times f)^*\rho_1.\)
\end{lemma}
\begin{proof}
It is obvious that we may chose the geometric datum in such a way that  \(f^{-1}(C_1) = C_2\) and \(f^*\eta_1 = \eta_2\).
Consider the commutative diagram
\begin{align}\label{eq:Resseq_for_equivalences}
    \xymatrix{0 \ar[r]& f^*\cA_1 \boxtimes f^*\cA_1|_C \ar[r]\ar[d]^{f^*(\phi) \boxtimes f^*(\phi)}& f^*\cA_1 \boxtimes f^*\cA_1|_C(\Delta_2) \ar[r]\ar[d]^{f^*(\phi) \boxtimes f^*(\phi)}& \delta_{2,*}\sheafEnd_{\cO_{C_2}}(f^*\cA_1|_{C_1})\ar[r]\ar[d]^{f^*(\phi) - f^*(\phi)^{-1}}& 0\\
    0 \ar[r]& \cA_2 \boxtimes \cA_2|_C \ar[r]& \cA_2 \boxtimes \cA_2|_C(\Delta_2) \ar[r]& \delta_{2,*}\sheafEnd_{\cO_{C_2}}(\cA_2|_{C_2})\ar[r]& 0},
\end{align}
where the upper row is \((f\times f)^*\) of \cref{eq:Resseq} for \(\cA= \cA_1\), the lower row is \cref{eq:Resseq} for \(\cA= \cA_2\), \(\Delta_2\) is the image of the diagonal embedding \(\delta_2\colon C_2 \to X_2 \times C_2\) and the isomorphisms \(\cA_i|_{C_i} \otimes \cA_i|_{C_i} \to \sheafEnd_{\cO_{C_i}}(\cA_i|_{C_i})\) for \(i \in \{1,2\}\) were used. Here, the commutativity of the right square follows form the fact that for a free Lie algebra \(\mathfrak{l}\) of finite rank over an unital commutative ring \(R\) with Killing form \(\kappa_\mathfrak{l}\), the adjoint of \(\psi \in \textnormal{Aut}_{R\textnormal{-alg}}(\mathfrak{l})\) with respect to \(\kappa_\mathfrak{l}\) coincides with \(\psi^{-1}\), so
\begin{equation}
    \widetilde{\kappa}_\mathfrak{l}(\psi(a) \otimes \psi(b)) = \psi\widetilde{\kappa}_\mathfrak{l}(a\otimes b)\psi^{-1}
\end{equation}
holds for \(\widetilde{\kappa}_\mathfrak{l}(a\otimes b) \coloneqq \kappa_\mathfrak{l}(b,\cdot)a \in \End_R(\mathfrak{l})\) and \(a,b \in \mathfrak{l}\).

It is straight forward to show that the inverse image along \(f\times f\) of the diagonal residue sequence \cref{eq:resseq} for \(X=X_1\) and \(\eta = \eta_1\) coincides with the diagonal residue sequence for \(X = X_2\) and \(\eta = f^*\eta_1 = \eta_2\). This implies that \((f\times f)^*\rho_1\) is mapped to \(\textnormal{id}_{f^*\cA_1|_{C_1}}\) under \(\textnormal{H}^0(f^*\cA_1 \boxtimes f^*\cA_1|_{C_1}(\Delta_2)) \to \End_{\cO_{C_2}}(f^*\cA_1|_{C_1})\). Thus, the commutativity of \cref{eq:Resseq_for_equivalences} implies that 
\begin{equation}
    (f^*(\phi) \boxtimes f^*(\phi))(f\times f)^*\rho_1\longmapsto \textnormal{id}_{\cA_2|_{C_2}}
\end{equation}
under \(\textnormal{H}^0(\cA_2 \boxtimes \cA_2|_{C_2}(\Delta_2)) \to \End_{\cO_{C_2}}(\cA_2|_{C_2})\), so \(\rho_2=(f^*(\phi) \boxtimes f^*(\phi))(f\times f)^*\rho_1\) holds.
\end{proof}

The geometric \(r\)-matrix has a simple pole along the diagonal with predetermined residue. This can be used to derive a local standard form in the vein of the standard form of formal generalized \(r\)-matrices.

\begin{lemma}\label{lem:geomstdform}
There exists an affine open neighbourhood \(U \subseteq C\) of \(p\) admitting a local parameter \(u \in \Gamma(U,\cO_X)\) of \(p\) such that \(\textnormal{d}u\) generates \(\Gamma(U,\Omega^1_C)\) and \(u-v\) generates  \(\Gamma(U\times U,\cO_{X\times C}(-\Delta))\).
Let \(\mu \in \Gamma(U,\cO_X)\) be defined by 
\(\eta = \mu^{-1}\textnormal{d}u\) and fix an arbitrary preimage \(\chi\) of \(\textnormal{id}_{\cA|_U}\) under the surjective map 
\begin{equation}\label{eq:preimage_of_identity}
    \Gamma(U\times U,\cA \boxtimes \cA|_C) \longrightarrow \Gamma(U,\cA|_C \otimes_{\cO_C}\cA|_C) \stackrel{\widetilde{K}_U}{\longrightarrow} \textnormal{End}_{\cO_U}(\cA|_U).
\end{equation}
There exists \(t \in \Gamma(U \times U,\cA\boxtimes \cA)\) such that
\(
    \rho|_{U\times U} = \frac{1\otimes \mu}{u-v}\chi + t.
\)
\end{lemma}
\begin{proof}
Let \(u \in \Gamma(V,\cO_X)\) be any local parameter of \(p\), where \(V\subseteq C\) is an affine open neighbourhood of \(p\). Chose \(U\) to be some affine open subset of the intersection of
\begin{itemize}
    \item the projection of the open set \(\{(q,q')\in V\times V\mid (u-v)(q,q') \neq 0\}\) to the first component with
    \item an open neighbourhood of \(p\) where \(\textnormal{d}u\) generates \(\Omega^1_C\). 
\end{itemize} 
It is straight forward to show that \(\textnormal{d}u\) generates \(\Gamma(U,\Omega^1_C)\) and \(u-v\) generates  \(\Gamma(U\times U,\cO_{X\times C}(-\Delta))\). By construction, \(\textnormal{res}_\Delta^\eta \frac{1\otimes \mu}{u-v} = 1\). Thus, we can see that \(\rho|_{U\times U}\) and \(\frac{1\otimes \mu}{u-v}\chi\) map to \(\textnormal{id}_{\cA|_U}\) under
\begin{equation}
    \Gamma(U\times U,\cA\boxtimes\cA|_C(\Delta)) \longrightarrow \Gamma(U,\cA|_C\otimes \cA|_C) \longrightarrow \textnormal{End}(\cA|_U).
\end{equation}
Applying the exact functor \(\Gamma(U \times U,\cdot)\) to \cref{eq:Resseq} implies that 
\(\rho|_{U\times U} - \frac{1\otimes \chi}{u-v}\chi\) is an element of  \(\Gamma(U \times U,\cA \boxtimes \cA)\), concluding the proof.
\end{proof}

Recall that \(c\) induces an isomorphism \(\textnormal{Q}(\widehat{\cO}_{X,p}) \to \Bbbk(\!(z)\!)\) which is again denoted by \(c\) and for all \(U \subseteq X\) open we have a natural inclusion \(\Gamma(U,\cO_X) \subseteq \textnormal{Q}(\widehat{\cO}_{X,p})\).

\begin{definition}\label{not:mathfrak_X}
Let \(\mathfrak{X}\) be the locally ringed space with underlying topological space \(X\) and sheaf of rings \(\cO_\mathfrak{X}\) defined by \(\Gamma(U,\cO_\mathfrak{X}) \coloneqq c(\Gamma(U,\cO_X))[\![y]\!]\) for all \(U \subseteq X\) open. Furthermore, let \(\jmath = (\jmath,\jmath^\flat)\colon\mathfrak{X} \to X \times X\) be the morphism of ringed space defined as follows: \(\jmath\) is the composition \(
X \to X\times \{p\} \to X \times X\) and for an affine open \(U \subseteq X\) and an affine open neighbourhood \(V \subseteq X\) of \(p\) the morphism \(\jmath^\flat\colon \Gamma(U\times V,\cO_{X\times X}) \to \Gamma(U,\cO_\mathfrak{X})\) is the projective limit over \(k\) of the morphisms
\[\Gamma(U \times V,\cO_{X\times X}) \cong \Gamma(U,\cO_X) \otimes \Gamma(V,\cO_X) \stackrel{c \otimes \overline{c}_k}{\longrightarrow} c(\Gamma(U,\cO_X)) \otimes \Bbbk[y]/(y^k),\]
where \(\overline{c}_k\) is the composition of \(c\) with the canonical map \(\Bbbk[\![z]\!] \to \Bbbk[z]/(z^k)\).
\end{definition}

\begin{remark}\label{rem:mathfrak_X}
\(\mathfrak{X}\) coincides, up to application of \(c\) and the identification \(X \cong X\times\{p\}\), with the formal scheme obtained by completing \(X\times X\) along \(X\times \{p\}\). In particular, \(\mathfrak{X}\) is indeed a locally ringed space.
Using the Künneth formula and \(\zeta\), it is possible to identify \(\jmath^*(\cA \boxtimes \cA)\) with the sheaf of Lie algebras on \(\mathfrak{X}\) defined by \(U \mapsto (\zeta(\Gamma(U,\cA)) \otimes \fg)[\![y]\!]\) for all \(U \subseteq X\) open. The canonical morphism \(\jmath^*\colon\cA \boxtimes \cA \to \jmath_*\jmath^*(\cA \boxtimes \cA)\) is injective, since, up to isomorphism, it amounts to the completion of a coherent sheaf at a sheaf of ideals.
In particular,
this construction yields an injective map
\begin{equation}\label{eq:geometric_taylor_series}
    \jmath^*\colon \Gamma(X \times C \setminus \Delta,\cA \boxtimes\cA) \longrightarrow \Gamma(X\setminus\{p\},\jmath^*(\cA \boxtimes\cA)) = (\fg(r) \otimes \fg)[\![y]\!],
\end{equation}
where \(\jmath^{-1}(X\times C \setminus\Delta) = X\setminus\{p\}\) and \(\zeta(\Gamma(X\setminus \{p\},\cA)) = \fg(r)\) was used. 
\end{remark}

The following statement can be seen as a generalization of \cite[Theorem 6.4]{burban_galinat}.

\begin{theorem}\label{thm:geometric_rmatrix_Taylor_series}
Let \(c^*(\widehat{\eta}_p) = \lambda_1(z)dz\), assume that \(r(x,y) - \lambda_2(y)r_{\textnormal{Yang}}(x,y)\in (\fg \otimes \fg)[\![x,y]\!]\) and put \(\lambda \coloneqq (\lambda_1\lambda_2)^{-1}\). The injective morphism \(\jmath^*\) maps \(\rho \in \textnormal{H}^0(\cA \boxtimes\cA|_C(\Delta)) \subseteq \Gamma(X \times C\setminus\Delta,\cA \boxtimes\cA)\) to \(\lambda(y)r(x,y)\). In particular, \(\jmath^*\rho = r\) if \(r\) is skew-symmetric and normalized and \(X, \eta\) are chosen as in \cref{thm:geometry_of_formal_rmatrices}.
\end{theorem}
\begin{proof}
\textbf{Step 1.} \emph{Setup. } Choose \(U, u, \mu,\chi\) and \(t\) as in \cref{lem:geomstdform}, i.e.\
\begin{equation}\label{pf:geometric_rmatrix_Taylor_series1}
    \rho|_{U\times U} = \frac{1\otimes \mu}{u-v}\chi + t.
\end{equation}
Let us write \(\widetilde{\mu} \coloneqq c(\mu), \widetilde{u} \coloneqq c(u)\in\Bbbk[\![z]\!]\). Then
\begin{equation}
    \jmath^\flat(1\otimes \mu) = \widetilde{\mu}(y), \jmath^\flat(u-v)=\widetilde{u}(x)-\widetilde{u}(y) \in \Bbbk[\![x,y]\!].
\end{equation}
Similarly, let \(\widetilde{t} \coloneqq \jmath^*(t),\widetilde{\chi}\coloneqq \jmath^*(\chi) \in (\fg \otimes \fg)[\![x,y]\!]\). \\
\textbf{Step 2.} \emph{The image of \(\rho\) under \cref{eq:geometric_taylor_series} is in standard form. }
The diagram 
\begin{gather}
\xymatrix{\Gamma(X\times C \setminus\Delta,\cA\boxtimes\cA) \ar[r]^-{\jmath^*}\ar[d]& (\fg(r)\otimes \fg)[\![y]\!]\ar[r] &(\fg\otimes \fg)(\!(x)\!)[\![y]\!]
\\\Gamma(U\times U\setminus\Delta,\cA \boxtimes\cA)\ar[r]_-{\jmath^*}& (\zeta(\Gamma(U\setminus\{p\},\cA))\otimes \fg)[\![y]\!] \ar[ur]&}
\end{gather}
commutes. Therefore, using the notations in Step 1 and \cref{pf:geometric_rmatrix_Taylor_series1}, the image of \(\rho\) under \cref{eq:geometric_taylor_series} is of the form
\begin{align*}
    \widetilde{r}(x,y) = \frac{\widetilde{\mu}(y)}{\widetilde{u}(x)-\widetilde{u}(y)}\widetilde{\chi}(x,y) + \widetilde{t}(x,y).
\end{align*}
\cref{lem:series_vanishing_at_diagonal} can be used to see that \((\widetilde{u}(x)-\widetilde{u}(y))^{-1}-(\widetilde{u}'(y)(x-y))^{-1} \in \Bbbk[\![x,y]\!]\) and \(\widetilde{\chi}(x,y) -\gamma \in (x-y)(\fg \otimes \fg)[\![x,y]\!]\). Moreover, \(\eta_p = \mu^{-1}\textnormal{d}u\) implies that \(\lambda_1 = \widetilde{u}'/\widetilde{\mu}\). Summarized, we obtain
\begin{align}
\widetilde{r}(x,y) = \frac{\gamma}{\lambda_1(y)(x-y)} + s(x,y)
\end{align}
for some \(s \in (\fg \otimes \fg)[\![x,y]\!]\). \\
\textbf{Step 3.} \emph{Concluding the proof. } The image \(\widetilde{r}\) of \(\rho\) under \cref{eq:geometric_taylor_series} is by definition in \((\fg(r) \otimes \fg)(\!(x)\!)[\![y]\!]\). But so is \(\lambda(y)r(x,y) \in \lambda_1(y)^{-1}r_{\textnormal{Yang}}(x,y) + (\fg \otimes \fg)[\![x,y]\!]\) and
\begin{align*}
    \widetilde{r}(x,y) - \lambda(y)r(x,y) \in (\fg(r) \otimes \fg)[\![y]\!] \cap (\fg \otimes \fg)[\![x,y]\!] = \{0\},
\end{align*}
concludes the proof.
\end{proof}

We conclude the section with discussing how the property of solving the CYBE resp. GCYBE can be formulated more globally using the geometric \(r\)-matrix. For this purpose we need to sheafify \cref{not:ij_notations}. In the following we assume that \(C\) is affine.

\begin{notation}
Let \(\mathcal{U}\) be the quasi-coherent sheaf on the affine scheme \(C\) associated to the universal enveloping sheaf of \(\tH^0(\cA|_C)\) as \(\tH^0(\cO_C)\)-Lie algebra and \(\iota \colon \tH^0(\cA|_C) \to \textnormal{H}^0(\mathcal{U})\) be the canonical map. For \(ij \in \{12,13,23\}\), let \(\pi_{ij} \colon C\times C \times C \to C \times C\) denote the natural projections defined through \((x_1,x_2,x_3) \mapsto (x_i,x_j)\) and note that there are natural maps 
\((\cdot)^{ij}\colon \cA|_C \boxtimes \cA|_C \to \pi_{ij,*}(\mathcal{U}\boxtimes \mathcal{U} \boxtimes \mathcal{U})\)
defined, under consideration of the K\"unneth formulas
\begin{align*}
    &\tH^0(\cA|_C \boxtimes \cA|_C) \cong \tH^0(\cA|_C) \otimes \tH^0(\cA|_C) \\ &\tH^0(\mathcal{U}\boxtimes \mathcal{U} \boxtimes \mathcal{U}) \cong \tH^0(\mathcal{U})\otimes\tH^0(\mathcal{U})\otimes\tH^0(\mathcal{U}),
\end{align*}
by \(t^{12} =\iota(a)\otimes \iota(b) \otimes 1,t^{13} = \iota(a) \otimes 1 \otimes \iota(b)\) and \(t^{23} = 1 \otimes \iota(a)\otimes \iota(b)\) for \(t = a \otimes b \in \tH^0(\cA|_C) \otimes \tH^0(\cA|_C)\). Furthermore, if \(\sigma \colon C\times C \to C\times C\) denotes the map \((x,y) \mapsto (y,x)\), let \(\overline{(\cdot)} \colon \cA|_C \boxtimes \cA|_C \to \sigma_*(\cA|_C \boxtimes \cA|_C)\) be the morphism defined on global sections by the \(\sigma\)-equivariant automorphism \(a\otimes b \mapsto -b\otimes a\).
\end{notation}

\begin{remark}
It can be shown that \(\iota\) is injective; see \cite[Lemma 1.6]{galinat_thesis}.
\end{remark}

The following result can be proven as \cite[Theorem 3.11 \& Theorem 4.3]{burban_galinat}. However, we will sketch a possible self-contained proof.
\begin{theorem}\label{thm:geometric_GCYBE}
The geometric \(r\)-matrix \(\rho\), treated as an element of \(\Gamma(C\times C\setminus\Delta,\cA\boxtimes\cA)\), solves the \emph{geometric GCYBE}
\begin{align}\label{eq:geomGCYBE}
    [\rho^{12},\rho^{13}]+[\rho^{12},\rho^{23}]+[\rho^{13},\overline{\rho}^{23}] = 0,
\end{align}
where the commutators on the left-hand side are understood in \(\Gamma(C\times C \times C\setminus \Sigma,\mathcal{U}\boxtimes \mathcal{U} \boxtimes \mathcal{U})\) for
\[\Sigma = \{(x_1,x_2,x_3) \in C\times C \times C\mid x_i\neq x_j,i\neq j\}.\]
Furthermore, if \(r\) is skew-symmetric then \(\overline{\rho} = \rho\).
\end{theorem}
\begin{proof}[Sketch of proof]
Similar calculations as in \cref{rem:formal_GCYBE_domain} show that the left-hand side of the geometric GCYBE \cref{eq:geomGCYBE} is actually contained in \(\Gamma(C \times C \times C\setminus \Sigma,\cA \boxtimes \cA \boxtimes \cA)\).
It is possible to construct an injective morphism 
\begin{align}\label{eq:geometric_Taylor_series_in_three_variables}
\Gamma(C \times C \times C\setminus \Sigma,\cA \boxtimes \cA \boxtimes \cA) \longrightarrow (\fg \otimes \fg \otimes \fg)(\!(x_1)\!)(\!(x_2)\!)[\![x_3]\!],    
\end{align}
in the same vein as \cref{eq:geometric_taylor_series}, where the only additional observation necessary in this construction is that \(\Gamma(C\times C\setminus \Delta,\cA \boxtimes \cA)\to (\fg \otimes\fg)(\!(x)\!)[\![y]\!]\) extends to a morphism
\begin{equation}
    \Gamma((C\setminus\{p\})\times (C\setminus\{p\})\setminus \Delta,\cA \boxtimes \cA)\to (\fg \otimes\fg)(\!(x)\!)(\!(y)\!).
\end{equation}
Using \cref{thm:geometric_rmatrix_Taylor_series}, it can be shown that
the map \cref{eq:geometric_Taylor_series_in_three_variables} sends the left hand side of \cref{eq:geomGCYBE} to \(\textnormal{GCYB}(\widetilde{r})\), where \(\widetilde{r}(x,y) = \lambda(y)r(x,y)\) for some \(\lambda \in \Bbbk[\![z]\!]^\times\). This is 0 since \(r\) is a formal generalized \(r\)-matrix. The injectivity of \cref{eq:geometric_Taylor_series_in_three_variables} implies that \(\rho\) solves \cref{eq:geomGCYBE}. Furthermore, it can be shown that \(\overline{\rho}\) is mapped to \(\overline{r}\) via the injective map \(\Gamma(C\times C\setminus \Delta,\cA \boxtimes \cA)\to (\fg \otimes\fg)(\!(x)\!)[\![y]\!]\). Therefore \(r = \overline{r}\) implies \(\rho = \overline{\rho}\).
\end{proof}

\subsection{Global extension of formal generalized \(r\)-matrices}\label{subsec:rational_extension_of_formal_generalized_rmatrices} Let \(\fg\) be a finite-dimensional, central, simple Lie algebra over \(\Bbbk\), \(r \in (\fg\otimes\fg)(\!(x)\!)[\![y]\!]\) be a generalized \(r\)-matrix, \(O\subseteq \textnormal{Mult}(\fg(r))\) be an unital subalgebra of finite codimension and write \(\textnormal{GD}(O,\fg(r)) = ((X,\cA),(p,c,\zeta))\) for the associated geometric datum.

Chose a smooth open neighbourhood \(C\) of \(p\) such that \(\cA|_C\) is \'etale \(\fg\)-locally free and there exists a non-vanishing 1-form \(\eta\) on \(C\). Let \(\rho \in \Gamma(X\times C \setminus \Delta,\cA\boxtimes \cA)\) be the geometric \(r\)-matrix of \(((X,\cA),(C,\eta))\). If we fix an \'etale morphism \(f \colon Y \to X\) such that \(p\in f(Y)\subseteq C\) and there exists an isomorphism \(\psi \colon f^*\cA \to \fg \otimes \cO_Y\) of sheaves of Lie algebras, the morphism \(\psi\boxtimes \psi\) defines an isomorphism \((f\times f)^*(\cA \boxtimes \cA) \cong (\fg \otimes \fg)\otimes \cO_{Y\times Y}\). Therefore, we can consider 
\begin{equation}\label{eq:etale_trivialized_rmatrix}
    \varrho \coloneqq (\psi\boxtimes \psi)(f\times f)^*\rho \in (\fg \otimes \fg)\otimes \Gamma(Y \times Y \setminus \Delta_f,\cO_{Y\times Y}),
\end{equation}
where \(\Delta_f = (f\times f)^{-1}(\Delta) = \{(x,y)\in Y\times Y\mid f(x) = f(y)\}\).
Observe that for \(\Bbbk = \overline{\Bbbk}\) this section can be identified with a rational map \(\varrho \colon Y(\Bbbk)\times Y(\Bbbk) \to \fg \otimes \fg\), where \(Y(\Bbbk)\) is the algebraic variety of closed points of \(Y\).

Let \(q \in f^{-1}(p)\) be a closed point. Using e.g.\ \cite[Theorem 19.6.4]{egaIV_part_1} it is possible to chose a \(\Bbbk\)-algebra isomorphism \(c'\colon \compO_{Y,q} \to \Bbbk(q)[\![z]\!]\), where \(\Bbbk(q)\) is the residue field of \(q\).  Similar to the construction of \(\mathfrak{X}\) in \cref{not:mathfrak_X}, we have a locally ringed space \(\mathfrak{Y}\) equipped with a morphism of locally ringed spaces \(\jmath'\colon \mathfrak{Y} \to Y\times Y\) whose topological space coincides with the one of \(Y\) and for all \(U\subseteq Y\) open \(\Gamma(U,\cO_{\mathfrak{Y}}) = c'(\Gamma(U,\cO_Y))[\![y]\!]\), where \(c'\) is extended to an isomorphism \(\textnormal{Q}(\compO_{Y,q}) \to \Bbbk(q)(\!(z)\!)\). Clearly
\begin{equation}
    \Gamma(V,\jmath^{\prime,*}((\fg \otimes \fg)\otimes \cO_{Y\times Y})) = (\fg \otimes \fg)\otimes c'(\Gamma(U,\cO_Y))[\![y]\!] \subseteq  (\fg_{\Bbbk(q)}\otimes_{\Bbbk(q)} \fg_{\Bbbk(q)})(\!(x)\!)[\![y]\!]
\end{equation}
for \(V \subseteq Y\) open such that \(\jmath^{\prime, -1}(V) = U\). If \(\Bbbk = \overline{\Bbbk}\) and \(c'\) is induced by the choice of a local parameter \(u\) of \(\cO_{Y,q}\), we can understand \(\jmath^{\prime,\flat}\) as the Taylor expansion in the second coordinate in \(q\) with respect to the local parameter \(u\).

\begin{theorem}\label{thm:extension_of_generalized_rmatrices}
Let \(r_{\Bbbk(q)} \in (\fg_{\Bbbk(q)}\otimes_{\Bbbk(q)} \fg_{\Bbbk(q)})(\!(x)\!)[\![y]\!]\) be the formal generalized \(r\)-matrix constructed from \(r\) in \cref{rem:field_extensions}. Then \(\jmath^{\prime, *}\varrho\) is equivalent to \(r_{\Bbbk(q)}\), where \(\varrho\) is defined in \cref{eq:etale_trivialized_rmatrix}. In particular, if \(\Bbbk = \overline{\Bbbk}\), \(r\) is equivalent to the Taylor series of the rational map \(\varrho\colon Y(\Bbbk) \times Y(\Bbbk) \to \fg \otimes \fg\) in the second variable in \(q\) with respect to some local parameter. 
\end{theorem}
\begin{proof}
Using \cref{thm:geometric_rmatrix_Taylor_series} and \cref{rem:field_extensions} we see that the image of \(\rho\) under  
\[\Gamma(X \times C \setminus\Delta,\cA \boxtimes \cA) \stackrel{\jmath^*}\longrightarrow (\fg \otimes \fg)(\!(x)\!)[\![y]\!]\longrightarrow (\fg_{\Bbbk(q)} \otimes_{\Bbbk(q)}\fg_{\Bbbk(q)})(\!(x)\!)[\![y]\!]\]
is equivalent to \(r_{\Bbbk(q)}\) by a rescaling \(\lambda \in \Bbbk[\![z]\!]^\times \subseteq \Bbbk(q)[\![z]\!]^\times\). The image of \(z\) under  
\[\Bbbk[\![z]\!]\stackrel{c^{-1}}\longrightarrow \widehat{\cO}_{X,p}\stackrel{\widehat{f}^\sharp_p}\longrightarrow\widehat{\cO}_{Y,q} \stackrel{c'}\longrightarrow \Bbbk(q)[\![z]\!]\]
defines a coordinate transform \(w \in z\Bbbk(q)[\![z]\!]^\times\). Similarly, the chain of maps
\[\fg[\![z]\!] \stackrel{\zeta^{-1}}\longrightarrow  \widehat{\cA}_p \longrightarrow \widehat{f^*\cA}_q \stackrel{\widehat{\psi}_q}\longrightarrow \fg \otimes \widehat{\cO}_{Y,q} \stackrel{\textnormal{id}_{\fg}\otimes c'}\longrightarrow \fg_{\Bbbk(q)}[\![z]\!],\]
where the middle arrow is the composition of \(\widehat{f}^*_p\) with the completion of the canonical map \((f_*f^*\cA)_p \to f^*\cA_q\), induces a gauge transformation \(\varphi \in \Aut_{\Bbbk(q)\textnormal{-alg}}(\fg_{\Bbbk(q)}[\![z]\!])\).
It is straight forward to show from the constructions of \(\mathfrak{X}\) and \(\mathfrak{Y}\) that the diagram
\[\xymatrix{\Gamma(X \times C \setminus\Delta,\cA \boxtimes \cA) \ar[r]^{\jmath^*}\ar[d]_{(f\times f)^*}& (\fg \otimes \fg)(\!(x)\!)[\![y]\!] \ar[r]&(\fg_{\Bbbk(q)} \otimes_{\Bbbk(q)}\fg_{\Bbbk(q)})(\!(x)\!)[\![y]\!] \\ \Gamma(Y\times Y,f^*\cA \boxtimes f^*\cA) \ar[r]_-{\psi \boxtimes \psi}&
(\fg \otimes \fg) \otimes \Gamma(Y\times Y\setminus \Delta_f,\cO_{Y\times Y}) \ar[ru]_{\jmath^{\prime, *}} &}\]
commutes, where the arrow in the upper right is defined by \(s(x,y) \mapsto (\varphi(x) \otimes \varphi(y))s(w(x),w(y))\) for all \(s\in (\fg \otimes\fg)(\!(x)\!)[\![y]\!]\). In particular, this implies that the series \(\jmath^{\prime, *}\varrho\) is equivalent to \(r_{\Bbbk(q)}\) via the equivalence \((\lambda,w,\varphi)\). If \(\Bbbk = \overline{\Bbbk}\), we can apply a coordinate transformation to assume that \(c'\) is induced by a local parameter of \(q\). This proves the second part of the statement.
\end{proof}

\begin{remark}
The section \(\varrho\) is the solution of yet another version of the GCYBE. 
For \(ij \in \{12,13,23\}\), let \(\pi'_{ij}\colon Y \times Y \times Y \to Y \times Y\) be the canonical projections \((y_1,y_2,y_3) \to (y_i,y_j)\) and \(\varrho^{ij}\) be the image of \(\varrho\) under
\begin{equation*}
    \xymatrix{(\fg \otimes \fg) \otimes \cO_{Y\times Y} \ar[r]^{\pi_{ij}^{\prime,*}}& (\fg \otimes \fg)\otimes \cO_{Y\times Y \times Y} \ar[rr]^-{(\cdot)^{ij}\otimes \textnormal{id}_{\cO_{Y\times Y \times Y}}}&& (\textnormal{U}(\fg)\otimes\textnormal{U}(\fg)\otimes \textnormal{U}(\fg))\otimes \pi_{ij,*}\cO_{Y\times Y \times Y}},
\end{equation*}
where \cref{not:ij_notations} was used. Furthermore, let \(\overline{\varrho} = -\tau'\sigma^{\prime,*}\varrho\), where \(\sigma' \colon Y\times Y \to Y\times Y\) is given by \(\sigma'(x,y) = (y,x)\) and \(\tau' \colon (\fg \otimes \fg)\otimes \cO_{Y\times Y}\to (\fg \otimes \fg)\otimes \cO_{Y\times Y}\) is the \(\cO_{Y\times Y}\)-linear extension of \(a\otimes b \mapsto b\otimes a\) for \(a,b\in \fg\). 
Then we have
\begin{equation}
    [\varrho^{12},\varrho^{13}]+[\varrho^{12},\varrho^{23}]+[\varrho^{13},\overline{\varrho}^{23}] = 0
\end{equation}
and if \(r\) is skew-symmetric, the identity \(\varrho = \overline{\varrho}\) holds. This can be proven exactly as was sketched for \cref{thm:geometric_GCYBE} using \cref{thm:extension_of_generalized_rmatrices}.
\end{remark}

\begin{remark}
Passing to the smooth completion \(Z\) of the normalization of \(Y\), we can consider \(\varrho\) as a rational section of \((\fg \otimes \fg) \otimes \cO_{Z\times Z}\), where now \(Z\) is an integral smooth projective curve over \(\Bbbk\). In particular, for \(\Bbbk = \bC\) the closed points of \(Z\) are form a compact Riemann surface \(R\) and \(\varrho\) is a rational map \(R \otimes R \to \fg \otimes \fg\).
\end{remark}

\begin{example}
Let \(\fg = \mathfrak{so}(n,\Bbbk)\) and  \(X_{ij}= E_{ij}-E_{ji}\), where \(E_{ij} = (\delta_{ik}\delta_{j\ell })_{k,\ell = 1}^n\).
In \cite{skrypnyk_new_integrable_gaudin_type_systems} the author examines 
\[r(x,y) = \frac{1}{x-y}\sum_{1 \le i < j \le n} \frac{\alpha_i(x)\alpha_j(x)}{\alpha_i(y)\alpha_j(y)}X_{ij}\otimes X_{ji},\]
where \(\alpha_i(z) \coloneqq (1-a_iz)^{1/2} = \sum_{k = 0}^\infty (-1)^k\binom{1/2}{k}(a_iz)^k \in \Bbbk[\![z]\!]\) for some non-zero \(a_i \in \Bbbk\), \(i \in \{1,\dots,n\}\), in the case that \(\Bbbk = \bC\) and considered as a meromorphic function. If we write \(A = \textnormal{diag}(a_1,\dots,a_n)\), it is easy to see that \(\fg(r)\) is the image of \(z^{-1}\fg[z^{-1}]\) under 
\[a(z) \longmapsto (1-Az)^{1/2}a(z)(1-Az)^{1/2}.\]
This is a homogeneous subalgebra of \(\fg(\!(z)\!)\) the sense of \cref{ex:homogeneous_Lie_subalgebras}. Therefore, \(r\) is a non-skew-symmetric normalized generalized \(r\)-matrix by \cref{prop:Lie_subalgebra_of_generalized_rmatrix}. The associated sheaf of Lie algebras \(\cA\) on \(X = \mathbb{P}^1_\Bbbk\) is obtained by gluing \(\fg(r)\) and \(\Gamma(U,\cA) \coloneqq (1-Az)^{1/2}\fg[z](1-Az)^{1/2}\), where \(\textnormal{Spec}(\Bbbk[z])\cong U \coloneqq \mathbb{P}^1_\Bbbk\setminus \{\infty\} \subset\mathbb{P}^1_\Bbbk\), along \(\textnormal{Spec}(\Bbbk[z,z^{-1}])\). Note that \(p = (z) \in U\).

Consider the canonical map 
\[\Bbbk[z]\longrightarrow R \coloneqq\Bbbk[z,u_1,\dots,u_n]/(u_1^2-(1-a_1z),\dots,u_n^2-(1-a_nz)).\]
Observe that \(u_i \mapsto \alpha_i\) defines an injective map \(R\to \Bbbk[\![z]\!]\) which identifies \(R\) with the subalgebra of \(\Bbbk[\![z]\!]\) generated by \(\{\alpha_1,\dots,\alpha_n\}\).
Let \(Y = \textnormal{Spec}(S)\), where \(S \coloneqq R_u\) for \(u\coloneqq {u_1\cdot} \dots {\cdot u_n}\), and \(f\colon Y \to U\) be the map induced by \(\Bbbk[z]\to R\to R_u = S\). Then \(f\) is \'etale by \cite[Corollary 3.16]{milne_etale_cohomology} and maps the \(\Bbbk\)-rational point \(\overline{(z,u_1-1,\dots,u_n-1)}_u \in Y\) to \(p = (z)\). Furthermore, the chain of maps
\[\Gamma(Y,f^*\cA) \cong \Gamma(U,\cA) \otimes_{\Bbbk[z]} S \to \fg \otimes S \cong \Gamma(Y,\fg \otimes \cO_Y),\] 
where the middle arrow is defined by \(\alpha_i\alpha_jX_{ij} \otimes 1 \mapsto X_{ij} \otimes u_iu_j\), is an isomorphism of Lie algebras over \(S\). Here we used that \(\{\alpha_i\alpha_jX_{ij}\}_{i \neq j}\) forms a \(\Bbbk[z]\)-basis of \(\Gamma(U,\cA)\) and \(u_iu_j \in S\) is a unit for all \(i,j\in \{1,\dots,n\}\). Thus, \(f^*\cA \cong \fg \otimes \cO_Y\) and \(r\) can be obtained from
\[\sum_{1 \le i < j \le n}  (X_{ij} \otimes X_{ji}) \otimes \frac{u_i u_j\otimes (u_i u_j)^{-1}}{z \otimes 1 - 1 \otimes z}\in (\fg \otimes \fg) \otimes (S\otimes S) \left[\frac{1}{z \otimes 1 - 1 \otimes z}\right]\]
in the way described in this section.
\end{example}

\section{Real and complex analytic generalized \(r\)-matrices}\label{sec:real_and_complex_generalized_rmatrices}

\subsection{Definition and relation to formal generalized \(r\)-matrices}\label{subsec:real_and_complex_generalized_rmatrices} In this subsection we define the classic complex analytic notion of generalized \(r\)-matrices and relate these to formal generalized \(r\)-matrices over \(\bC\). It is also possible to consider a real analytic context. Both cases work rather parallel and are developed simultaneously in the following. For this purpose, let \(\Bbbk \in \{\bR,\bC\}\) and \(\fg\) be a finite-dimensional Lie algebra over \(\Bbbk\).

\begin{definition}\label{def:real_meromorphic_maps}
For two real finite-dimensional vector spaces \(V,W\) and \(U \subseteq V\) open, a \emph{(real) meromorphic map} \(U \to W\) is a map defined on a dense open subset of \(U\) with values in \(W\) which is the restriction of a meromorphic map (in the usual complex analytic sense) \(\widetilde{U} \to W\otimes_\bR \bC\) to \(U\), where \(\widetilde{U} \subseteq V\otimes_\bR \bC\) is an open neighbourhood of \(U\). 
\end{definition}
\begin{remark}\label{rem:real_meromorphic_maps}
This definition could be done without reference to complex numbers by considering maps defined on a dense subset of \(U\) with values in \(V\) which are locally quotients of analytic maps.
\end{remark}
\begin{definition}\label{def:analytic_(generalized)_rmatrices}
Let \(U \subset \Bbbk\) be a connected open subset and \(r\colon U \times U \to \fg \otimes \fg\) be a meromorphic map.  Then \(r\) is called a \emph{generalized \(r\)-matrix} if it solves the \emph{generalized classical Yang-Baxter equation} (GCYBE)
\begin{align}\label{eq:GCYBE}
    [r^{12}(x_1,x_2),r^{13}(x_1,x_3)] + [r^{12}(x_1,x_2),r^{23}(x_2,x_3)] + [r^{32}(x_3,x_2),r^{13}(x_1,x_3)] = 0,
\end{align}
and \emph{\(r\)-matrix} if it solves the \emph{classical Yang-Baxter equation} (CYBE)
\begin{align}\label{eq:CYBE}
    [r^{12}(x_1,x_2),r^{13}(x_1,x_3)] + [r^{12}(x_1,x_2),r^{23}(x_2,x_3)] + [r^{13}(x_1,x_3),r^{23}(x_1,x_3)] = 0.
\end{align}
for all \(x_1,x_2,x_3 \in \Bbbk\) where these equations are defined respectively.
Here, \cref{not:ij_notations} as well as \(r^{ij}(x_i,x_j) = r(x_i,x_j)^{ij}, r^{32}(x_3,x_2)=\tau(r(x_2,x_3))^{23}\)  was used and the brackets on the left hand side of both equations are understood as the usual commutators in \(\textnormal{U}(\fg)\otimes \textnormal{U}(\fg) \otimes \textnormal{U}(\fg)\). If we want to emphasize that \(\Bbbk = \bC\) (resp. \(\Bbbk = \bR\)) we call \(r\) \emph{complex} (resp. \emph{real}) and we say that \(r\) is \emph{analytic} if we want to emphasize that we are not in the formalism of \cref{sec:formal_generalized_rmatrices}.
\end{definition}

\begin{definition}\label{def:non-degenerate}
A meromorphic map \(r\colon U \times U \to \fg \otimes \fg\), where \(U \subseteq \Bbbk\) is a connected open subset, is called 
\begin{itemize}
    \item \emph{non-degenerate} if for some \((x_0,y_0)\) in the domain of \(r\) the linear map \(\fg^* \to \fg\) defined by \(\alpha \mapsto (1 \otimes \alpha)r(x_0,y_0)\) is bijective and
    \item \emph{skew-symmetric} if \(r(x,y) = -\tau(r(x,y))\), where \(\tau(a\otimes b) = b\otimes a\) for \(a,b\in \fg\).
\end{itemize} 
\end{definition}

\begin{remark}\label{rem:non-degenerate}
If \(r\colon U \times U \to \fg \otimes \fg\) is a non-degenerate meromorphic map, where \(U \subseteq \Bbbk\) is a connected open subset, the map \(\fg^*\to \fg\) defined by \(\alpha \mapsto (1 \otimes \alpha)r(x,y)\) is bijective for all \((x,y)\) in a dense open subset of the domain of definition of \(r\).
\end{remark}

\noindent
Assume that \(\fg\) is semi-simple with Killing form \(\kappa\) and Casimir element \(\gamma \in \fg \otimes \fg\). Then \(\widetilde\kappa(\gamma) = \textnormal{id}_\fg\), where \(\widetilde\kappa \colon \fg\otimes\fg \to \textnormal{End}(\fg)\) is the isomorphism defined by \(a\otimes b \mapsto \kappa(b,\cdot)a\).
Clearly, the linear map \(\fg^*\to \fg\) defined by a tensor \(t \in \fg \otimes \fg\) is bijective if and only if \(\det(\widetilde{\kappa}(t)) \neq 0\). Consider a meromorphic map \(r\) of the form \begin{align}\label{eq:complex_rmatrix_standard_form}
    r(x,y) = \frac{\lambda(y)\gamma}{x-y} + r_0(x,y),
\end{align}
where \(\lambda \colon U \to \Bbbk^\times\) and \(r_0\colon U \times U \to \fg \otimes \fg\) are analytic for some connected open \(U \subseteq \Bbbk\). Then \((x-y)\widetilde{\kappa}(r(x,y))\) equals \(\lambda(y) \textnormal{id}_\fg\) for \(x = y\). Thus, its determinant is non-vanishing in an open neighbourhood of \(\{(x,y) \in U\times U\mid x = y\}\). In particular, \(r\) is seen to be non-degenerate. It turns out that every non-degenerate (generalized) \(r\)-matrix is locally of the form \cref{eq:complex_rmatrix_standard_form} if \(\fg\) is central simple. The proof in the case of a solution of the CYBE is mentioned in \cite{belavin_drinfeld_diffrence_depending} and the GCYBE case is proven similar. However, since the details are omitted in \cite{belavin_drinfeld_diffrence_depending} and the real case is not discussed, we will present a proof using the coordinate-free methods from \cite{kierans_kreussler}.

\begin{proposition}\label{prop:complex_rmatrix_standard_form}
Let \(\fg\) be a finite-dimensional central simple Lie algebra and \(r\colon U \times U \to \fg \otimes\fg\) be a non-degenerate generalized \(r\)-matrix (resp.\ \(r\)-matrix), where \(U \subseteq \Bbbk\) is connected and open. After probably shrinking \(U\) and shifting the origin, \(0 \in U\) and there exists analytic maps \(\lambda\colon U \to \Bbbk^\times\) and \(r_0\colon U \times U \to \fg \otimes \fg\) such that
\begin{align}\label{eq:analytic_standard_form}
    r(x,y) = \frac{\lambda(y)}{x-y}\gamma + r_0(x,y).
\end{align}
In particular, the Taylor series of \(r\) in \(0\) in the second parameter identifies \(r\) with a formal generalized \(r\)-matrix (resp.\ formal \(r\)-matrix).
\end{proposition}
\begin{proof}
\textbf{Step 1.} \emph{Translating the CYBE and GCYBE into an endomorphism language. }
Consider the isomorphisms \(\tilde{\kappa}\colon \fg \otimes \fg \to \End(\fg)\) and \(\Tilde{\kappa}_3\colon\fg \otimes \fg \otimes \fg \to \textnormal{Hom}(\fg\otimes\fg,\fg)\) defined by
\begin{align*}
    &\tilde{\kappa}(a_1\otimes a_2)(b_1) = \kappa(a_2,b_1)a_1&\Tilde{\kappa}_3(a_1\otimes a_2\otimes a_3)(b_1 \otimes b_2) = \kappa(b_2,a_3)\kappa(b_1,a_2)a_1
\end{align*}
for all \(a_1,a_2,a_3,b_1,b_2 \in \fg\).
Assume first that \(r\) is a generalized \(r\)-matrix. Then, applying \(\Tilde{\kappa}_3\) to the GCYBE \cref{eq:GCYBE} and evaluating in an arbitrary \(t_1 \otimes t_2 \in \fg\otimes\fg\) yields
\begin{equation}\label{eq:kappaGCYBE}
    \begin{split}
        [\Tilde{\kappa} (r(x_1,x_2))t_1,\Tilde{\kappa} (r(x_1,x_3))t_2] &= \Tilde{\kappa} (r(x_1,x_2))[t_1,\Tilde{\kappa} (r(x_2,x_3))t_2] \\&+ \Tilde{\kappa} (r(x_1,x_3))[\Tilde{\kappa} (r(x_3,x_2))t_1,t_2],
    \end{split}
\end{equation}
where it was used that e.g.\ for any \(a,b,c,d \in \fg\) we have:
\begin{align*}
    &\Tilde{\kappa}_3([(a\otimes b)^{32},(c\otimes d)^{13}])(t_1 \otimes t_2) = \Tilde{\kappa}_3(c \otimes b \otimes [a,d])(t_1 \otimes t_2) = \kappa([a,d],t_2)\kappa(b,t_1)c \\&= \kappa([\kappa(b,t_1)a,d],t_2)c = -\kappa([\kappa(b,t_1)a,t_2],d)c = -\Tilde{\kappa}(c\otimes d)[\Tilde{\kappa}(a\otimes b)t_1,t_2].
\end{align*}
The other identities used can be derived similarly; see \cite[Proposition 2.14]{kierans_kreussler}.
If \(r\) is a solution of the CYBE \cref{eq:CYBE}, we find
\begin{equation}\label{eq:kappaCYBE}
    \begin{split}
        [\Tilde{\kappa} (r(x_1,x_2))t_1,\Tilde{\kappa} (r(x_1,x_3))t_2] &= \Tilde{\kappa} (r(x_1,x_2))[t_1,\Tilde{\kappa} (r(x_2,x_3))t_2] \\&- \Tilde{\kappa} (r(x_1,x_3))[\Tilde{\kappa} (r(x_2,x_3))^*t_1,t_2],
    \end{split}
\end{equation}
by applying \(\widetilde\kappa_3\). Here \((\cdot)^*\) denotes the adjoint with respect to \(\kappa\).\\
\textbf{Step 2.} \emph{\(r\) has poles along the diagonal. } 
By \cref{rem:non-degenerate}.(1) we can chose a point \((u_0,v_0)\) in the domain of \(r\) such that \(\Tilde{\kappa}(r(u_0,v_0))\) is an isomorphism and \(u \mapsto T_u \coloneqq \Tilde{\kappa}(r(u,v_0))\) is analytic along the line connecting \(u_0\) and \(v_0\) but excluding \(v_0\). We prove by contradiction that \(r\) has a pole at \((v_0,v_0)\). Assume that \(r\) is analytic in \((v_0,v_0)\), i.e.\ \(u \mapsto T_u\) is analytic in \(v_0\). The equations \cref{eq:kappaGCYBE} and \cref{eq:kappaCYBE} reduce to
	\begin{align}\label{eq:nonondegsol1a}
	&[T_ut_1,T_ut_2] = T_u([t_1,T_{v_0}t_2] + [T_{v_0}t_1,t_2])\\\label{eq:nonondegsol1b}
	&[T_ut_1,T_ut_2] = T_u([t_1,T_{v_0}t_2] - [T_{v_0}^*t_1,t_2]).
	\end{align}
respectively by setting \(x_1 = u, x_2,x_3 = v_0\).
Applying \(\psi_u \coloneqq T_u \circ T_{u_0}^{-1}\) to \cref{eq:nonondegsol1a} and \cref{eq:nonondegsol1b} evaluated at \(u = u_0\) results in
	\begin{align}
	&\psi_{u}[T_{u_0}t_1,T_{u_0}t_2] = T_u([t_1,T_{v_0}t_2] + [T_{v_0}t_1,t_2]),\\
	&\psi_u[T_{u_0}t_1,T_{u_0}t_2] = T_u([t_1,T_{v_0}t_2] - [T_{v_0}^*t_1,t_2]).
	\end{align}
Comparing these equations with \cref{eq:nonondegsol1a} and \cref{eq:nonondegsol1b} evaluated at \(u = u_0\) and using the fact that \(T_{u_0}\) is bijective, we see that \(\psi_u\) is a Lie algebra homomorphism in both cases. Therefore, the fact that \(\psi_u\) is orthogonal with respect to \(\kappa\) if it is invertible, implies that \( \textnormal{det}(\psi_u) \in \{0,\pm1\}\); see e.g.\ \cite[Lemma 2.3.]{kierans_kreussler} for details. A continuity argument and \(\psi_{u_0} = \textnormal{id}_\fg\) forces \(\psi_{v_0}\) and consequently \(T_{v_0} = \psi_{v_0}\circ T_{u_0}\) to be an isomorphism. 

Setting \(u = v_0\) in equation \cref{eq:nonondegsol1a}, we see that \(T_{v_0}^{-1}\) is an invertible derivation of \(\fg\), contradicting the simplicity of \(\fg\). Setting \(u = v_0\) in equation \cref{eq:nonondegsol1b} leads to the same contradiction, considering the fact that 
\begin{equation}\label{eq:T_v_0_skew_symmetric}
    \det(T_{v_0}) \neq 0 \implies T_{v_0}^* = -T_{v_0}.    
\end{equation}
The proof of \cref{eq:T_v_0_skew_symmetric} can be found in \cite[Lemma 3.2 and Lemma 3.4]{kierans_kreussler} for \(\Bbbk = \bC\) and uses Schur's Lemma as well as the fact that every automorphism of \(\fg\) has a fixed vector. Since \(\fg\) is assumed to be central, Schur's Lemma applies for \(\Bbbk = \bR\) and an automorphism of \(\fg\) without fixed vector defines one on the simple complex Lie algebra \(\fg\otimes_\bR \bC\) by extension of scalars. Thus, the proof in \cite[Lemma 3.2 and Lemma 3.4]{kierans_kreussler} also applies to the case \(\Bbbk = \bR\).  Summarized, the assumption that \(r\) is a solution of either the CYBE or GCYBE without a pole along the diagonal leads to a contradiction. We have shown that \(r\) has a pole along the diagonal.\\
\textbf{Step 3.} \emph{After shrinking \(U\), \(r(x,y) = \frac{\lambda(y)\gamma}{(x-y)^k} + \frac{f(x,y)}{(x-y)^{k-1}}\). }
Using \cref{lem:series_vanishing_at_diagonal} for \(V = \fg \otimes \fg\) as well as \(\bigcap_{k = 0}^\infty (x-y)^kV[\![x,y]\!] = \{0\}\), we can find a \(k \in \bN_0\) and shrink \(U\) in such a way that \(s(x,y) = (x-y)^kr(x,y)\) is analytic on \(U \times U\) and \(h(z) \coloneqq s(z,z)\) is an analytic function on $U$ which is not identically 0. After probably shrinking \(U\) further, we may assume that \(h\) is non-vanishing and \(s(x,y) - h(y) = (x-y)f(x,y)\) for an analytic function \(f\colon U \times U \to \fg \otimes \fg\). Multiplying either \cref{eq:nonondegsol1a} or \cref{eq:nonondegsol1b} with \((x_1-x_2)^k\) and setting \(x_1= x_2\) results in
    \begin{align}
	    [\Tilde{\kappa}(h(x_2))t_1,\Tilde{\kappa}(r(x_2,x_3))t_2] = \Tilde{\kappa}(h(x_2))[t_1,\Tilde{\kappa}(r(x_2,x_3))t_2]
    \end{align}
    in both cases.
    Choosing \(x_3\) in such a way that \(\Tilde{\kappa}(r(x_2,x_3))\) is an isomorphism, we see that \(\widetilde{\kappa}(h(x_2))\) is an equivariant endomorphism of \(\fg\) with respect to the adjoint representation. In other words, \(\kappa(h(x_2))\) is in the centroid of \(\fg\) (see \cref{def:central_algebra}). Hence, \(\widetilde{\kappa}(h(x_2)) = \lambda(x_2)\textnormal{id}_{\fg}\) since \(\fg\) is central, where \(\lambda \colon U \to\Bbbk^\times\) is an analytic function. This implies that \(h(x_2) = \lambda(x_2)\gamma\). Summarized, we obtain
    \begin{align*}
        r(x,y) = \frac{\lambda(y)\gamma}{(x-y)^k} + \frac{f(x,y)}{(x-y)^{k-1}}.
    \end{align*}
    \textbf{Step 4.} \emph{\(k = 1\). }
    Assume that \(k > 1\). Then 
    \[(x_1-x_2)^{k-1}[r^{32}(x_3,x_2),r^{13}(x_1,x_3)] \textnormal{ and } (x_1-x_2)^{k-1}[r^{13}(x_1,x_3),r^{23}(x_2,x_3)]\]
    vanish for \(x_1 = x_2\). Therefore, multiplying the CYBE or the GCYBE with \((x_1 - x_2)^{k-1}\), using \[[\gamma^{12},r(x_2,y_3)^{23}] = -[\gamma^{12},r(x_2,y_3)^{13}]\] 
    and taking the limit \(x_1 \to x_2\) results in
	\begin{align*}
	    0 &= [h(x_2)^{12},\partial_{x_2}r(x_2,x_3)^{13}] + [f(x_2,x_2)^{12},r(x_2,x_3)^{13} + r(x_2,x_3)^{23}]\\&= [h(x_2)^{12},(x_2 - x_3)^{-k}\partial_{x_2}s(x_2,x_3)^{13} - k(x_2 - x_3)^{-k-1}s(x_2,x_3)^{13}]\\&+ [f(x_2,x_2)^{12},r(x_2,x_3)^{13} + r(x_2,x_3)^{13}],
	\end{align*}
	where \(h,s\) and \(f\) are defined in Step 3 and the limit definition of \(\partial_{x_2}\) was used. 
	Multiplying this with \((x_2 - x_3)^{k+1}\) and taking the limit \(x_1 \to x_3\) under consideration of \(h(z) = \lambda(z)\gamma\) yields \(-k\lambda(x_3)^2[\gamma^{12},\gamma^{13}] = 0\). This contradicts the fact that \([\gamma^{12},\gamma^{13}]\neq 0\). Indeed, \([\gamma^{12},\gamma^{13}]\) maps to \(\gamma\) with respect to the linear map defined by \(a\otimes b \otimes c\mapsto [b,a] \otimes c\); see \cref{eq:fact_about_Casimir_element}. Therefore, the assumption \(k > 1\) leads to a contradiction and we can conclude that \(k = 1\).
\end{proof}

\begin{remark}
A formal generalized \(r\)-matrix associated to a non-degenerate analytic generalized \(r\)-matrix \(r\) is skew-symmetric if and only if  \(r\) is skew-symmetric. In particular, \cref{prop:formal_rmatrices_are_skew-symmetric} gives a new proof of the fact that non-degenerate analytic generalized \(r\)-matrices are exactly the non-degenerate analytic \(r\)-matrices.
\end{remark}

\begin{definition}
A meromorphic map \(s \colon U \times U \to \fg \otimes \fg\) is said to be in \emph{standard form} if \(U\subseteq \Bbbk\) is a connected open neighbourhood of \(0 \in \Bbbk\) such that 
\begin{equation}
    s(x,y) = \frac{\lambda(y) \gamma}{x-y} + s_0(x,y)
\end{equation}
for some analytic maps \(\lambda \colon U \to \Bbbk\) and \(s_0 \colon U \times U \to \fg \otimes \fg\)
\end{definition}

\begin{remark}
\Cref{prop:complex_rmatrix_standard_form} allows us to restrict our study of non-degenerate generalized \(r\)-matrices to those in standard form.
\end{remark}

\begin{definition}\label{def:analytic_equivalence}
Let \(r\colon U \times U \to \fg \otimes \fg\), \(\widetilde{r}\colon V \times V \to \fg \otimes \fg\) be meromorphic maps in standard form. Then \(\widetilde{r}\) is called \emph{analytically equivalent} to \(r\) if
\[\widetilde{r}(x,y) = \lambda(y)(\varphi(x) \otimes \varphi(y))r(w(x),w(y)),\]
where the triple \((\lambda,w,\varphi)\) is called an \emph{analytic equivalence} and consists of
\begin{itemize}
    \item an analytic embedding \(w \colon W \to U\), for some connected open neighbourhood \(W \subseteq V\) of \(0\) satisfying \(w(0) = 0\), called \emph{coordinate transformation},
    \item a non-zero analytic \(\lambda \colon W \to \Bbbk\) called \emph{rescaling} and
    \item an analytic \(\varphi \colon W \to \Aut_{\Bbbk\textnormal{-alg}}(\fg)\) called \emph{gauge transformation}.
\end{itemize} 
\end{definition}

\begin{remark}
A similar result to \cref{lemm:formal_equivalence} holds for analytic equivalences. More precisely, analytic equivalences preserve non-degeneracy and the property of solving the GCYBE \cref{eq:GCYBE}. Furthermore, analytic equivalences with constant rescaling part preserve skew-symmetry and the property of solving the CYBE \cref{eq:CYBE}. 
The proof of these statements uses a reduction to the complex case and the identity theorem (see e.g.\ \cite[Chapter I.A, Theorem 6]{gunning_rossi_1965}), under consideration that the domain of definition of a complex meromorphic function on a connected open set is connected.
\end{remark}

\begin{definition}\label{def:Theta}
For a meromorphic map \(s \colon U \times U \to \fg \otimes \fg\) such that \(z \mapsto s(z,0)\) is meromorphic, let \(\Theta s \in (\fg \otimes \fg)(\!(x)\!)[\![y]\!]\) denote the series
\begin{equation}
    \Theta s \coloneqq \sum_{k = 0}^\infty \frac{1}{k!} \frac{\partial s}{\partial y^k}(x,0)y^k,
\end{equation}
i.e.\ \(\Theta s\) is the Taylor expansion of \(s\) in the second variable in 0.
\end{definition}

\noindent
Clearly, \(\Theta r\) is a formal generalized \(r\)-matrix for every solution of the GCYBE \(r\) in standard form. The following result shows that \(\Theta\) essentially identifies the formal and analytic setting in this way.

\begin{proposition}\label{prop:analytic_and_formal_theories_are_equivalent}
Every formal generalized \(r\)-matrix over \(\bC\) (resp.\ \(\bR\)) is of the form \(\Theta r\) for a complex (resp.\ real) analytic generalized \(r\)-matrix in standard form. Furthermore, two complex (resp.\ real) analytic generalized \(r\)-matrices \(r_1\) and \(r_2\) in standard form are analytically equivalent if and only if \(\Theta r_1\) and \(\Theta r_2\) are equivalent.
In other words, \(\Theta\) defines a bijection between equivalence classes of formal generalized \(r\)-matrices over \(\bC\) (resp.\ \(\bR\)) and analytic equivalence classes of complex (resp.\ real) analytic generalized \(r\)-matrices in standard form. 
\end{proposition}
\begin{proof}
\textbf{Step 1.} \emph{Setup. } Let \(i\in\{1,2\}\) and \(r_i(x,y) \in (\fg \otimes \fg)(\!(x)\!)[\![y]\!]\) be a formal generalized \(r\)-matrix and \(\textnormal{GD}(\textnormal{Mult}(\fg(r_i)),\fg(r_i)) \coloneqq ((X_i,\cA_i),(p_i,c_i,\zeta_i))\) be the associated geometric datum. Pick a smooth open neighbourhood \(C_i\) of \(p_i\) such that \(\cA_i|_{C_i}\) is \'etale \(\fg\)-locally free and there exists a non-vanishing 1-form \(\eta_i\) on \(C_i\). The geometric \(r\)-matrix defined by \(((X_i,\cA_i),(C_i,\eta_i))\) will be denoted by \(\rho_i\).\\ 
\textbf{Step 2.} \emph{\(r_i\) is equivalent to a Taylor series of an analytic generalized \(r\)-matrix \(\varrho_i\). }
Let \(C_i^\textnormal{an} = (C_i^\textnormal{an},\cO_{C_i}^\textnormal{an})\) be the analytic manifold defined by the \(\Bbbk\)-rational points of \({C_i}\). Note that \(C_i\) is 1-dimensional, where in the real case this may be seen through the implicit function theorem and the fact that \(p_i\in C_i^\textnormal{an}\). Let \({U_i} \to C_i^\textnormal{an}\) be an analytic parameterization around \(p_i\), where \({U_i} = (U_i,\cO_{U_i}^\textnormal{an})\) is the locally ringed space associated to an open disc if \(\Bbbk = \bC\) (resp. an open interval if \(\Bbbk = \bR\)) around the origin. We write \(\iota_i \colon {U_i} \to C_i^\textnormal{an} \to X_i\) for the resulting morphism of locally ringed spaces such that \(\iota_i(0) = p_i\). The sheaf of Lie algebras
\(\iota_i^*\cA_i\) can be identified with an analytic fiber bundle on \({U_i}\) with fiber \(\fg\) and structure group \(\textnormal{Aut}_{\Bbbk\textnormal{-alg}}(\fg)\). Indeed, for \(\Bbbk = \bC\) this follows from \cref{thm:etale_trivial_sheaves_of_Lie_algebras} and the observation that \'etale \(\fg\)-local triviality implies local triviality in the complex topology, while for \(\Bbbk = \bR\) this is due to \cite[Lemma 2.1]{kirangi_semi_simple}. These fiber bundles are always trivial since \(U_i\) is contractible; see \cite[Satz 6]{grauert} for the complex and  \cite[Chapter VIII,Propositions 1.10 \& 1.19]{guaraldo_macri_Tancredi} for the real case. Thus, there exists an isomorphism \(\psi_i\colon \iota_i^*\cA_i \to \fg \otimes \cO_{U_i}^\textnormal{an}\) of sheaves of Lie algebras and \(\psi_i \boxtimes \psi_i\) defines an isomorphism \(\iota_i^*\cA_i \boxtimes \iota_i^*\cA_i \to (\fg \otimes \fg)\otimes \cO_{U_i\times U_i}^\textnormal{an}\). Consider the meromorphic map
\[\varrho_i \coloneqq (\psi_i \boxtimes \psi_i)(\iota_i \times \iota_i)^*\rho_i\colon U_i \times U_i \to \fg \otimes \fg.\]
The Taylor series of \(\varrho_i\) in the second variable in the preimage of \(p_i\) under \(\iota_i\) is equivalent to \(r_i\). This can be deduced with an argument similar to the proof of \cref{thm:extension_of_generalized_rmatrices}.\\
\textbf{Step 3.} \emph{An equivalence between \(r_1\) and \(r_2\) defines an analytical equivalence between \(\varrho_1\) and \(\varrho_2\). } The equivalence between \(r_1\) and \(r_2\) defines isomorphisms \(f\colon X_2 \to X_1\) and \(\phi \colon \cA_1 \to f_*\cA_2\) such that, after probably adjusting \(C_1,C_2,\eta_1\) and \(\eta_2\), we have \(f(p_2) = p_1\), \(f^{-1}(C_1) = C_2\), \(f^*\eta_1 = \eta_2\) and
\begin{equation}
    (f^*(\phi) \boxtimes f^*(\phi))(f\times f)^*\rho_1 = \rho_2;
\end{equation}
see \cref{lem:equivalence_of_geometric_rmatrices}. Application of \((\psi_2 \boxtimes \psi_2)(\iota_2\times \iota_2)^*\) results in 
\begin{equation}\label{eq:analytic_equivalence_from_geometry}
    \big(\psi_2\big((f\iota_2)^*\phi\big) \boxtimes \psi_2\big((f\iota_2)^*\phi\big)\big)(f\iota_2\times f\iota_2)^*\rho_1 = \varrho_2.
\end{equation}
After probably shrinking \(U_2\), there exists an analytic embedding \(w\colon U_2 \to U_1\) such that \(\iota_1w = f\iota_2\), since \(f\) is an isomorphism and maps \(p_2\) to \(p_1\). In particular, \(w(0) = 0\). We can rewrite \cref{eq:analytic_equivalence_from_geometry} as \((\varphi(x) \otimes \varphi(y))\varrho_1(w(x),w(y)) = \varrho_2(x,y)\), where \(\varphi \colon U_2 \to \textnormal{Aut}_{\Bbbk\textnormal{-alg}}(\fg)\) is the analytic map induced by the chain
\begin{equation}
    \xymatrix{\fg \otimes \cO_{U_2} \ar[rr]^-{w^*\big(\psi_1^{-1}\big)}&& w^*\iota_1^*\cA_1 = (f\iota_2)^*\cA_1 \ar[rr]^-{(f\iota_2)^*\phi}&& \iota_2^*\cA_2 \ar[r]^-{\psi_2}& \fg \otimes \cO_{U_2}}
\end{equation}
of isomorphisms of sheaves of Lie algebras.
\end{proof}

\begin{remark}
Let us point out that for a meromorphic map \(s\) in standard form, \(\Theta s\) is skew-symmetric if and only if \(s\) is. Therefore, \cref{prop:analytic_and_formal_theories_are_equivalent} induces a bijection between equivalence classes of formal \(r\)-matrices and analytic \(r\)-matrices as well.
\end{remark}

\subsection{Selected classification results for sheaves of Lie algebras}\label{subsec:classification_results_for_etale_locally_free_sheaves_of_Lie_algebras} In this section we present some classification results for sheaves of Lie algebras with simple fibers over an algebraically closed field. These will yield a new proof of the Belavin-Drinfeld trichotomy in the next section by refining \cref{thm:geometry_of_formal_rmatrices} using the local \'etale triviality and the standard form of the geometric \(r\)-matrix in \cref{lem:geomstdform}.

Isomorphism classes of \'etale locally trivial sheaves of algebras can be understood using the notion of torsors. Let us briefly recall said notion in the limited scope needed in the following, we refer to e.g.\  \cite[Section III.4]{milne_etale_cohomology} for more details. Let \(X\) be a \(\Bbbk\)-scheme of finite type and \(G\) be a group scheme over \(X\). Note that \'etale coverings of \(X\) are without loss of generality surjective \'etale morphisms \(Y \to X\); see \cref{rem:etale_morphisms_wlog_surjective}. For a surjective \'etale morphism \(Y \to X\), the set \(Z^1(Y/X,G)\) of \emph{1-cocycles} trivialized on \(Y\)  with values in \(G\) consists of morphisms \(g\colon Y\times_X Y \to G\) of \(X\)-schemes satisfying \(\textnormal{pr}_{31}^*(g) = \textnormal{pr}_{32}^*(g)\textnormal{pr}_{21}^*(g)\), where
\begin{align}
    &\textnormal{pr}_{ij}\colon Y\times_X Y \times_X Y \longrightarrow Y \times_X Y&(y_1,y_2,y_3) \longmapsto (y_i,y_j)
\end{align}
are the canonical projections for \(ij \in \{21,31,32\}\). Two 1-cocycles \(g,g' \in Z^1(Y/X,G)\)  are called \emph{cohomologous}, written in symbols as \(g\sim g'\), if there exists a morphism \(h \colon Y \to G\)
of \(X\)-schemes such that \(g' = \textnormal{pr}_2^*(h)g\textnormal{pr}_1^*(h)^{-1}\), where  \(\textnormal{pr}_1,\textnormal{pr}_2\colon Y\times_X Y \to Y\) are the canonical projections. We write \(\check{\textnormal{H}}{}^1(Y/X,G) \coloneqq Z^1(Y/X,G)/\sim\). If we have another surjective \'etale map \(Y' \to X\) factoring over \(Y\to X\), there exists a natural induced map \(\check{\textnormal{H}}(Y/X,G) \to \check{\textnormal{H}}(Y'/X,G)\). The set of \'etale \(G\)-torsors is given by \(\check{\textnormal{H}}{}^1(X_{\textnormal{\'et}},G) \coloneqq \varinjlim \check{\textnormal{H}}{}^1(Y/X,G)\), where the limit is taken over the directed set of surjective \'etale morphisms \(Y\to X\).

\begin{lemma}\label{lem:sheaves_of_Lie_algebras_and_torsors}
Let \(\Bbbk = \overline{\Bbbk}\), \(X\) be a reduced \(\Bbbk\)-scheme of finite type and \(A\) be a finite-dimensional \(\Bbbk\)-algebra. The isomorphism classes of \'etale \(A\)-locally free sheaves of algebras are in bijection with the set \( \check{\textnormal{H}}{}^1(X_{\textnormal{\'et}},\textnormal{Aut}_{\Bbbk\textnormal{-alg}}(A)_X)\), where
\(\textnormal{Aut}_{\Bbbk\textnormal{-alg}}(A)_X \coloneqq X\times G\) for the unique group scheme \(G\) over \(\textnormal{Spec}(\Bbbk)\) with closed points \(\textnormal{Aut}_{\Bbbk\textnormal{-alg}}(A)\).
\end{lemma}
\begin{proof}
Let \(E\) be the set of isomorphism classes of \'etale \(A\)-locally free sheaves of algebras.\\
\textbf{Step 1.} \emph{Construction of \(E \to \check{\textnormal{H}}{}^1(X_{\textnormal{\'et}},\textnormal{Aut}_{\Bbbk\textnormal{-alg}}(A)_X)\). }
Let \(\cA\) be an \'etale \(A\)-locally free sheaf over \(X\). Under consideration of \cref{rem:etale_morphisms_wlog_surjective}, there exists a surjective \'etale morphism \(f \colon Y \to X\) of \(\Bbbk\)-schemes such that we may choose an isomorphism \(f^*\cA \cong A \otimes \cO_{Y}\) of sheaves of algebras. Let \(\textnormal{pr}_1,\textnormal{pr}_2\colon Y\times_X Y \to Y\) be the canonical projections. The chain of isomorphisms 
\[A \otimes \cO_{Y\times_X Y} \cong \textnormal{pr}_1^*f^*\cA \cong \textnormal{pr}_2^*f^* \cA \cong A \otimes \cO_{Y\times_X Y}\]
determines a regular map from the algebraic prevariety of closed points of  \(Y\times_X Y\) to \(\textnormal{Aut}_{\Bbbk\textnormal{-alg}}(A)\). This induces an unique morphism \(Y\times_X Y \to G\) of \(\Bbbk\)-schemes, which in turn defines a morphism \(g \colon Y \times_X Y \to X \times G\) of \(X\)-schemes. It is straight forward to see that \(\textnormal{pr}_{31}^*(g) = \textnormal{pr}_{32}^*(g)\textnormal{pr}_{21}^*(g)\) for the canonical projections \(\textnormal{pr}_{ij}\colon Y\times_X Y \times_X Y \to Y \times_X Y\), where \(ij \in \{21,31,32\}\), and that a sheaf of algebras isomorphic to \(\cA\) defines a 1-cocycle cohomologous to \(g\). Thus, we have constructed a map \(E \to \check{\textnormal{H}}{}^1(X_{\textnormal{\'et}},\textnormal{Aut}_{\Bbbk\textnormal{-alg}}(A)_X)\).\\
\textbf{Step 2.} \emph{Construction of the inverse of \(E \to \check{\textnormal{H}}{}^1(X_{\textnormal{\'et}},\textnormal{Aut}_{\Bbbk\textnormal{-alg}}(A)_X)\). }
Let \(f \colon Y \to X\) be a surjective \'etale morphism and \(g \colon Y\times_X Y \to X \times G\) be a 1-cocycle. Then \(g\) defines an isomorphism \(\psi\colon A \otimes \cO_{Y\times_X Y} \to A\otimes \cO_{Y\times_XY}\) of sheaves of algebras such that 
\(\textnormal{pr}_{31}^*(\psi) = \textnormal{pr}_{32}^*(\psi)\textnormal{pr}_{21}^*(\psi)\) for the canonical projections \(\textnormal{pr}_{ij}\colon Y\times_X Y \times_X Y \to Y \times_X Y\), where \(ij \in \{21,31,32\}\). Therefore, we obtain a quasi-coherent sheaf \(\cA\) on \(X\) equipped with an isomorphism \(f^*\cA \cong A \otimes \cO_Y\) using faithfully flat descent; see e.g.\ \cite[Proposition 2.22]{milne_etale_cohomology}. If \(X\) and \(Y\) are affine, \(\cA\) is simply the direct image via \(Y \to X\) of the kernel of \(\psi \textnormal{pr}^*_1 - \textnormal{pr}^*_2 \colon A \otimes \cO_Y \to A \otimes \cO_{Y\times Y}\). This can be used to see that \(\cA\) is a subsheaf of algebras of \(A\otimes \cO_Y\), since \(\psi\) is an automorphism of sheaves of algebras, and that a 1-cocycle cohomologous to \(g\) defines a sheaf of algebras isomorphic to \(\cA\). The resulting map \(\check{\textnormal{H}}{}^1(X_{\textnormal{\'et}},\textnormal{Aut}_{\Bbbk\textnormal{-alg}}(A)_X)\to E\) is clearly inverse to \(E \to \check{\textnormal{H}}{}^1(X_{\textnormal{\'et}},\textnormal{Aut}_{\Bbbk\textnormal{-alg}}(A)_X)\).\\
\end{proof}

\begin{theorem}\label{thm:weakly_g_locally_free_Lie_algebra}
Let \(\fg\) be a finite-dimensional simple Lie algebra over \(\Bbbk\),  \(\Bbbk = \overline{\Bbbk}\) and \(\cA\) be a weakly \(\fg\)-locally free sheaf of Lie algebras on a \(\Bbbk\)-scheme \(X\) of finite type.
\begin{enumerate}
\item If \(X = \textnormal{Spec}(\Bbbk[u,u^{-1}])\), there exists \(\sigma \in \Aut_{\Bbbk\textnormal{-alg}}(\fg)\) of order \(m\in\bN\) and a primitive \(m\)-th root of unity \(\varepsilon \in \Bbbk\) such that \(\cA\) is isomorphic to the sheaf of Lie algebras associated to the twisted loop algebra 
\[\mathfrak{L}(\fg,\sigma) \coloneqq \{f(\tilde{u}) \in \fg[\tilde{u},\tilde{u}^{-1}]\mid f(\varepsilon\tilde{u}) = \sigma(f(\tilde{u}))\}\] 
on \(X\), where the module structure of \(\mathfrak{L}(\fg,\sigma)\) is defined by \(u = \tilde{u}^m\).
\item If \(X = \textnormal{Spec}(\Bbbk[u])\), \(\cA\) is isomorphic to the sheaf of Lie algebras associated to \( \fg[u]\) on \(X\).
\end{enumerate}
\end{theorem}
\begin{proof}
\textbf{Step 1.} \emph{Setup. } 
In both cases \cref{thm:etale_trivial_sheaves_of_Lie_algebras} states that \(\cA\) is automatically \'etale \(\fg\)-locally free on \(X\) and for this reason, up to isomorphism, determined by an element of \(\check{\textnormal{H}}^1(X_{\textnormal{\'et}},\textnormal{Aut}_{\Bbbk\textnormal{-alg}}(\fg)_X)\) by virtue of \cref{lem:sheaves_of_Lie_algebras_and_torsors}.
The arguments in \cite{pianzola} imply that, since the Lie group of \(\fg\) is reductive, there is a canonical injection \(\check{\textnormal{H}}^1(X_{\textnormal{\'et}},\textnormal{Aut}_{\Bbbk\textnormal{-alg}}(\fg)_X) \to \check{\textnormal{H}}
^1(X_{\textnormal{\'et}},\textnormal{Out}(\fg)_X)\), where \(\textnormal{Out}(\fg)\) is the automorphism group of the Dynkin diagram of \(\fg\) and \(\textnormal{Out}(\fg)_X \coloneqq X\times \textnormal{Out}(\fg)\). The case of \(X = \textnormal{Spec}(\Bbbk[u,u^{-1}])\) is thereby considered explicitly in \cite{pianzola} while the case of \(X = \textnormal{Spec}(\Bbbk[u])\) works analogous. Since \(\textnormal{Out}(\fg)\) is finite, we have a bijection of  \(\check{\textnormal{H}}
^1(X_{\textnormal{\'et}},\textnormal{Out}(\fg)_X)\) and the non-abelian continuous cohomology group \(\textnormal{H}^1(\pi_1(X,x),\textnormal{Out}(\fg))\), where \(x \in X\) is a closed point, \(\pi_1(X,x)\) is the associated \'etale fundamental group and the action of \(\pi_1(X,x)\) on \(\textnormal{Out}(\fg)\) is trivial.\\
\textbf{Step 2.} \emph{Proof of (1). } If \(X = \textnormal{Spec}(\Bbbk[u,u^{-1}])\), we can choose \(x = (u-1)\) and then it is explained in \cite{pianzola} that \(\textnormal{H}^1(\pi_1(X,x),\textnormal{Out}(\fg))\) is in bijection with the conjugacy classes in \(\textnormal{Out}(\fg)\) and the \(\textnormal{Aut}_{\Bbbk\textnormal{-alg}}(\fg)_X\)-torsor of \(\mathfrak{L}(\fg,\sigma)\) is mapped to the conjugacy class of the class of \(\sigma^{-1}\) in \(\textnormal{Out}(\fg)\). In particular, every \(\textnormal{Aut}_{\Bbbk\textnormal{-alg}}(\fg)_X\)-torsor is represented by some \(\mathfrak{L}(\fg,\sigma)\) for an appropriate \(\sigma\).\\
\textbf{Step 3.} \emph{Proof of (2). } If \(X = \textnormal{Spec}(\Bbbk[u])\) and \(x = (u)\), the group \(\pi_1(X,x)\) is trivial. Therefore, \(\check{\textnormal{H}}^1(X_{\textnormal{\'et}},\textnormal{Aut}_{\Bbbk\textnormal{-alg}}(\fg)_X)\) has only one element, consisting of the trivial \(\textnormal{Aut}_{\Bbbk\textnormal{-alg}}(\fg)_X\)-torsor on \(X\), i.e.\ the one represented by \(\fg[u]\).
\end{proof}

Let us now consider an elliptic curve \(X\) over \(\Bbbk = \bC\) and write \(X^\textnormal{an} = (X^\textnormal{an},\cO^\textnormal{an}_X)\) for the locally ringed space obtained by considering this curve as a complex manifold. Then \(X^\textnormal{an}\) is isomorphic to \(\bC/\Lambda\) for some lattice \(\Lambda = \bZ\lambda_1 + \bZ\lambda_2\subset \bC\) of rank two. Let \(\iota \colon X^\textnormal{an} \to X\) denote the morphism of locally ringed spaces which identifies points of \(X^\textnormal{an}\) with closed points of \(X\) on the level of topological spaces, while \(\iota^\flat \colon \cO_X \to \iota_* \cO^\textnormal{an}_X\) recognizes regular functions as holomorphic ones. Then the results of Serre's GAGA \cite{GAGA}, tell us that \(\mathcal{F} \mapsto \mathcal{F}^\textnormal{an} \coloneqq \iota^*\mathcal{F}\) defines an equivalence between the category of locally free sheaves on \(X\) and the category of sheaves of sections of holomorphic vector bundles over \(X^\textnormal{an}\). Furthermore, the field of global meromorphic functions on \(X^\textnormal{an}\) coincides with the field of rational functions of \(X\); see e.g.\ \cite[Section 3.1]{griffith_harris}. Consequently, the pull-back \(\iota^*\) identifies \(\Gamma(U,\mathcal{F})\), for \(U\subseteq X\) open, with the global meromorphic sections of \(\mathcal{F}^\textnormal{an}\) that are holomorphic on \(\iota^{-1}(U)\). 

Let \(\pi \colon \bC \to \bC/\Lambda \cong X^\textnormal{an}\) be the canonical map, \(\cA\) be a locally free sheaf of rank \(d\) on \(X\) and \(E \to X^\textnormal{an}\) be the vector bundle with sheaf of holomorphic sections \(\cA^\textnormal{an}\).
In e.g.\ \cite{iena} it is explained that, since \(\pi^*E\) is trivial, \(E\) is determined by some holomorphic map
\(\phi \colon \Lambda \times \bC \to \textnormal{GL}(d,\bC)\) satisfying 
\begin{align}\label{eq:factorofautomorphy}
    &\Phi(\lambda + \lambda',z) = \Phi(\lambda,z+\lambda')\Phi(\lambda',z)&\lambda,\lambda'\in\Lambda, z\in\bC,
\end{align}
called \emph{factor of automorphy}, in the sense that 
\begin{align}
    &E = \bC \times \bC^d/\sim & (z,a) \sim (z + \lambda,\Phi(\lambda,z)a), \quad \forall \lambda \in\Lambda.
\end{align}
Assume that \(\cA\) is an \'etale \(\fg\)-locally free sheaf of Lie algebras for some finite-dimensional complex Lie algebra \(\fg\). Then it is easy to see that \(E\) is a holomorphic fiber bundle with fiber \(\fg\) and structure group \(\textnormal{Aut}_{\bC\textnormal{-alg}}(\fg)\). Therefore, \cite[Satz 6]{grauert} implies that \(\pi^*E \cong \bC \times \fg\) as holomorphic fiber bundles. This implies that \(\Phi\) takes values in  \(\textnormal{Aut}_{\bC\textnormal{-alg}}(\fg)\).

\begin{theorem}\label{thm:acyclic_sheaves_of_Lie_algebras_on_elliptic_curve}
Let \(\cA\) be an acyclic (i.e.\ \(\textnormal{h}^1(\cA) = 0\)) sheaf of Lie algebras on \(X\), which is  weakly \(\fg\)-locally free in all closed points of \(X\) for some simple, finite-dimensional, complex Lie algebra \(\fg\). There exists \(n \in \mathbb{N}_0\) and \(0 < m < n\) such that \(\textnormal{gcd}(n,m)= 1\), \(\fg \cong \mathfrak{sl}(n,\bC)\) and \(\cA^\textnormal{an}\) is isomorphic to the sheaf of holomorphic sections of \(\bC \times \mathfrak{sl}(n,\bC)/\sim\), where
for \(\varepsilon \coloneqq \exp(2\pi i m/n)\) and
\begin{align*}
    &T_1 \coloneqq \begin{pmatrix}1 & 0 &\dots & 0\\
    0 & \varepsilon & \dots & 0\\
    \vdots & \vdots & \ddots & \vdots\\
    0 & 0 & \dots & \varepsilon^{n-1}
    \end{pmatrix},&T_2 \coloneqq \begin{pmatrix}0 & \dots & 0 & 1\\
    1 & \dots & 0 & 0\\
    \vdots & \ddots & \vdots & \vdots\\
    0 & \dots & 1 & 0
    \end{pmatrix}.
\end{align*}
the relation \(\sim\) is defined by 
\(
    (z,a)\sim (z+\lambda_1,T_1aT_1^{-1})\sim(z+\lambda_2,T_2aT_2^{-1}).
\) 
\end{theorem}
\begin{proof}
\textbf{Step 1.} \emph{\(\textnormal{h}^0(\cA) = 0 = \textnormal{h}^1(\cA)\). }
The morphism \(\cA \to \cA^*\) induced by the Killing form of \(\cA\) is an isomorphism, since  \(\cA|_p \cong \fg\) is simple for all \(p \in X\) closed; see \cref{lem:fiber_of_sheaf_Killing_form}. Therefore, using Serre duality, we see that \(\textnormal{h}^0(\cA) = \textnormal{h}^1(\cA^*) = \textnormal{h}^1(\cA) = 0\).\\ 
\textbf{Step 2.} \emph{Description of \(\Gamma(X\setminus\{p\},\cA)\), where \(p\) is the image of \(\overline{0}\) under \(\bC/\Lambda \to X^\textnormal{an}\to X\). } 
\Cref{thm:etale_trivial_sheaves_of_Lie_algebras} states that \(\cA\) is \'etale \(\fg\)-locally trivial. As argued above, this implies that the holomorphic vector bundle \(E \to X^\textnormal{an}\) with holomorphic sheaf of sections \(\cA^\textnormal{an}\) is determined by a factor of automorphy \(\Phi \colon \bC \times \Lambda \to \textnormal{Aut}_{\bC\textnormal{-alg}}(\fg)\). In particular,  
\(\Gamma(X\setminus\{p\},\cA)\) can be identified with the algebra of meromorphic functions \(a \colon \bC \to \fg\) holomorphic on \(\bC \setminus \Lambda\) and satisfying \(a(z+\lambda) = \Phi(\lambda,z)a(z)\) for all \(\lambda \in \Lambda,z \in \bC \setminus\Lambda\). \\
\textbf{Step 3.} \emph{\(\Phi\) is locally constant up to isomorphism. } The Taylor series of elements of \(\cA^\textnormal{an}\) with respect to a chosen holomorphic coordinate \(z\) on \(\bC\) combined with the canonical isomorphism \(\compA_p \cong \compA^\textnormal{an}_p\) results in an isomorphism \(\zeta \colon \compA_p \to \fg[\![z]\!]\). Using Step 1 and \cref{eq:adelicsequence} yields \(\fg(\!(z)\!) = \fg[\![z]\!] \oplus \zeta(\Gamma(X\setminus\{p\},\cA))\), where we note that \(\zeta\) identifies sections of \(\cA\), viewed as meromorphic functions \(\bC \to \fg\) with their Laurent series in \(z = 0\). Therefore, \cref{prop:Lie_subalgebra_of_generalized_rmatrix} implies that there exists a normalized formal generalized \(r\)-matrix \(r\) such that \(\fg(r) = \zeta(\Gamma(X\setminus\{p\},\cA))\). We can chose a global 1-form \(\eta\) on \(X\) such that \(\pi^*\iota^*\eta = \textnormal{d}z\) as a holomorphic 1-form on \(\bC\). The residue theorem implies that
\[\textnormal{res}_0 \kappa(\zeta(a),\zeta(b))\textnormal{d}z = \textnormal{res}_p K(a,b) \eta = 0\]
for all \(a,b\in \Gamma(X\setminus\{p\},\cA)\), where \(K\) is the Killing form of \(\cA\). Therefore,
\(\fg(r)^\bot = \fg(r)\) and \cref{prop:formal_rmatrices_are_skew-symmetric}  forces \(r\) to be skew-symmetric. Combining \cref{lem:geometry_of_equivalences} and \cref{prop: formal_rmatrix_depends_on_difference}, we may assume that \(\fg(r)\), and therefore \(\zeta(\Gamma(X\setminus \{p\},\cA))\), is closed under the derivation with respect to \(z\), after probably replacing \(\cA\) with an isomorphic sheaf of Lie algebras. In particular, we have
\begin{align}
    \Phi(\lambda,z)\frac{da}{dz}(z) = \frac{da}{dz}(z+\lambda) =  \frac{\partial \Phi}{\partial z}(\lambda,z)a(z) + \Phi(\lambda,z)\frac{da}{dz}(z).
\end{align}
for every \(a \in \Gamma(X\setminus \{p\},\cA)\).
Therefore, \(\frac{\partial\Phi}{\partial z}(\lambda,z)a(z) = 0\) for all \(z \in \bC \setminus\Lambda\), \(a \in \Gamma(X\setminus \{p\},\cA)\). Since \(\zeta(\Gamma(X\setminus \{p\},\cA)) \otimes \bC(\!(z)\!) = \fg(\!(z)\!)\), we see that \( \frac{\partial\Phi}{\partial z}(\lambda,z) = 0\), and thus \(\Phi_\lambda \coloneqq \Phi(\lambda,z) \in \textnormal{Aut}_{\bC\textnormal{-alg}}(\fg)\) is independent of \(z\). \\
\textbf{Step 4.} \emph{\( P \coloneqq \{\Phi_\lambda\}_{\lambda \in \Lambda}\) is a finite abelian group. } Equation \cref{eq:factorofautomorphy} gives \(\Phi_\lambda \Phi_{\lambda'} = \Phi_{\lambda + \lambda'}\) for all \(\lambda,\lambda'\in\Lambda\), so \(\Phi\) is a commutative subgroup of \(\textnormal{Aut}_{\bC\textnormal{-alg}}(\fg)\) generated by \(\Phi_1 \coloneqq \Phi_{\lambda_1}, \Phi_2 \coloneqq \Phi_{\lambda_2}\). A non-zero element in \(\fg\) which is fixed by all elements in \(P\) would define a global section of \(\cA\). Hence such an element does not exist by Step 1. Assume that \(P\) has infinite order and let \(\mathfrak{s}\) be the Lie algebra of the smallest algebraic subgroup \(S\) of  \(\textnormal{Aut}_{\bC\textnormal{-alg}}(\fg)\) containing \(P\). Since \(P\) is infinite \(\mathfrak{s}\) can be identified with a non-zero subalgebra of \(\fg\) and since \(P\) is abelian and dense (with respect to the Zariski topology) in \(S\), it can be shown that \(S\) is abelian. Therefore, the action of \(S\) on \(\mathfrak{s}\) is trivial and each non-zero element of \(\mathfrak{s}\) is fixed by all elements in \(P\). This is a contradiction. We can conclude that \(P\) has finite order.\\
\textbf{Step 5.} \emph{Concluding the proof. }
The commuting automorphisms \(\Phi_1\) and \(\Phi_2\) have finite order and no non-zero common fixed vector. Thus,  \cite[Theorem 9.3]{belavin_dirnfeld_triangle} implies that \(\fg\) is isomorphic to \(\mathfrak{sl}(n,\bC)\) via an isomorphism which identifies \(\Phi_1\) and \(\Phi_2\) with the conjugations by \(T_1\) and \(T_2\) respectively, for an appropriate choice of \(m\).
\end{proof}

\begin{remark}
For our purposes, it is actually sufficient that \(\cA^\textnormal{an}\) is isomorphic to the sheaf of sections of \(\bC \times \fg/\sim\), where 
\begin{align*}
    (z,a)\sim (z+\lambda_1,\Phi_1 a)\sim(z+\lambda_2,\Phi_2 a),
\end{align*}
for some \(\Phi_1,\Phi_2 \in \textnormal{Aut}_{\bC\textnormal{-alg}}(\fg)\) of finite order. Therefore, we could adjust the statement of \cref{thm:acyclic_sheaves_of_Lie_algebras_on_elliptic_curve} accordingly in order to drop the reference to \cite{belavin_dirnfeld_triangle} in Step 5 of its proof and stay self-contained. Nevertheless, we chose this presentation to show a conclusive result to the given classification problem. 
\end{remark}

\subsection{The Belavin-Drinfeld trichotomy of complex \(r\)-matrices} We are now able to proof the following version of the Belavin-Drinfeld trichotomy \cite[Theorem 1.1]{belavin_drinfeld_solutions_of_CYBE_paper} in two spectral parameters. The proof presented in \cref{sec:proof_of_BD} below completely relies on the algebro-geometric methods established in this paper and is hence independent of \cite{belavin_drinfeld_solutions_of_CYBE_paper}. 

\begin{theorem}\label{thm:Belavin-Drinfeld_trichotomy}
Let \(\fg\) be a finite-dimensional simple complex Lie algebra, \(r
\in (\fg \otimes \fg)(\!(x)\!)[\![y]\!]\) be a normalized formal \(r\)-matrix and \(X\) be the irreducible cubic plane curve associated to \(r\) in \cref{thm:geometry_of_formal_rmatrices} (see also \cref{rem:curves_of_arithmetic_genus_one}). The following results are true.
\begin{enumerate}
    \item \(X\) is elliptic if and only if \(r\) is gauge equivalent to \(\Theta \varrho\) for a meromorphic map \(\varrho \colon \bC \times \bC \to \fg \otimes \fg\) satisfying   \(\varrho(x+\lambda,y+\lambda') = \varrho(x,y)\) for some \(n\in\bN\) and all \(\lambda,\lambda'\in n\Lambda\). Here, \(\Theta\) was defined in \cref{def:Theta}.
    \item \(X\) is nodal if and only if \(r\) is gauge equivalent to \(\Theta \varrho\), where
    \begin{align}\label{eq:sigma_trigonometric_form}
         \varrho(x,y) = \frac{1}{\exp(x-y)-1}\sum_{k =  0}^{m-1}\exp\left(\frac{k(x-y)}{m}\right)\gamma_k + t\left(\exp\left(\frac{x}{m}\right),\exp\left(\frac{y}{m}\right)\right)
    \end{align}
    for some \(\sigma \in\Aut_{\bC\textnormal{-alg}}(\fg)\) of order \(m\) and \(t \in \mathfrak{L}(\fg,\sigma)\otimes \mathfrak{L}(\fg,\sigma)\). Here, \(\mathfrak{L}(\fg,\sigma)\) is defined in \cref{thm:weakly_g_locally_free_Lie_algebra} and \(\gamma = \sum_{k = 0}^{m-1}\gamma_k\) is the unique decomposition such that \((\sigma \otimes 1)\gamma_k = \exp(2\pi ik/m)\gamma_k\).
    \item \(X\) is cuspidal if and only if \(r\) is gauge equivalent to \(\Theta \varrho\), where \(\varrho(x,y) = \frac{\gamma}{x-y} + t(x,y)\)
    for some \(t \in (\fg \otimes \fg)[x,y]\).
\end{enumerate}
\end{theorem}

\begin{remark}
The sheaves of Lie algebras coming from complex \(r\)-matrices have been studied in a case by case fashion using the Belavin-Drinfeld trichotomy \cite[Theorem 1.1]{belavin_drinfeld_solutions_of_CYBE_paper}. In the elliptic case it is shown in \cite{burban_heinrich} that these sheaves, which are described in \cref{thm:acyclic_sheaves_of_Lie_algebras_on_elliptic_curve}, are all of the form \(\textnormal{Ker}(\textnormal{Tr}_\cF\colon \sheafEnd_{\cO_X}(\cF)\to \cO_X)\) for some simple vector bundle \(\cF\) on a complex elliptic curve \(X\). In the rational case the corresponding sheaves on the cuspidal curve were constructed in \cite{burban_galinat}, using the structure theory of rational \(r\)-matrices from \cite{stolin_sln,stolin_maximal_orders}.

The sheaves of Lie algebras on the nodal curve corresponding to trigonometric \(r\)-matrices were recently constructed independently in \cite{polishchuck_trigonometric_geometrisation} and \cite{abedin_burban}, completing the geometrization of non-degenerate solutions of the CYBE \cref{eq:intro_one_parameter_CYBE}. In \cite{polishchuck_trigonometric_geometrisation}, the author calculates the set of multipliers of the subalgebras associated to trigonometric \(r\)-matrices using the classification of said \(r\)-matrices from \cite{belavin_drinfeld_solutions_of_CYBE_paper} and then applies the geometrization procedure presented in \cref{subsec:geometry_of_lattices}. The construction in \cite{abedin_burban} uses a different approach based on twisting the standard Lie bialgebra structure of loop algebras and a result in \cite{kac_wang}, which can be seen as an analog of the theory of maximal orders from \cite{stolin_sln,stolin_maximal_orders} for subalgebras of Kac-Moody algebras.  

The derivation of \cref{thm:Belavin-Drinfeld_trichotomy} simultaneously yields a new proof of the Belavin-Drinfeld trichotomy and an alternative way to execute the geometrization program for non-degenerate solutions of the CYBE \cref{eq:intro_one_parameter_CYBE}.
\end{remark}

\subsection{Proof of \cref{thm:Belavin-Drinfeld_trichotomy}.}\label{sec:proof_of_BD}
Let \(((X,\cA),(p,c,\zeta))\) as well as \(\eta \in \textnormal{H}^0(\omega_X)\) be the geometric datum associated to \(r\) in \cref{thm:geometry_of_formal_rmatrices}.
Furthermore, let \(X^\textnormal{an}=(X^\textnormal{an},\cO_X^\textnormal{an})\) be the complex analytic space associated to \(X\), \(\iota \colon X^\textnormal{an} \to X\) be the canonical morphism of locally ringed spaces and write \(\breve{X} \subseteq X\), \(\iota^{-1}(\breve{X}) = \breve{X}^\textnormal{an}\) for the respective smooth loci.
If \(X = \breve{X}\) is smooth, it is a complex elliptic curve, so there exists a lattice \(\Lambda \subseteq \bC\) of rank two as well as a biholomorphic map \(\widetilde\nu \colon \bC/\Lambda \to X^\textnormal{an}\) such that \(\widetilde{\nu}(\Lambda) = p\).  Otherwise, \(X\) has a unique singular closed point \(s\) and the normalization \(\nu \colon \mathbb{P}^1_\bC \to X\), where one of the following cases occurs:
\begin{itemize}
    \item \(s\) is nodal and we can choose coordinates \((u_0:u_1)\) on \(\mathbb{P}^1_\bC\) such that \(\nu^{-1}(s) = \{(1:0),(0:1)\}\) and \(\nu(1:1) = p\).  In particular, \(\nu\) restricts to an isomorphism \(\textnormal{Spec}(\bC[u,u^{-1}]) \to \breve{X}\) for \(u = u_1/u_0\) and induces a biholomorphic map \(\widetilde\nu \colon\bC^\times \to \breve{X}^\textnormal{an}\).
    \item \(s\) is a cuspidal and we can choose coordinates \((u_0:u_1)\) on \(\mathbb{P}^1_\bC\) such that \(\nu^{-1}(s) = \{(0:1)\}\) and \(\nu(1:0) = p\). In particular, \(\nu\) restricts to an isomorphism \(\textnormal{Spec}(\bC[u]) \to \breve{X}\) for \(u = u_1/u_0\) and induces a biholomorphic map \(\widetilde\nu \colon\bC  \to \breve{X}^\textnormal{an}\).
\end{itemize}
Summarized, we have a holomorphic covering \(\widetilde\pi \colon \bC = (\bC,\cO_\bC^\textnormal{an}) \to \breve{X}^\textnormal{an}\) satisfying \(\iota\widetilde{\pi}(0) = p\), defined by
\begin{align}\label{eq:holomorphic_covering}
    \widetilde\pi(\tilde{z}) = \begin{cases} \widetilde\nu(\tilde{z} + \Lambda)& \textnormal{if }X \textnormal{ is elliptic}\\
    \widetilde\nu\left(\exp\left({\tilde{z}}\right)\right)&\textnormal{if } X \textnormal{ is nodal}\\
    \widetilde\nu(\tilde{z})&\textnormal{if } X \textnormal{ is cuspidal}
    \end{cases},
\end{align}
in some holomorphic coordinates \(\tilde{z}\) on \(\bC\). The invertible sheaf \(\Omega_{\breve{X}}^\textnormal{an} = \iota^*\Omega_{\breve{X}}\) can be identified with the sheaf of holomorphic 1-forms on \(\breve{X}^\textnormal{an}\), so \(\iota^*\eta\) can be viewed as holomorphic 1-form. Let us write \(\pi \coloneqq \iota \widetilde{\pi}\). We can assume that \(\pi^*\eta = \textnormal{d}\tilde{z}\). Indeed, it is well-known that there exists a \(\lambda \in \Bbbk^\times\) such that \(\pi^*\eta = \lambda\textnormal{d}\tilde{z}\) if \(X\) is elliptic and
\begin{equation}
    \nu^*(\eta) = \begin{cases} \lambda{\textnormal{d}u}/{u}& \textnormal{if } X \textnormal{ is nodal}\\
    \lambda\textnormal{d}u&\textnormal{if } X \textnormal{ is cuspidal}
    \end{cases}.
\end{equation}
We can achieve that \(\lambda = 1\) by replacing \(r\) with \(\lambda r(\lambda x, \lambda y)\).

\begin{lemma}\label{lem:multipliers_are_rational}
Let \(\theta\colon \widehat{\cO}^\textnormal{an}_{\bC,0} \to \bC[\![\tilde{z}]\!]\) be the isomorphism defined by the Taylor series in 0. Then the diagram
\begin{equation}\label{eq:z_is_holomorphic_coordinate}
    \xymatrix{\compO_{X,p}\ar[r]^{\widehat{\pi}^\sharp_0}\ar[d]_c & \widehat{\cO}^\textnormal{an}_{\bC,0}\ar[d]^{\theta}\\\bC[\![z]\!]\ar[r]_{z\mapsto \tilde{z}}&\bC[\![\tilde{z}]\!]}
\end{equation}
commutes. Furthermore, any series in \(\bC(\!(z)\!)\) representing a rational function on \(X\) coincides with the Laurent series of a  meromorphic function on \(\bC\) in 0. This meromorphic function is elliptic (resp.\ a rational function of exponentials, resp.\ rational) if and only if \(X\) is elliptic (resp.\ nodal, resp.\ cuspidal).
\end{lemma}
\begin{proof}
Similar to \cref{rem:c*_construction_for_1_forms}, the canonical derivations \(\cO_{X,p} \to \omega_{X,p}\) and \(\cO^\textnormal{an}_{\bC,0} \to \Omega_{\bC,0}^\textnormal{an}\) induce continuous derivations  \(\compO_{X,p} \to \widehat\omega_{X,p}\) and \(\compO^\textnormal{an}_{\bC,0} \to \widehat\Omega_{\bC,0}^\textnormal{an}\) whose images generate the respective modules. These derivations will both be denoted by \(\textnormal{d}\), since it will be clear from the context which one is in use. The completion \(\widehat\omega_{X,p} \to \widehat\Omega_{\bC,0}^\textnormal{an}\) of 
\begin{equation}
    \omega_{X,p}\stackrel{\pi^*_p}\longrightarrow (\pi_*\pi^*\Omega_{\bC}^\textnormal{an})_p \longrightarrow \Omega_{\bC,0}^\textnormal{an}
\end{equation}
is described by \(\textnormal{d}f \mapsto \textnormal{d}\widehat{\pi}^\sharp_0(f)\) for all \(f \in \compO_{X,p}\). The identity \(c^*(\widehat{\eta}_p) = \textnormal{d}z\) implies that \(\widehat{\eta}_p = \textnormal{d}c^{-1}(z)\) (see \cref{rem:c*_construction_for_1_forms}) and \(\pi^*\eta = \textnormal{d}\tilde{z}\) implies that \(\textnormal{d}\widehat{\pi}^\sharp_0(c^{-1}(z)) = \textnormal{d}\tilde{z}\). This yields \(\widehat{\pi}^\sharp_0(c^{-1}(z)) = \tilde{z}\), i.e.\ \cref{eq:z_is_holomorphic_coordinate} is commutative, since \(\theta(\tilde{z}) = \tilde{z}\). 

Let \(f\) be a rational function of \(X\). Then \(\pi^\flat(f)(\tilde{z}) = f(\pi(\widetilde{z}))\) is a meromorphic function on \(\bC\) and its Laurent series in \(0\) coincides with its image of the extension of \(\theta\) to the respective quotient fields. The commutativity of \cref{eq:z_is_holomorphic_coordinate} implies that this Laurent series evaluated in \(z\) coincides with \(c(f) \in \bC(\!(z)\!)\). Looking at \cref{eq:holomorphic_covering}, we can see that \(f(\pi(\widetilde{z}))\) is elliptic if and only if \(X\) is elliptic, a rational function of exponentials if and only if \(X\) is nodal and a rational function if and only if \(X\) is cuspidal. Here we used \(f\iota \widetilde{\nu} = f \nu \iota\) and the fact that \(f\nu\) is a rational function on \(\mathbb{P}^1_\bC\) if \(X\) is singular, i.e.\ simply a quotient of two polynomials.
\end{proof}

\begin{lemma}
The ``if'' holds in all three parts of \cref{thm:Belavin-Drinfeld_trichotomy}.
\end{lemma}
\begin{proof}
Assume that \(\varrho\) is of the given form and let \(\widetilde{r} \coloneqq \Theta \varrho \in (\fg\otimes \fg)(\!(x)\!)[\![y]\!]\). By \cref{lem:generators_of_g(r)}, \(\fg(\widetilde{r})\) is generated by \(\{(1\otimes\alpha)\varrho(z,0)\mid \alpha \in \fg^*\}\), so it can be identified with a subalgebra of meromorphic maps \(\bC \to \fg\)
which are elliptic in case \emph{(1)}, rational functions of \(\textnormal{exp}(z/m)\) in case \emph{(2)} and rational in case \emph{(3)}. Since \(r\) and \(\widetilde{r}\) are gauge equivalent, it holds that \[\textnormal{Mult}(\fg(r)) = \textnormal{Mult}(\fg(\widetilde{r})) = \{f\in \fg(\!(z)\!)\mid f \fg(\widetilde{r}) \subseteq \fg(\widetilde{r})\}.\]
We have seen in Step 3 of the proof of \cref{thm:multipliers_are_klattice} that there exists a subalgebra \(O \subseteq \textnormal{Mult}(\fg(\widetilde{r}))\) of finite codimension with the property: for every \(f\in O\) exists a non-commutative polynomial \(P = P(x_1,\dots,x_q)\) and elements \(a_1,\dots,a_q \in \fg(\widetilde{r})\) satisfying \(f \textnormal{id}_\fg = P(\textnormal{ad}(a_1),\dots,\textnormal{ad}(a_q))\). In particular, \(O\) consists of elliptic functions in case \emph{(1)}, rational functions of exponentials in case \emph{(2)} and rational functions in case \emph{(3)}. Since the quotient field of \(O\) coincides with the rational functions on \(X\), this observation combined with \cref{lem:multipliers_are_rational} proves all "if" directions.
\end{proof}

\begin{lemma}
Let \(\rho\) be the geometric \(r\)-matrix of \(((X,\cA),(C,\eta))\).
There exists an isomorphism \(\psi \colon \pi^*\cA \to \fg \otimes \cO_\bC\) such that the analytic \(r\)-matrix \(\varrho \colon \bC \times \bC \to \fg \otimes \fg\) defined by 
\begin{equation}
    \varrho = (\psi \boxtimes \psi)(\pi \times \pi)^*\rho|_{C\times C}    
\end{equation}
has the following property depending on \(X\):
\begin{enumerate}
    \item If \(X\) is elliptic, \(\varrho(x+\lambda,y+\lambda') = \varrho(x,y)\) for some \(n\in\bN\) and all \(\lambda,\lambda'\in n\Lambda\).
    \item If \(X\) is nodal, \(\varrho\) is of the form \cref{eq:sigma_trigonometric_form}
    for some \(t \in \mathfrak{L}(\fg,\sigma)\otimes \mathfrak{L}(\fg,\sigma)\), where \(\sigma\) is some automorphism of \(\fg\) of order \(m\) and \(\gamma = \sum_{k = 0}^{m-1}\gamma_k\) is the unique decomposition such that \((\sigma \otimes 1)\gamma_k = \exp(2\pi ik/m)\gamma_k\).
    \item If \(X\) is cuspidal, \(\varrho(x,y) = \frac{\gamma}{x-y} + t(x,y)\)
    for some \(t \in (\fg \otimes \fg)[x,y]\).
\end{enumerate}
\end{lemma}
\begin{proof}
\emph{Construction of \(\varrho\) in (1): } Since \(\cA\) is weakly \(\fg\)-locally free, \Cref{thm:acyclic_sheaves_of_Lie_algebras_on_elliptic_curve} provides an isomorphism  \(\psi_1 \colon \widetilde{\nu}^*\iota^*\cA\to\mathcal{S}\), where \(\mathcal{S}\) is the sheaf on \(\bC/\Lambda\) of holomorphic sections of
\begin{align*}
    &\bC \times \mathfrak{sl}(n,\bC)/\sim&(z,a)\sim (z+\lambda_1,T_1aT_1^{-1})\sim(z+\lambda_2,T_2aT_2^{-1}),
\end{align*}
for \(T_1\) and \(T_2\) as defined in \cref{thm:acyclic_sheaves_of_Lie_algebras_on_elliptic_curve}.
Let \(\psi_2 \colon \textnormal{pr}^*\mathcal{S} \to \fg \otimes \cO_\bC^\textnormal{an}\) denote the canonical isomorphism, where \(\textnormal{pr}\colon\bC \to \bC/\Lambda\) is the canonical projection. Then \(\psi_2\textnormal{pr}^*a \colon \bC \to \fg\) is an \(n\Lambda\)-periodic meromorphic function for any rational section \(a\) of \(\mathcal{S}\), since \(T_1^n = T_2^n = \textnormal{id}_\fg\).
Therefore,
\[{\varrho}\coloneqq (\psi \boxtimes \psi)(\pi\times \pi)^*\rho \colon \bC\times \bC \to \fg \otimes \fg\]
is a meromorphic map satisfying the desired periodicity, where \(\psi \coloneqq \psi_2 (\textnormal{pr}^*\psi_1)\) and \(\pi = \iota\widetilde{\nu}\textnormal{pr}\) was used.

\emph{Construction of \(\varrho\) in (2): } Since \(\cA|_{\breve{X}}\) is weakly \(\fg\)-locally free, \Cref{thm:weakly_g_locally_free_Lie_algebra}.\emph{(1)} provides an isomorphism \({\psi}_1 \colon \nu^*(\cA|_{\breve{X}}) \to \mathcal{L}\), where \(\mathcal{L}\) is the sheaf associated to \(\mathfrak{L}(\fg,\sigma) \subseteq \fg[\tilde{u},\tilde{u}^{-1}]\) on \(\textnormal{Spec}(\bC[u,u^{-1}]) = \nu^{-1}(\breve{X})\), for some \(\sigma \in \textnormal{Aut}_{\bC\textnormal{-alg}}(\fg)\) of finite order \(m\). Using \cref{lem:geomstdform} for \(U = \breve{X}\), the local parameter \(\nu^{\flat,-1}(u-1)\) of \(p\), and \(\chi = ({\psi}_1\boxtimes {\psi}_1)^{-1}(\nu^{-1} \times \nu^{-1})^*\widetilde{\chi}\), where 
\begin{equation}
    \widetilde{\chi} \coloneqq \sum_{j= 0}^{m-1}(\tilde{u}/\tilde{v})^j\gamma_j\in \mathfrak{L}(\fg,\sigma)\otimes \mathfrak{L}(\fg,\sigma) \subseteq  \fg[\tilde{u},\tilde{u}^{-1}]\otimes  \fg[\tilde{u},\tilde{u}^{-1}] \cong (\fg \otimes \fg)[\tilde{u},\tilde{u}^{-1},\tilde{v},\tilde{v}^{-1}],
\end{equation}
results in 
\begin{equation}
    \widetilde{\rho}\coloneqq ({\psi}_1 \boxtimes {\psi}_1)(\nu \times \nu)^*\rho|_{\breve{X}\times \breve{X}} = \frac{v\widetilde{\chi}}{u-v} + t,
\end{equation}
for some \(t \in \mathfrak{L}(\fg,\sigma)\otimes \mathfrak{L}(\fg,\sigma)\). Here we used \(\tilde{u}^m = u\) and \(\eta = \textnormal{d}u/u\). Moreover, it was used that \((\sigma \otimes \sigma)\gamma = \gamma\) gives \(\gamma = \sum_{j = 0}^{m-1}\gamma_j\) and this implies that \(\chi\) is indeed a preimage of \(\textnormal{id}_{\cA|_U}\) under \cref{eq:preimage_of_identity}. The mapping \(a(\tilde{u}) \mapsto a(\textnormal{exp}(z/m)))\) induces an isomorphism \({\psi}_2 \colon \textnormal{exp}^*\iota^*\mathcal{L} \to \fg\otimes\cO_\bC^\textnormal{an}\) such that
\begin{equation}
    \varrho \coloneqq ({\psi}_2 \boxtimes {\psi}_2)(\iota\exp \times \iota \exp)^*\widetilde\rho|_{\breve{X}\times \breve{X}} = (\psi \boxtimes \psi)(\pi \times \pi)^*\rho|_{\breve{X}\times \breve{X}}\colon \bC \times \bC \to \fg \otimes \fg
\end{equation}
is of the form \cref{eq:sigma_trigonometric_form}, where \(\psi \coloneqq \psi_2\left( \exp^*\iota^*\psi_1\right)\) and \(\nu\iota\textnormal{exp} =\iota  \widetilde{\nu}\textnormal{exp} = \pi\) was used.

\emph{Construction of \(\varrho\) in (3): } Since \(\cA|_{\breve{X}}\) is weakly \(\fg\)-locally free, \Cref{thm:weakly_g_locally_free_Lie_algebra}.\emph{(2)} provides an isomorphism \({\psi} \colon \nu^*(\cA|_{\breve{X}}) \to \fg \otimes \cO_{\textnormal{Spec}(\bC[u])}\). Using \cref{lem:geomstdform} for \(U = \breve{X}\), the local parameter \(\nu^{\flat,-1}(u)\) of \(p\), and \(\chi = ({\psi}\boxtimes {\psi})^{-1}(\nu^{-1} \times \nu^{-1})^*\gamma\), where \(\gamma \in \fg \otimes \fg\) is considered as a global section \((\fg \otimes \cO_{\textnormal{Spec}(\bC[u])}) \boxtimes (\fg \otimes \cO_{\textnormal{Spec}(\bC[u])}) \cong (\fg \otimes \fg) \otimes \cO_{\textnormal{Spec}(\bC[u,v])}\), results in the desired
\begin{equation}
    \varrho \coloneqq ({\psi} \boxtimes {\psi})(\nu \times \nu)^*\rho|_{\breve{X}\times \breve{X}} = \frac{\gamma}{u-v} + t,
\end{equation}
for some \(t \in (\fg \otimes \fg)[u,v]\). Here we used \(\eta = \textnormal{d}u\) and the fact that \(\chi\) is obviously a preimage of \(\textnormal{id}_{\cA|_U}\) under \cref{eq:preimage_of_identity}. 
\end{proof}

\begin{lemma}
The formal \(r\)-matrix \(\Theta \varrho\) is gauge equivalent to \(r\), where \(\Theta\) was defined in \cref{def:Theta}.
\end{lemma}
\begin{proof}
In all three cases, \(\varrho = (\psi \times \psi)(\pi \times \pi)^*\rho|_{\breve{X}\times \breve{X}}\) for an isomorphism \(\psi \colon \pi^*\cA \to \fg \otimes \cO_\bC\). Using \cref{lem:multipliers_are_rational}, we can see that the composition
\begin{equation}
    \fg[\![z]\!] \stackrel{\zeta^{-1}} \longrightarrow \compA_p  \longrightarrow \widehat{\pi^*\cA}_0 \stackrel{\widehat{\psi}_0}\longrightarrow \fg \otimes \compO^\textnormal{an}_{\bC,0} \stackrel{\textnormal{id}_\fg \otimes \theta}\longrightarrow \fg[\![z]\!]
\end{equation}
defines a \(\bC[\![z]\!]\)-linear Lie algebra automorphism \(\varphi\) of \(\fg[\![z]\!]\). Here, the second arrow is the isomorphism obtained by completing \(\cA_p \to (\pi_*\pi^*\cA)_p \to \cA_p \otimes_{\cO_{X,p}} \cO^\textnormal{an}_{\bC,0} \cong (\pi^*\cA)_0\) and \(\theta\) is the map defined by the Taylor expansion in 0. It is straight forward to show that the diagram
\begin{equation}
    \xymatrix{\Gamma(\breve{X}\times \breve{X}\setminus\Delta,\cA \boxtimes \cA) \ar[r]^{\jmath^*}\ar[d]_{(\pi\times \pi)^*} & (\fg \otimes \fg)(\!(x)\!)[\![y]\!] \ar[r]^-{\varphi(x) \otimes \varphi(y)}& (\fg \otimes \fg)(\!(x)\!)[\![y]\!]\\
    \Gamma(\bC \times \bC\setminus\Delta_\pi,\pi^*\cA \boxtimes \pi^*\cA) \ar[r]_-{\psi \boxtimes \psi}& (\fg \otimes \fg ) \otimes \Gamma(\bC \times \bC\setminus\Delta_\pi,\cO_{\bC\times\bC}^\textnormal{an})\ar[ur]_-\Theta}
\end{equation}
commutes, where \(\Delta_\pi \coloneqq \{(x,y)\in \bC\times \bC\mid \pi(x) = \pi(y)\}\) and \(\jmath\) was defined in \cref{not:mathfrak_X}. The upper row maps \(\rho|_{\breve{X}\times \breve{X}}\) to \((\varphi(x) \otimes \varphi(y))r(x,y)\) by virtue of \cref{thm:geometric_rmatrix_Taylor_series}, so we can conclude that \((\varphi(x) \otimes \varphi(y))r(x,y) = \Theta \varrho(x,y)\).
\end{proof}

\printbibliography
\end{document}